\documentclass[a4paper, reqno, 14pt]{amsart}
\usepackage[utf8]{inputenc}

\usepackage{amsthm,amsfonts,amssymb,amsmath}
\usepackage{xr-hyper}
\usepackage[colorlinks=true]{hyperref}
\usepackage{verbatim}
\usepackage{pdfsync}

\usepackage{latexsym}
\usepackage{amssymb}
\usepackage{amsmath}
\usepackage{amsthm}
\usepackage{verbatim}
\usepackage{pdfsync}
\usepackage[curve,matrix,arrow,frame,tips]{xy}

\begin{document}








\newcommand{\A}{{\mathbb A}}
\newcommand{\C}{{\mathbb C}}
\newcommand{\F}{{\mathbb F}}
\newcommand{\G}{{\mathbb G}}
\newcommand{\R}{{\mathbb R}}
\newcommand{\Q}{{\mathbb Q}}
\newcommand{\X}{{\mathbb X}}
\newcommand{\Z}{{\mathbb Z}}
\newcommand{\HZ}{\widehat{\Z}}


\newcommand{\rom}[1]{\text{\rm #1}}
\renewcommand{\roman}{\rm}

\newcommand{\Aut}{\text{\rm Aut}}
\newcommand{\CH}{\widehat{\text{\rm CH}}}
\newcommand{\cha}{{\text{\rm char}}}
\newcommand{\CHe}{\text{\rm CHeeg}}
\newcommand{\degh}{\widehat{\text{\rm deg}}}
\newcommand{\degH}{\widehat{\text{\rm deg}}}    
\newcommand{\diag}{{\text{\rm diag}}}
\newcommand{\Diff}{\text{\rm Diff}}
\newcommand{\disc}{\text{\rm discr}}
\renewcommand{\div}{\text{\rm div}}
\newcommand{\divh}{\widehat{\text{\rm div}}}
\newcommand{\DS}{\text{\rm DS}}
\newcommand{\Ei}{\text{\rm Ei}}
\newcommand{\End}{\text{\rm End}}
\newcommand{\ev}{{\text{\rm ev}}}
\newcommand{\Gal}{\text{\rm Gal}}
\newcommand{\GL}{\text{\rm GL}}
\newcommand{\GSpin}{\text{\rm GSpin}}
\newcommand{\Hom}{\text{\rm Hom}}
\newcommand{\hor}{{\text{\rm horiz}}}
\newcommand{\id}{\text{\rm id}}
\newcommand{\im}{\text{\rm im}}
\renewcommand{\Im}{\text{\rm Im}}
\newcommand{\inv}{{\text{\rm inv}}}
\newcommand{\Jac}{\text{\rm Jac}}
\newcommand{\Leray}{{\mathrm L}}
\newcommand{\Lie}{\text{\rm Lie}}
\newcommand{\Mp}{\text{\rm Mp}}
\newcommand{\mult}{\text{\rm mult}}
\newcommand{\MW}{\text{\rm MW}}
\newcommand{\MWt}{\widetilde{\MW}}
\newcommand{\new}{\text{\rm new}}
\newcommand{\Nm}{\text{\rm Nm}}
\newcommand{\ord}{\text{\rm ord}}
\newcommand{\PGL}{\text{\rm PGL}}
\newcommand{\Pic}{\text{\rm Pic}}
\newcommand{\Pich}{\widehat{\text{\rm Pic}}}
\newcommand{\pr}{\text{\rm pr}}
\newcommand{\ra}{\text{\rm ra}}
\newcommand{\Rao}{\mathrm R}
\renewcommand{\Re}{\text{\rm Re}}
\newcommand{\sgn}{\text{\rm sgn}}
\newcommand{\sig}{\text{\rm sig}}
\newcommand{\SL}{\text{\rm SL}}
\newcommand{\SO}{\text{\rm SO}}
\newcommand{\Sp}{\text{\rm Sp}}
\newcommand{\Spec}{\text{\rm Spec}\, }
\newcommand{\Spf}{\text{\rm Spf}}
\newcommand{\supp}{\text{\rm supp}}
\newcommand{\Sym}{{\text{\rm Sym}}}
\newcommand{\tr}{\text{\rm tr}}
\newcommand{\type}{\text{\rm type}}
\newcommand{\Ver}{\text{\rm Vert}}
\newcommand{\vol}{\text{\rm vol}}
\newcommand{\Wald}{\text{\rm Wald}}


\newcommand{\Cal}{\mathcal}     

\newcommand{\AHH}{\hat{\Cal A}}   
\newcommand{\CHH}{\hat{\Cal C}}
\newcommand{\MM}{\Cal D}          
\newcommand{\MMb}{\MM^\bullet}
\newcommand{\ssplit}{\text{\bf split}}
\newcommand{\whcc}{\widehat{\Cal C}}
\newcommand{\CO}{\mathcal O}
\newcommand{\COH}{\widehat{\CO}}
\newcommand{\M}{\Cal M}
\newcommand{\OB}{\Cal O_B}
\newcommand{\XX}{\mathcal X}
\newcommand{\bXX}{\bar\XX}
\newcommand{\wc}{\hat{\Cal C}}
\newcommand{\wch}{\wc^{\text{\rm hor}}}
\newcommand{\ZZ}{\Cal Z}
\newcommand{\ZH}{\widehat{\Cal Z}}   
\newcommand{\Zh}{\widehat{\Cal Z}}
\newcommand{\ZZh}{\ZZ^{\text{\rm hor}}}
\newcommand{\ZZv}{\ZZ^{\text{\rm ver}}}
\newcommand{\ZZhh}{\Zh^{\text{\rm hor}}}
\newcommand{\ZZhv}{\Zh^{\text{\rm ver}}}


\newcommand{\nass}{\noalign{\smallskip}}
\newcommand{\snass}{\noalign{\vskip 2pt}}
\newcommand{\tent}[1]{ \vphantom{\vbox to #1pt{}} }   


\newcommand{\scr}{\scriptstyle}
\newcommand{\disp}{\displaystyle}

\font\cute=cmitt10 at 12pt
\font\smallcute=cmitt10 at 9pt
\newcommand{\kay}{{\text{\cute k}}}
\newcommand{\smallkay}{{\text{\smallcute k}}}

\renewcommand{\a}{\alpha}
\renewcommand{\b}{\beta}
\newcommand{\e}{\epsilon}
\renewcommand{\l}{\lambda}
\renewcommand{\L}{\Lambda}
\renewcommand{\o}{\omega}
\renewcommand{\O}{\Omega}
\renewcommand{\P}{\Phi}
\newcommand{\ph}{\varphi}
\newcommand{\phih}{\widehat{\phi}}
\newcommand{\wphi}{\widehat{\phi}}
\newcommand{\phit}{\widetilde{\phi}}
\newcommand{\s}{\sigma}
\newcommand{\vth}{\vartheta}


%

\newcommand{\Pt}{P}
\newcommand{\Ph}{\P}
\newcommand{\Pht}{\tilde \P}   
\newcommand{\Kt}{K}           
\newcommand{\Mt}{M}

\newcommand{\pht}{\widetilde{\phi}}
\newcommand{\It}{I}
\newcommand{\Jt}{\widetilde{J}}
\newcommand{\lt}{\widetilde{\l}}
\newcommand{\vp}{\varpi}

\newcommand{\bom}{{\boldsymbol{\o}}}
\newcommand{\hbom}{\widehat{\bom}}
\newcommand{\ff}{{\bold f}}
\newcommand{\fsp}{\boldsymbol{f}_{\!\rm sp}}
\newcommand{\fev}{\boldsymbol{f}_{\!\rm ev}}
\newcommand{\fb}{\boldsymbol{f}}
\newcommand{\J}{\und{J}'}
\newcommand{\JJ}{\bold J'}
\newcommand{\V}{\bold V}
\newcommand{\xx}{\bold x}

\newcommand{\g}{{\mathfrak g}}
\renewcommand{\H}{\mathfrak H}


\newcommand{\back}{\backslash}
\newcommand{\CT}[1]{\operatornamewithlimits{CT}_{#1}}
\renewcommand{\d}{\partial}
\newcommand{\db}{\bar\partial}
\newcommand{\dbar}{\bar{\partial}}
\newcommand{\gs}[2]{\langle \,#1,#2\,\rangle}
\newcommand{\Gt}{G}
\newcommand{\hfal}{h_{\text{\rm Fal}}}
\newcommand{\II}{\int^\bullet}
\newcommand{\isoarrow}{\ {\overset{\sim}{\longrightarrow}}\ }
\newcommand{\lisoarrow}{\ {\overset{\sim}{\longleftarrow}}\ }
\newcommand{\limdir}[1]{\underset{\underset{#1}{\rightarrow}}{\lim}}
\newcommand{\lan}{\operatorname{\langle}\hskip .5pt}
\newcommand{\ran}{\,\operatorname{\rangle}}
\newcommand{\lra}{\longrightarrow}
\newcommand{\doublelra}{\ {\overset{\scr\lra}{\scr\lra}}\ }
\newcommand{\nat}{\natural}
\newcommand{\notmid}{\mkern-5mu\not\mkern5mu\mid}
\newcommand{\Optoc}{\text{\rm Opt}(O_{c^2d},O_B)}
\newcommand{\psim}{\psi^{-}}
\newcommand{\qeq}{\ \overset{??}{=}\ }
\newcommand{\sh}{\sharp}
\newcommand{\thCH}{\theta^{\text{\rm ar}}}
\newcommand{\wht}{\widehat{\theta}}     
\newcommand{\triv}{1\!\!1}
\renewcommand{\tt}{\otimes}
\newcommand{\und}[1]{\underline{#1}}
\newcommand{\z}{z}  

\newcommand{\thMW}{\theta^{\text{\rm ar}}}
\newcommand{\tph}{\widetilde{\widehat\phi_1}}
\newcommand{\Pet}{\text{\rm Pet}}





\newcommand{\thing}{ \raisebox{-6.4pt}{$\tilde{\tilde{}}$}  }   
\newcommand{\smallthing}{ \raisebox{-4.4pt}{$\scr\tilde{\tilde{}}$}  }
\newcommand{\ttilde}[1]{\overset{\smash{\thing}}{#1}}
\newcommand{\smallttilde}[1]{\overset{\smash{\smallthing}}{#1}}
\newcommand{\downhookarrow}{\hbox{$\downarrow\hskip -6.1pt\raisebox{6pt}{$\cap$}$}}


\providecommand{\bysame}{\makebox[3em]{\hrulefill}\thinspace}   
\newcommand{\hfb}{\hfill\break}
\newcommand{\margincom}[1]{\marginpar{\bf\raggedright #1}}
\newcommand{\Sec}{\S}


\numberwithin{equation}{section}
\setcounter{section}{0}
\setcounter{MaxMatrixCols}{15}


\newtheorem{theo}{Theorem}[section]
\newtheorem{lem}[theo]{Lemma}
\newtheorem{prop}[theo]{Proposition}
\newtheorem{cor}[theo]{Corollary}
\newtheorem{conj}[theo]{Conjecture}
\newtheorem{rem}[theo]{Remark}      
\newtheorem{defn}[theo]{Definition}



\newcommand{\E}{\mathbb E}

\newcommand{\OO}{\text{\rm O}}
\newcommand{\UU}{\text{\rm U}}

\newcommand{\OK}{O_{\smallkay}}
\newcommand{\DI}{\mathcal D^{-1}}

\newcommand{\pre}{\text{\rm pre}}

\newcommand{\Bor}{\text{\rm Bor}}
\newcommand{\Rel}{\text{\rm Rel}}
\newcommand{\rel}{\text{\rm rel}}
\newcommand{\Res}{\text{\rm Res}}
\newcommand{\TG}{\widetilde{G}}

\parindent=0pt
\parskip=6pt
\baselineskip=14pt

\newcommand{\PP}{\mathcal P}
\renewcommand{\OO}{\mathcal O}
\newcommand{\BB}{\mathbb B}
\newcommand{\GU}{\text{\rm GU}}
\newcommand{\Herm}{\text{\rm Herm}}

\newcommand{\FF}{\mathbb F}
\renewcommand{\MM}{\mathbb M}
\newcommand{\YY}{\mathbb Y}
\newcommand{\VV}{\mathbb V}
\newcommand{\LL}{\mathbb L}

\newcommand{\subover}[1]{\overset{#1}{\subset}}
\newcommand{\supover}[1]{\overset{#1}{\supset}}

\newcommand{\hgs}[2]{\{#1,#2\}}

\newcommand{\wh}[1]{\widehat{#1}}

\newcommand{\Ker}{\text{\rm Ker}}
\newcommand{\YYbar}{\overline{\YY}}
\newcommand{\MMbar}{\overline{\MM}}
\newcommand{\oY}{\overline{Y}}
\newcommand{\dra}{\dashrightarrow}
\newcommand{\dlra}{\longdashrightarrow}

\newcommand{\yy}{\text{\bf y}}

\newcommand{\red}{\text{\rm red}}

\newcommand{\inc}{\text{\rm inc}}

\newcommand{\OKs}{O_{\smallkay,s}}
\newcommand{\OKr}{O_{\smallkay,r}}
\newcommand{\Xs}{X^{(s)}}
\newcommand{\Xo}{X^{(0)}}

\newcommand{\xss}{\underset{\sim}{x}}
\newcommand{\xxss}{\underset{\sim}{\xx}}
\newcommand{\yss}{\underset{\sim}{y}}

\renewcommand{\ss}{\text{\rm ss}}

\newcommand{\OKp}{O_{\smallkay,p}}

\newtheorem{example}[theo]{Example}

\newcommand{\cutter}{\vskip .14in\hrule\vskip .04in}

\newcommand{\I}{\mathbb I}

\newcommand{\hh}{\mathtt h}
\newcommand{\Nilp}{\text{\rm Nilp}}
\newcommand{\nai}{\text{\rm naive}}

\newcommand{\nut}{\widetilde{\nu}}
\newcommand{\oht}{\widehat{O}^\times}
\newcommand{\oh}{\widehat{O}}
\newcommand{\zh}{\widehat{\Z}}
\newcommand{\zht}{\zh^\times}
\newcommand{\ktaf}{\kay^\times_{\A_f}}
\newcommand{\aaa}{\frak a}


\renewcommand{\Lie}{\text{\rm Lie}\, }
\newcommand{\CK}{C_{\smallkay}}

\newcommand{\N}{\Cal N}
\newcommand{\uuxx}{\und{\und{\xx}}}

\newcommand{\OKt}{\widetilde{O}_\smallkay}
\newcommand{\Lt}{\widetilde{L}}

\newcommand{\sra}{\rightarrow}

\newcommand{\MHr}{\M}   
\newcommand{\MH}{\M}
\newtheorem{notation}[theo]{Notation}
\newtheorem{remark}[theo]{Remark}

\newcommand{\Sh}{\text{\rm Sh}}
\newcommand{\SSh}{\und{\Sh}}

\renewcommand{\aa}{a}

\newcommand{\LLL}{[[L]]}

\setcounter{tocdepth}{1}

\newcommand{\LM}{M}

\newcommand{\form}{\rho}

\title{Special cycles on unitary Shimura varieties II: global theory}

\author{Stephen Kudla
\medskip\\
and
\medskip\\
Michael Rapoport}

\maketitle

\centerline{\bf Abstract} 
We introduce moduli spaces of abelian varieties which are arithmetic models of Shimura varieties 
attached to unitary groups of signature $(n-1, 1)$. We define arithmetic cycles on these models and 
study their intersection behavior. In particular,  in the non-degenerate case,  we prove a relation 
between their intersection numbers and Fourier coefficients of the derivative at $s=0$ of a certain 
incoherent Eisenstein series for the group $\UU(n, n)$. This is done by relating the arithmetic cycles to 
their formal counterpart from part I via non-archimedean uniformization, and 
by relating the Fourier 
coefficients to the derivatives of representation densities of hermitian forms. The result 
then follows from the main theorem of \cite{KRunitary.I} and a  counting argument.

\section{Introduction}\label{introduction}
This paper is the global counterpart to \cite{KRunitary.I}. Our purpose here is to\hfb 
$\bullet$ introduce moduli spaces of abelian varieties which are  arithmetic models of Shimura varieties attached 
to unitary groups of signature $(n-1, 1)$,\hfb
$\bullet$ define arithmetic cycles on these models and\hfb 
$\bullet$ study the
intersection pattern of these arithmetic cycles, and in particular prove in the {\it non-degenerate} case the
relation between their arithmetic intersection numbers and Fourier coefficients of the derivative at $s=0$ of a
certain incoherent Eisenstein series $E( z, s, \Phi)$ for the group $\UU (n, n)$ that was predicted in \S 16 of \cite{kudlaannals}.

We now describe our results in more detail. Let $\kay$ be an imaginary quadratic field, with ring of integers
$\OK$. Let $n$ be a positive integer and let $r$ with $0\leq r\leq n$.
We then consider the moduli space $\M (n-r, r)$ over $\Spec\OK$ which parametrizes principally 
polarized abelian varieties $(A, \lambda)$ of dimension $n$ with an action $\iota : \OK\to\End (A)$
of $\OK$.  We require that the Rosati involution corresponding to $\lambda$ induces the non-trivial 
automorphism of $\OK$ and that the representation of $\OK$ on the Lie algebra of $ A$ is equivalent
to the sum of $n-r$ copies of the natural representation and of $r$ copies of the conjugate representation.
Hence, for $n = 1$ and $r = 0$ we obtain the usual moduli space $\M_0$ of elliptic curves with complex
multiplication by $\OK$ (these moduli spaces are Deligne-Mumford stacks and not schemes, but
we will neglect this fact in the introduction).

The special cycles of interest to us are defined as follows. 
To a pair $(A, \iota, \lambda)$, resp.\ $(E, \iota_0, \lambda_0)$ of objects of $\M (n-r, r)$ resp.\ $\M_0$
over a common connected base $S$,
we associate the free $\OK$-module $\Hom_{\OK} (E, A)$ with the positive definite $\OK$-valued hermitian form
given by 
\begin{equation*}
h' (x, y) = \lambda_0^{-1}\circ y^\vee\circ\lambda\circ x\in\End_{O_\smallkay} (E) = \OK\ .
\end{equation*}
For $m>0$, let $T\in\Herm_m (\OK)_{\ge0}$ be a positive semi-definite hermitian matrix of size $m$. The special cycle
$\ZZ(T)$ is the moduli space over $\M (n-r, r)\times\M_0$ which parametrizes $m$-tuples of 
homomorphism $\xx = [x_1, \ldots , x_m]\in\Hom_{\OK} (A, E)^m$ such that
$h' (\xx, \xx) = (h(x_i,x_j))=T$. We refer to $T$ as the {\it fundamental matrix}  of the collection $\xx$. 
After extending scalars from $\OK$ to $\C$, these cycles coincide with the 
cycles studied in \cite{kudlamillson, KMihes}.

These cycles are most interesting in the case when
$r = 1$. In this case, when $m = 1$ and $T = t\in\Z_{> 0}$, they are divisors, which for
$n = 2$ are essentially identical with the cycles considered by Gross and Zagier \cite{grosszagier}.  For
$m\geq 1$ and $T\in\Herm_m (\OK)_{> 0}$,  the cycle $\ZZ(T)$ has codimension $m$ in the generic
fiber of $\M (n-1, 1)\times\M_0$, which has dimension $n-1$, but may have lower  codimension in $\M (n-1, 1)\times\M_0$, which has dimension $n$. We call $T\in\Herm_m (\OK)_{> 0}$ {\it non-degenerate} if $\ZZ(T)$ has codimension $m$ in $\M (n-1, 1)\times\M_0$. We are  led to study the following
intersection problem. 

Let $n = \sum_{i=1}^r m_i$ and let $T_i\in\Herm_{m_i} (\OK)$. The intersection (fiber product)
of the cycles $\ZZ(T_i)$ decomposes according to the fundamental matrices into the disjoint sum, 
comp.\,\cite{kudlaannals},
\begin{equation*}
\ZZ(T_1)\times\dots\times\ZZ(T_r)= \coprod_{T\in\Herm_n (\OK)} \ZZ(T)\ ,
\end{equation*}
where the fiber product is taken over $\M (n-1, 1)\times\M_0$.
Here the matrices $T$ have diagonal blocks $T_1, \ldots , T_r$. The most tractable part of this 
intersection are the summands corresponding to {\it nonsingular} matrices $T\in\Herm_n (\OK)$. In
this case it is easy to see that the support of $\ZZ(T)$ is concentrated in finitely many fibers in positive
characteristics.  We define the contribution of such $T$ to the arithmetic intersection number of {\it non-degenerate} 
$\ZZ(T_1), \ldots , \ZZ(T_r)$ as
\begin{equation*}
\langle \ZZ(T_1), \ldots , \ZZ(T_r)\rangle_T = \sum\nolimits_p\chi \big(\ZZ(T)_p, \Cal O_{\ZZ(T_1)}\otimes^{\mathbb L}\ldots\otimes^{\mathbb L}\Cal O_{\ZZ(T_r)}\big)\cdot\log p\ .
\end{equation*}
Here the derived tensor product appears because the cycles $\ZZ(T_1), \ldots , \ZZ(T_r)$ do not in general
intersect properly, not even along the part $\ZZ(T)$ corresponding to a non-singular fundamental matrix
$T$. Our main result concerns the case when $T$ itself is {\it non-degenerate}, i.e. when $\ZZ(T)$ is a finite set of
points. In this case, all diagonal blocks occurring in $T$ are automatically  non-degenerate. To formulate our main result we must introduce the 
{\it incoherent Eisenstein series}. 

These series are the analogues for $\UU(n,n)$ of the incoherent Eisenstein series for 
$\Sp(2n)$ defined and discussed in detail in \cite{kudlaannals}. To introduce them, 
we fix a character  
$\eta$ of $\kay^{\times}_{\A}/\kay^\times$ whose restriction to $\Q^\times_{\A}$ 
is the quadratic character associated to $\kay$. (Such a choice determines splittings of 
metaplectic covers of unitary groups and hence allows us to work on the linear groups.) 
If $V$ is a hermitian space over $\kay$ of dimension $n$ and $\ph =\tt_v\ph_v \in S(V(\A)^n)$ is a factorizable Schwartz 
function on $V(\A)^n$, with a suitable $K_\infty$-finiteness condition on $\ph_\infty$,
there is a corresponding standard section $\P(s) = \P_{\ph}(s) = \tt_v\P_{\ph_v}(s)$
of the global degenerate 
principal series representation $I(s,\eta)$ of $\UU(n,n)$, induced from the character $\eta(\det)\,|\det|^{s}$
of the maximal parabolic with 
Levi factor $\GL(n)/\kay$ (Siegel parabolic).  This section is {\it coherent} in the sense that it arises 
from the global hermitian space $V$. More precisely, at $s=0$ and for a finite place $v$ non-split in $\kay$, the local degenerate principal series 
$I_v(0,\eta_v)$ is the direct sum of two irreducible representations $R_n(U_v^\pm)$ for local hermitian spaces $U_v^\pm$ 
of dimension $n$ over $\kay_v$, 
where $\chi_v(\det(U_v^\pm))= \pm1$, \cite{kudlasweet}. The space $R_n(U_v^\pm)$ is the image of the local Schwartz space $S((U_v^{\pm})^n)$ 
under the Rallis coinvariant map.  At a split finite place, the local degenerate principal series $I_v(0,\eta_v)$ is irreducible. At the archimedean place, 
$$I_\infty(0,\eta_\infty) = \oplus_{0\le r\le n} R_n(n-r,r),$$
where $R_n(n-r,r)$ is the image of the Schwartz space of a hermitian space of signature $(n-r,r)$, \cite{lee.zhu}. 
For a global hermitian space $V$ the image $R_n(V)$ of $S(V(\A)^n)$ in the global degenerate principal series $I(0,\eta)$ 
under the map $\ph \to \P_{\ph}(0)$ is
the irreducible representation
$$R_n(V) = \tt_v R_n(V_v).$$
This representation is then realized as an automorphic representation of $\UU(n,n)$ by taking the value at $s=0$ 
of the Siegel-Eisenstein series $E(h,s,\P_{\ph})$ formed from a coherent section $\P_{\ph}(s)$. This is a special 
case of the regularized Siegel-Weil formula, proved in the case of unitary groups by Ichino \cite{ichino}, \cite{ichino.II}. 
There is a second type of irreducible constituent of $I(0,\eta)$ -- note that this representation is unitarizable and hence completely 
reducible. These have the form
$$R_n(\Cal C) = \tt_v R_n(\Cal C_v),$$
where $\{\Cal C_v\}_v$ is a collection of local hermitian spaces of dimension $n$ such that $\Cal C_v= V_v$ for almost all $v$, and
$$\prod_v \chi_v(\det(\Cal C_v)) = -1.$$
These constituents of $I(0,\eta)$ and the associated Siegel-Eisenstein series are {\it incoherent} in the sense that they do not arise from a 
global hermitian space. For any standard section $\P(s)$ with $\P(0) \in R_n(\Cal C)$, one has 
$E(h,0,\P)=0$, and the kernel of the Eisenstein map $E(0)$ from $I(0,\eta)$ to the space of automorphic forms 
on $\UU(n,n)$ is precisely the direct sum of the $R_n(\Cal C)$'s as $\Cal C$ runs over the incoherent collections.  

The incoherent Eisenstein series of interest to us are obtained by fixing a global hermitian space $V$ of signature $(n-1,1)$ 
for which there exists a self-dual $\OK$-lattice $L$. Such a space will be called a {\it relevant hermitian space}; 
the set $\Cal R_{(n-1,1)}(\kay)$ of isomorphism classes of such spaces is finite. 
Let $\ph_f\in S(V(\A_f)^n)$ be the characteristic function of $(L\tt_\Z\widehat{\Z})^n$, 
and let $\ph'_\infty$ be the Gaussian in $S((V'_{\infty})^n)$, where $V'_\infty$ has signature $(n,0)$.  The resulting standard
global section 
$$\P(s,L) = \P_{\ph'_\infty}(s)\tt \P_{\ph_f}(s,L)$$
is incoherent and the corresponding Siegel-Eisenstein series $E(h,s,L)$
depends only on the genus of $L$, i.e., the orbit of 
$L$ under the action of $G^V_1(\A_f)$, where $G^V_1 =U(V)$ is the isometry group of 
$V$. There are either one or two genera of self-dual lattices in $V$; we denote by $E(h,s,V)$ 
the sum of the $E(h,s,L)$ as $L$ runs over representatives for the genera.
Finally, we let 
\begin{equation} 
E(h,s) = \sum_{V\in \Cal R_{(n-r,r)}(\smallkay)} E(h,s,V)
\end{equation}
be the sum over the isomorphism classes of relevant hermitian spaces. 
Our main result expresses some of the non-singular Fourier coefficients 
$E'_T(h,0)$
of the derivative $E'(h,0)$ at $s=0$ in terms of  arithmetic intersection numbers of special cycles. For this, it is 
more convenient to pass to the corresponding classical Eisenstein series $E(z,s)$, as in (\ref{semi.classical}), 
where $z\in D_n$, the hermitian 
symmetric space for $\UU(n,n)$. 
This 
Eisenstein series has the form, \cite{braun}, \cite{shimura}, \cite{nagaoka}, 
$$E(z,s) = \det v(z)^{\frac{s}2} \sum_{\gamma\in \Gamma_\infty\back \Gamma} \det(cz+d)^{-n}|\det(cz+d)|^{-s}\,\P_f(\gamma, s),$$
where the series is convergent for $\Re(s)> n$.
Here $\Gamma=U(n,n)\cap \GL_{2n}(\OK)$,  $\Gamma_{\infty}$ is the subgroup for which lower left $n\times n$ block is 
zero. The meromorphic analytic continuation 
and holomorphy on the line $\Re(s)=0$ follow from Langlands' general results. By the incoherence condition, 
$E(z,0)=0$,
and, as explained in section \ref{sectionincoherent}, the Fourier expansion of the central derivative has the form
\begin{align*}
E'(z,0) 
{}& =\sum_{T\in \Herm_n(\OK)_{>0}} a(T)\, q^T + \sum_{\substack{T \\ \snass \text{other}}} a(T,v(z))\,q^T,
\end{align*}
where $q^T = \exp(2\pi i \tr(Tz))$. In particular, the terms for positive definite $T$ are holomorphic functions of $z$. 

\begin{theo}
Let $T\in\Herm_n (\OK)_{>0}$ be nonsingular with diagonal blocks $T_1, \ldots , T_r$. Let $\Diff_0(T)$ be the set of primes $p$ that are inert in $\kay$  
for which $\ord_p(\det(T))$ is odd.

(i) If $|\Diff_0 (T)|\geq 2$, then $\ZZ(T) = \emptyset$ and $E'_T (z, 0) = 0$. 

(ii) If $\Diff_0 (T) = \{p\}$ with $p > 2$, then $T$ is non-degenerate if and only if it is
$\GL_n ( O_{\smallkay, p})$-equivalent to $\diag (1_{n-2}, p^a, p^b)$ for some $0\leq a < b$ with $a+b$
odd. In this case $\ZZ(T)$ has support in the supersingular locus in characteristic  $p$ and
\begin{equation*}
\langle \ZZ(T_1), \ldots , \ZZ(T_r)\rangle_T = \text{\rm  length}(\ZZ(T))\cdot\log p\ .
\end{equation*}
Furthermore, in this case
\begin{equation*}
E'_T (z, 0) =E'_T(z,0,V)= C_1\cdot\langle \ZZ(T_1) , \ldots , \ZZ(T_r)\rangle_T\cdot q^T\ .
\end{equation*}
for an explicit constant $C_1$, independent of $T$, where $V\in \Cal R_{(n-1,1)}(\kay)$ is the unique relevant hermitian space which 
differs from $V_T$ only at $\infty$ and $p$.\hfb
 Here $V_T$ denotes the space $\kay^n$ with hermitian form defined by $T$.
\end{theo}

The strategy of the proof of this theorem is similar to that of the proof of the analogous theorem for Shimura curves \cite{KRYbook}.
We first prove that for nonsingular $T\in\Herm_n (\OK)_{>0}$ the cycle $\ZZ(T)$ is either empty or concentrated in the
supersingular locus in finitely many characteristics $p$, where $p$ is not split in $\kay$, and in fact concentrated in characteristic $p$,  if 
$\Diff_0 (T) = \{p\}$. Then we use the theory of non-archimedean uniformization of \cite{RZ},  to
reduce the calculation of the length of $\ZZ(T)$ to a combination of a local calculation on a formal moduli space
of $p$-divisible groups and a point count. The first problem was solved in  our previous paper \cite{KRunitary.I}.
The second problem is solved here in section \ref{sectionarith.degree}. It then remains to
calculate the Fourier coefficient corresponding to $T$ of the derivative of the 
incoherent Eisenstein series. For this we use the 
Siegel-Weil formula established by Ichino \cite{ichino}, \cite{ichino.II}, in this case. We must ultimately calculate some
representation densities for hermitian forms, which is in general a very difficult task, even though
a general formula due to Hironaka \cite{hironaka}, \cite{hironaka1}, exists. Fortunately in the non-degenerate case, these
calculations are manageable and a direct comparison gives the formula in our main theorem.

We now put our main result in perspective and indicate possible directions of further research. Of
course, the model for the results above and in fact the origin of the whole program lies in  the theory of 
special cycles on Shimura varieties attached to orthogonal groups of signature $(n-1, 2)$, and considered for low values of $n$ in our papers \cite{kudlaannals}, 
\cite{KR1}, \cite{KR2}, \cite{KRY.tiny}. However,
as explained in the introduction of \cite{KRunitary.I}, there are serious problems with the construction of arithmetic
models of these Shimura varieties as soon as $n\geq 6$. By contrast, in the case considered here, which
is related to Shimura varieties attached to unitary groups of signature $(n-1, 1)$, manageable arithmetic models
 exist and can be studied for arbitrary $n$. In fact, we define such models for unitary
groups of arbitrary signature $(n-r, r)$ -- but with no level structure. In the case of deeper level structure such arithmetic models can surely also be defined, although results on arithmetic intersection numbers of special cycles will then be harder to come by. 

In the case of signature $(n-1,1)$, the most promising next steps to be taken seem to be the following:\hfb
$\bullet$ Let $T\in\Herm_n (\OK)_{>0}$ with diagonal blocks $T_1, \ldots , T_r$, which are assumed to be non-degenerate. Assuming 
$\Diff_0 (T) =\{p\}$ with $p>2$ inert in $\kay$, prove that $\langle \ZZ(T_1), \ldots , \ZZ(T_r)\rangle_T$ is given by the
same formula as in the Main Theorem above. To prove this in general\footnote{In this context, we refer to Terstiege's forthcoming paper \cite{terstiege3}, in which he deals with this problem in the case $n=3$.}, one will have to deal with
degenerate intersections, which most probably also requires a better knowledge of the structure of the special cycles outside the supersingular locus. \hfb
$\bullet$ Investigate in detail the reduction modulo a ramified prime $p$ of the moduli space $\M (n-1, 1)$ and
its supersingular locus and use this to establish results on the arithmetic intersection numbers $\langle \ZZ(T_1), \ldots , \ZZ(T_r)\rangle_T$ in the case
when $T\in\Herm_n (\OK)_{>0}$ has $\Diff_0 (T) = \emptyset$.\hfb
$\bullet$  One can define variants of our arithmetic models which involve some parahoric
level structure.  It should be   possible to exploit our considerable
knowledge on such models obtained in recent years, \cite{pappas.JAG}, \cite{PR.I}, \cite{PR.II}, \cite{PR.III}, to investigate special cycles in this context.
  
The situation becomes more speculative when $T\in\Herm_n (\OK)$ is singular.  In this case one would like to
equip the special cycles with appropriate Green forms and define classes in certain arithmetic Chow
groups. Then various cup products of these arithmetic cycles should be related to various special values
of derivatives of Eisenstein series on unitary groups of type $(n, n)$. Here the fact that the
moduli space $\M(n-1, 1)$ is not proper will play a crucial role. The Eisenstein series we consider in
this paper are conjecturally related to the cup product with values in ${\CH}^n \big(\M (n-1, 1)\big)$ -- but the
non-compactness of these moduli spaces prevents us from making this more precise. No doubt the
extended versions of arithmetic Chow groups defined by Burgos, Kramer, K{\"u}hn \cite{BKK} will have an
impact on these questions.

One may also try to generalize our results in other directions. One may expect similar results when  $\kay$ is replaced by an arbitrary CM-field. 
Also, our main results in this paper concern special cycles which lie above
\begin{equation*}
\M=\M (n-1, 1)\times\M (1, 0)\ .
\end{equation*}
However, we define $\M(n-r, r)$ for arbitrary $r$ and one can similarly define more general cycles over
more general products. It is a challenge then to determine arithmetic intersection numbers and 
 to form a sensible generating function using them, which can be compared with some automorphic counterpart. 
 
 As should be clear from the above description, the investigation of special cycles on unitary Shimura varieties is largely uncharted territory. 
 This also explains why we have made an effort to explain the relation to  previous work; we explain in 
 section \ref{sectioncxunif} the precise relation to the cycles in \cite{kudlaU21, kudlamillson, KMihes} ({\it KM-cycles}), and
 in section \ref{sectionGrossZ} the relation to the cycles ({\it Heegner points})
 considered by Gross and Zagier.  In a sequel, \cite{KR.occult}, we explain how KM-cycles arise in the theory of {\it occult period mappings}. 
 
 We now explain the lay-out of the paper. The paper has five parts.\hfb  In Part I, we give the 
 definition of the moduli stack $\M(n-r, r)$ and define the special cycles $\ZZ(T)$ on them. 
 We also show how to uniformize the orbifold  of complex 
 points of these stacks in terms of  the space of negative $r$-planes in $\C^n$, and make the connection
between the moduli stacks $\M(n-r, r)$ and certain  Shimura varieties associated to unitary 
groups. In particular, this allows us to describe the set of 
connected components of $\M(n-r, r)_\C$, which is used later in the examples.\hfb 
 In Part II, we describe the completion of $\M(n-r, r)$ along the 
supersingular locus in the fiber of $\M(n-r, r)$ at a prime $p$ which is inert  in 
$\kay$, in terms of the formal moduli space of $p$-divisible groups
that was studied extensively in \cite{vollwedhorn}. This $p$-adic uniformization is 
essentially just spelling out the general method of 
\cite{RZ} in the special case at hand. We also exhibit an analogous $p$-adic 
uniformization of the completion of the special cycles  along their
supersingular locus. \hfb 
Part III is devoted to the computation of the nonsingular Fourier coefficients of the 
central derivative of an incoherent Eisenstein series.  First, we review the theory of 
theta integrals for unitary groups and the regularized Siegel-Weil formula, 
due in this case to Ichino \cite{ichino},  
 that relates these to special values of certain Eisenstein series for unitary groups. 
 Then, in the incoherent case, the $T$th Fourier  coefficient
  of the central derivative
 for a positive definite $T$ is expressed as a product of a representation number 
 of $T$ by a genus of definite lattices 
 and a derivative of a local Whittaker function. The values and derivatives of such Whittaker functions are given in terms of 
 representation densities for hermitian forms and their derivatives. \hfb
 In Part IV we prove our main results, by determining explicitly the arithmetic intersection numbers in 
 the non-degenerate case, and by comparing the result with the special values of the derivatives at $0$ 
 of the relevant Eisenstein series. \hfb
Finally,  Part V is devoted to 
 examples and variants of our main result. In particular, we give a more detailed description of the case $n=2$ 
 and explain the precise relation of our special cycles 
for $\M(1,1)$ to those introduced by Gross and Zagier \cite{grosszagier}.

We thank U.\ Terstiege for helpful discussions. 
This project was
started at the Hirschberg conference in 1996 organized by J. Rohlfs and J.
Schwermer.  We also gratefully 
acknowledge the hospitality of the Erwin-Schr\"odinger-Institut, 
where part of this work was done, and the support of the Hausdorff Center  of Mathematics in Bonn. 
The first author's research was partially supported by an NSERC Discovery Grant. 
Finally, we thank the referee for a thorough reading of the manuscript and for many helpful suggestions 
concerning both content and exposition. 

\bigskip

\centerline{{\it Notation and conventions}}

We fix an imaginary quadratic field $\kay= \Q(\sqrt{\Delta})$ with discriminant $\Delta$, ring of integers $\OK$, 
and nontrivial Galois automorphism 
$\aa\mapsto \aa^\s$, $\aa\in \kay$. As usual, we denote by $h_\smallkay$ the class number and by $w_\smallkay=|\OK^\times|$ the number of units.  
We view $\kay$ as a subfield of $\C$ via an embedding $\tau$ and require that $\tau(\sqrt{\Delta})$ have positive imaginary part. 
For a rational prime $p$, we write $O_{\smallkay,p}= \OK\tt_\Z\Z_p$ and $O_{\smallkay, ( p )} = \OK\tt_\Z \Z_{( p )}$ where $\Z_{( p )}$ is the localization of $\Z$ at $p$. 

For a hermitian space $V$, $(\ ,\ )$ of dimension $n$ over $\kay$, let $\det(V)\in \Q^\times/N(\kay^\times)$ be the determinant of the matrix $((v_i,v_j))$ where
$\{v_i\}$ is a $\kay$-basis for $V$. 
Note that we take $(\ ,\ )$ to be conjugate linear in the second argument.  Define invariants $\sig(V)= (r,s)$, $r+s=n$, and $\inv_p(V) = \chi_p(\det(V))$, where $\chi_p(a) = (a,\Delta)_p$. 
Here $(\ ,\ )_p$ is the quadratic Hilbert symbol for $\Q_p$.  Note that $\inv_\infty(V) = (-1)^s$ and that $\inv_p(V)=1$ for all split primes $p$. 
For a fixed dimension $n$, the isometry class of the hermitian space $V_p$ over $\kay_p$ 
(resp. $V_\infty$ over $\C$) is 
determined by $\inv_p(V_p)$ (resp. $\sig(V_\infty)$). Moreover, the isometry class of $V$ over $\kay$ is determined by the 
collection of its local invariants (Hasse principle) and, for any collection of local hermitian spaces $\{V_p\}$ satisfying the product formula
\begin{equation}\label{product.formula}
\prod_{p\le \infty} \inv_p(V_p) = 1,
\end{equation}
there is a unique global hermitian space with the $V_p$'s as its local completions (Landherr's Theorem). 

For a hermitian space $V$ over $\kay$, there is an associated alternating form defined by 
$\gs{x}{y} = \tr((x,y)/\sqrt{\Delta}).$
Note that, for $\aa\in \kay$,  $\gs{\aa x}{y}=\gs{x}{\aa^\s y}$ and that the hermitian form is given by
\begin{equation}\label{recover}
2(x,y) = \gs{\sqrt{\Delta}\,x}{y}+ \gs{x}{y}\sqrt{\Delta}.
\end{equation}
Conversely, if $V$ is a $\kay$-vector space with a $\Q$-bilinear alternating form $\gs{\ }{\ }$ satisfying $\gs{\aa x}{y}=\gs{x}{\aa^\s y}$, then 
(\ref{recover}) defines a hermitian form on $V$. 
An $\OK$-lattice in $V$ is self-dual for $(\ ,\ )$ 
if and only if it is self-dual for $\gs{\ }{\ }$.

\tableofcontents

\centerline{\bf\large Part I:  The global moduli problem and special cycles }

\section{The global moduli problem}\label{sectionglobalmoduliproblem}

\subsection{}
For an integer $n\ge1$ and a decomposition $n=(n-r)+r$ with $0\le r \le n$, we define a groupoid 
$$ \M(n-r,r)^{\nai} = \M(\kay, n-r, r)^\nai$$ 
fibered over $(\text{\rm Sch}/\Spec \OK)$ by associating to a locally noetherian $\OK$-scheme $S$ the groupoid
of triples $(A,\iota,\l)$. Here $A$ is an abelian scheme over $S$, $\iota:\OK\rightarrow \End_S(A)$ is an action 
of $\OK$ on $A$,  and $\l:A \rightarrow A^\vee$ is a principal polarization such that
$$\iota(\aa)^* = \iota(\aa^\s)$$
for the corresponding Rosati involution $*$. In addition, the following signature condition is imposed: 
\begin{equation}
\cha(T, \iota(\aa)\mid \Lie A) = (T -\ph(\aa))^{n-r}(T-\ph(\aa^\s))^{r},\qquad \aa\in \OK.
\end{equation}
where $\ph:\OK\rightarrow \Cal O_S$ is the structure homomorphism.
Here the left side is the characteristic polynomial in $\Cal O_S[T]$ of the $\Cal O_S$-module 
endomorphism of $\Lie A$ induced by $\iota(\aa)$. 
In particular, $A$ is of relative dimension $n$ over $S$.\hfb
A morphism in  $\M(n-r,r)^{\nai}(S)$ from $(A,\iota,\l)$ to $(A',\iota',\l')$ is an $\OK$-linear isomorphism $\alpha: A\to A'$ such that $\alpha^*(\l')=\l$.

\begin{prop} $\M(n-r,r)^{\nai}$ is a Deligne-Mumford stack over $\Spec \OK$. Furthermore, 
$\M(n-r,r)^\nai \times_{\Spec \OK}\Spec \OK[\Delta^{-1}]$ is smooth of relative dimension $(n-r)r$ 
over $\Spec \OK[\Delta^{-1}]$. 
\end{prop}
\begin{proof} The representability by a DM-stack follows from the representability by a DM-stack of the stack 
of principally polarized abelian varieties and the relative representability of the forgetful map which 
forgets the $\OK$-action. This relative representability follows from the theory of Hilbert schemes. 
The smoothness assertion is checked by the infinitesimal criterion for smoothness and Grothendieck-Messing theory
\cite{RZ}, \cite{pappas.JAG}. 
\end{proof}

\begin{example} Let $n=1, r=0$. Then $\M(1, 0)^\nai$ parametrizes triples $(E,\iota_0,\l_0)$ where $(E,\iota)$ is an elliptic curve 
with complex multiplication by $\OK$ such that the action of $\OK$ on $\Lie E$ is the natural one. In this case, the polarization $\l_0$ is uniquely determined. The coarse moduli space 
of $\M(1, 0)^\nai$ is $\Spec O_H$ where $H$ is the Hilbert class field of $\kay$. Note that $\End_{\OK}(E/S,\iota_0)= \OK$
for any $(E,\iota_0,\l_0)\in \M(1,0)^\nai(S)$. This example is discussed in \cite{KRY.tiny}. We will abbreviate $\M(1, 0)^\nai$ to $\M_0$.\hfb
We note that there is a natural isomorphism between the moduli stacks $\M(n-r, r)^{\nai}$ and $\M(r, n-r)^{\nai}$  which associates to $(A, \iota, \l)$ 
its conjugate $(A, \bar \iota, \l)$, where  the $\OK$-action on $A$ has been changed to its conjugate, i.e., 
$\bar\iota(\aa) = \iota(\aa^\s)$. 
\end{example}

\begin{remark} As was first pointed out by Pappas, \cite{pappas.JAG}, $\M(n-r,r)^\nai$ is not flat over $\Spec \OK$ for $n\ge 3$. 
Pappas defines a closed substack 
of $\M(n-r,r)^\nai$ by imposing an additional condition as in the following definition. 
\end{remark}
\begin{defn} Let $\M(n-r,r)$ be the closed substack of $\M(n-r,r)^{\nai}$ 
consisting of those triples $(A,\iota,\l)$ for which the action of $\OK$ on $\Lie A$ satisfies the {\it wedge condition}:
\begin{equation}
\wedge^{r+1}(\iota(\aa)-\aa)=0, \qquad \wedge^{n-r+1}(\iota(\aa)-\aa^\s)=0.
\end{equation}
\end{defn}
For  $n\le 2$, this condition follows from the signature condition. 
Furthermore, over $\Spec \OK[\Delta^{-1}] $, 
\begin{equation}
\M(n-r,r)[\Delta^{-1}]=\M(n-r,r)^\nai[\Delta^{-1}] .
\end{equation}

\begin{theo}{\rm (Pappas)}\label{pappasflatness}
Let $r=1$, and assume $2\nmid \Delta$. Then the  stack $\M(n-1,1)$ is flat 
over $\Spec \OK$. 
\end{theo}
\begin{proof}
This is  \cite{pappas.JAG}, Theorem ~4.5, a).
\end{proof}

Finally, the following stack will play a fundamental role, so that we introduce a symbol for it. 
\begin{notation} For a given fixed collection $\kay, n, r$, we  let
$$\MHr = \M(n-r,r)\times_{\Spec \OK} \M_0$$
be the base change of $\M(\kay; n-r,r)$ to $\M_0$. 
\end{notation}

\subsection{} Suppose that $(E,\iota_0,\l_0)\in \M_0(S)$ and 
$(A,\iota,\l)\in \M(n-r,r)(S)$, i.e., for an element of $\M(S)$, are given.  When $S$ is connected, we consider the free $\OK$-module of finite rank 
$$V'(A, E)=\Hom_{\OK}(E,A).
$$
 On this $\OK$-module there is a  $\OK$-valued hermitian form
given by
\begin{equation}\label{fundherm}
h'(x,y) = \iota_0^{-1}(\l_0^{-1}\circ y^\vee\circ \l\circ x) \in \OK.
\end{equation}
where $y^\vee:A^\vee\rightarrow E^\vee$ denotes the dual homomorphism. 
\begin{lem} The hermitian form $h'$ on $V'(A, E)$ is positive-definite.
\end{lem}
\begin{proof}
Consider the endomorphism $\alpha\in \End(E\times A)$ given by 
$$
\alpha=\begin{pmatrix}0&\l_0^{-1}x^\vee\l\\ x&0\end{pmatrix} . 
$$
The adjoint $\alpha^*$ with respect to the polarization $(\l_0, \l)$ of $E\times A$ is
$\alpha^*=\alpha$. Hence 
$$\alpha\alpha^*={\rm diag}(\l_0^{-1}x^\vee\l x, x \l_0^{-1}x^\vee\l)\in \End(E)\times \End(A). $$
The positivity of the Rosati involution implies that the first entry of this diagonal matrix is positive, as had to be shown. 
\end{proof}

\begin{defn}
Let $T\in \Herm_m(\OK)$ be an $m\times m$ hermitian matrix, $m\ge1$, with coefficients in $\OK$. 
The {\it special cycle} $\ZZ(T)$ attached to $T$ is the stack of collections
$(A,\iota, \l, E,\iota_0,\l_0;\xx)$ where $(A,\iota,\l)\in \M(n-r,r)(S)$, $(E,\iota_0,\l_0)\in \M_0(S)$, 
and $\xx = [x_1,\dots, x_m]\in \Hom_{\OK}(E,A)^m$ is an $m$-tuple of homomorphisms
such that 
\begin{equation}
h'(\xx,\xx)=(h'(x_i,x_j)) = T.
\end{equation}
\end{defn}

\begin{prop} $\ZZ(T)$ is representable by a DM-stack. The natural morphism
$\ZZ(T) \to \MHr$ 
is finite and unramified. 
\end{prop}
\begin{proof}
Given an $S$-valued point $(A,\iota, \l, E,\iota_0,\l_0)$ of $\MHr$, 
the functor \hfb 
$\und{\Hom}_{\OK}(E,A)$ on $(\text{\rm Sch}/S)$ defined by 
$$S'\ \rightsquigarrow \Hom_{\OK}(E\times_SS', A\times_SS')$$
is representable by a scheme which is unramified over $S$ (by rigidity) and satisfies the valuative criterion 
for properness (by the N\'eron property of abelian schemes).  The finiteness now follows, since, by the positive definiteness of 
$h'$, the 
set 
$$\{\ \xx\in \Hom_{\OK}(E, A)^m\mid h'(\xx, \xx) =T\ \}$$
is finite. 
\end{proof}

%
\subsection{}
There is an obvious Tate module variant of the hermitian space $V'(A, E)$. Let $F$ be an algebraically closed field of characteristic $p$ and let 
$$
(A, \iota, \l; E, \iota_0, \l_0)\in \MHr(F).
$$ Denoting by $T^p(A)^0$, resp. $T^p(E)^0$ the rational Tate modules prime to $p$ of $A$ resp. $E$, let 
$$V'_{\A_f^p}=\Hom_{\kay\tt\A_f^p}(T^p(E)^0, T^p(A)^0),
$$ with hermitian form 
$$
h'(x, y)= \iota_0^{-1}(\l_0^{-1}\circ y^\vee\circ\l\circ x)\in \kay\tt\A_f^p. 
$$
Then the natural embedding $V'(A, E)\to V'_{\A_f^p}$ is isometric. 
On the other hand, the hermitian form $h'(\ ,\ )$ on $V'_{\A_f^p}$ is related to the Weil pairing as follows.
We fix an isomorphism $\A_f^p(1)\simeq \A_f^p$. Then  the natural pairing $e_A$ takes values in $\A_f^p$, 
$$e_A:T^p(A)^0\times T^p(A^\vee)^0 \lra \A^p_f,$$
and there is an hermitian form $h=h_\l$ on $T^p(A)^0$  given by 
\begin{equation}\label{Weil-form}
2\, h_\l(x,y) = e_A(\delta x,\l(y)) + \delta e_A(x,\l(y)),
\end{equation}
where $\delta = \sqrt{\Delta}$. \hfb
 The same construction can be made with $E$ in place of $A$. The hermitian forms $h_\l$ and $h_{\l_0}$ define a hermitian structure 
  $h(\ ,\ )$ on $\Hom_{\smallkay\tt\A_f^p}(T^p(E)^0, T^p(A)^0)$, which is independent of the trivialization of $\A_f^p(1)$. 
  Hence we may replace the base scheme $\Spec F$ by any connected $O_{\smallkay, (p)}$-scheme $S$. 
\begin{lem}
The two hermitian forms $h'(\ ,\ )$ and $h(\ ,\ )$ on $$\Hom_{\smallkay\tt\A_f^p}(T^p(E)^0, T^p(A)^0)$$ are identical. 
\end{lem}\label{compWeil}
\begin{proof}
We choose an identification of $T^p(E)^0$ with $\kay\tt\A_f^p$, i.e., a basis vector $1$ in $T^p(E)^0$. 
\newcommand{\uxo}{\und{x}(1)}
We calculate for $x, y\in \Hom_{\smallkay\tt\A_f^p}(T^p(E)^0, T^p(A)^0)$
\begin{align*}
2\,h_\l(x(1),y(1)) &= e_A(\delta x(1),\l(y(1))) + \delta\,e_A(x(1),\l(y(1)))\\
\nass
{}&= e_E(\delta, x^\vee\l \,y(1))+\delta\,e_E(1,x^\vee\l\,y(1))\\
\nass
{}&= 2\, h_{\l_0}(1,\l_0^{-1}\,x^\vee\l \,y(1))\\
\nass
{}&= 2\,h'(x,y)\,h_{\l_0}(1,1).
\end{align*}
 This proves the claim. 
\end{proof}

\subsection{} We define the set $\Cal R_{(n-r, r)}(\kay)$ of {\it relevant hermitian spaces} of dimension $n$ over $\kay$ as
the set of isomorphism classes of 
hermitian spaces $V$ with 
$\sig(V) = (n-r,r)$ and which contain a self-dual $\OK$-lattice. 
\begin{lem}\label{relevant.herm} (i) The cardinality of  $\Cal R_{(n-r, r)}(\kay)$ is
$$|\Cal R_{(n-r, r)}(\kay)| = 2^{\delta -1} ,$$
where $\delta$ is the number of primes that ramify in $\kay$, i.e., the number of distinct prime divisors of $\Delta$. \hfb
(ii) The number of \emph {strict} similarity classes of relevant hermitian spaces is
$$|\Cal R_{(n-r,r)}(\kay)/\text{\rm str.sim.}| = \begin{cases}
2^{\delta -1}&\text{if $n$ is even,}\\
\nass
1&\text{if $n$ is odd.}
\end{cases}
$$
{\rm Here by a strict similarity we mean a similarity such that the scale factor is positive.}
\end{lem}
\begin{proof}
By the Hasse principle, the isomorphism class of 
a hermitian space is determined by its localizations. The existence of a self-dual lattice is equivalent to the condition that the 
determinant $\det(V)$ lies in $\Z_p^\times N(\kay_p^\times)$ for all finite primes $p$. For split and ramified primes, this local condition is 
automatic, while, for inert primes it is equivalent to $\inv_p(V)= (\det V,\Delta)_p=1$.  Since the signature is also fixed, the relevant spaces are determined by 
the collection of signs $\inv_p(V)$ for $p\mid \Delta$, and any collection of signs is realized, subject to the condition that 
$$\inv_\infty(V)= \prod_{p\mid \Delta}\inv_p(V). $$
Here note that $\inv_\infty(V)=(-1)^r$. 
This proves (i).\hfb
To prove (ii), note that  the determinants of 
similar spaces differ by the $n$th power of the scale factor.  Thus, 
for $n$ even, 
two hermitian space are similar if and only if they are isomorphic, while, for $n$ odd, two relevant hermitian spaces are locally similar by a unit
at each ramified prime and hence are globally similar, by the Hasse principle for similitudes. 
\end{proof}

\begin{prop}\label{disjsum} 
(i) There is a natural disjoint decomposition of algebraic stacks
$$\M(n-r,r) = \coprod_{V\in \Cal R_{(n-r, r)}(\smallkay)/\text{\rm str.sim.}} 
\M(n-r,r)^V.$$
(ii) 
There is a natural disjoint decomposition of algebraic stacks
$$\MHr = \M(n-r,r)\times_{\Spec \OK} \M_0= \coprod_{V\in \Cal R_{(n-r, r)}(\smallkay)} 
\MHr^V.$$
For $n$ even, this decomposition is obtained by base change from that in (i).
\end{prop}
\begin{remark} Of course, for $n$ odd, part (i) is trivial, since there is only one strict similarity class of relevant $V$'s. 
\end{remark} 
\begin{proof}
Let  $(A,\iota,\l)$ be in $\M(n-r,r)(S)$ for a connected base $S$.   Let $s:\Spec F\to S$ be a geometric point of $S$.  
First suppose that $F$ has characteristic zero, and choose an isomorphism $\hat\Z(1) \isoarrow \hat \Z$ over $F$. 
We obtain from the pull-back $\xi=\xi(s)$ to $F$ of $(A,\iota,\l)$ the rational Tate module $T(A)^0$ with its  
Riemann form $\gs{\ }{\ }_\l$  associated to the polarization $\l$. This satisfies $\gs{\aa x}{y}_\l=\gs{x}{\aa^\s y}_\l$
and hence determines a hermitian form $(\ ,\ )_\l$ by (\ref{recover}). Hence we obtain a 
hermitian space  $\Cal V$ over $\kay\tt_\Q\A_f$ with a self-dual lattice given by the Tate module of $(A,\iota,\l)$. 
We claim that there is a unique element $V(\xi)$ in $\Cal R_{(n-r, r)}(\kay)$ which after tensoring with $\A_f$ 
gives the hermitian space $\Cal V$. The uniqueness is clear by the Hasse principle and the product formula. 
For the existence, note  that the point $\xi$ 
arises, via base change, from a point $\xi_0\in \M(n-r,r)(F_0)$ for a subfield $F_0\subset F$ which is finitely 
generated over the prime field and 
hence has a complex embedding. If $F$ has a complex embedding $F\hookrightarrow \C$,  let $V(\xi) = H_1(A_\C,\Q)$ 
be the rational homology of the corresponding complex abelian variety, which, by the same argument as above, is a 
hermitian vector space over $\kay$.  By the signature condition, the space $V(\xi)$ has signature $(n-r,r)$, and 
by the compatibility between singular homology and Tate module, $V(\xi)\tt\A_f= \Cal V$. Thus $V(\xi)\in \Cal R_{(n-r, r)}(\kay)$. \hfb
Next, suppose that $F$ has characteristic $p>0$, and choose an isomorphism $\hat\Z^p(1) \isoarrow \hat \Z^p$ over $F$. We obtain a
hermitian space  $\Cal V^p$ over $\kay\tt_\Q\A_f^p$ with a self-dual lattice.  There exists a unique hermitian 
space $V(\xi)$ with $V(\xi)\tt\A_f^p\simeq \Cal V^p$ and with ${\rm sig}(V(\xi))=(n-r, r)$. 
We claim that $V(\xi)\in  \Cal R_{(n-r, r)}(\kay)$. If $p$ is ramified or split in $\kay$, the space $V(\xi)_p$ always 
has a self-dual $\OK\tt\Z_p$-lattice, so that $V(\xi)\in \Cal R_{(n-r, r)}(\kay)$.  If $p$ is inert in $\kay$, 
we use the fact that $\M(n-r,r)$ is smooth at $p$. 
Hence there exists a point $\tilde\xi= (\tilde A, \tilde \iota, \tilde\l)\in \M(n-r,r)(W(F))$ lifting $\xi$, where $W(F)$ is the ring of Witt vectors.   
Then 
$$T^p(A)^0 = T^p(\tilde A)^0 = V(\tilde \xi)\tt_\Q\A_f^p,$$
as hermitian spaces over $\kay\tt_\Q\A_f^p$. Also, by the previous argument,  $V(\tilde\xi)$
has signature $(n-r,r)$, so that $V(\xi)$ and $V(\tilde\xi)$ are locally isomorphic at all places other than $p$. 
Hence $V(\xi)_p\simeq V(\tilde \xi)_p$ as well and this space has a self-dual $\Z_p\tt\OK$-lattice. 
Again, we conclude that $V(\xi)\in \Cal R_{(n-r, r)}(\kay)$. 

The space $V(\xi)$ attached above to $(A, \iota, \l)$ depends on the choice of the geometric point $s$ of $S$, 
on the complex embedding, and on the trivialization of the group of roots of unity.
A change in these choices changes the space $V(\xi)$ within a strict similarity class, as we check below. 
However, if $(A, \iota, \l)$ 
and $(E, \iota_0, \l_0)$ are points of $\M(n-r, r)$ and $\M_0$ over $S$, then attaching as above the 
hermitian spaces $V(\xi)$ to $(A, \iota, \l)$ and $V(\xi_0)$ to $(E, \iota_0, \l_0)$, 
the hermitian space $V=\Hom_\smallkay(V(\xi_0), V(\xi))\in \Cal R_{(n-r, r)}(\kay)$ is 
independent of all choices. Indeed, if the trivialization $\widehat{\Z}^p(1)\simeq \widehat{\Z}^p$ of the prime-to-$p$ roots of unity is changed by a 
scalar $c\in \big(\widehat{\Z}^p\big)^\times$, then $V(\xi)\otimes\A_f^p$ and $V(\xi_0)\otimes\A_f^p$ are both scaled by the same scalar $c$, hence 
$V\otimes\A_f^p=\Hom_\smallkay(V(\xi_0), V(\xi))\otimes \A_f^p$ is unchanged. Since the archimedean localization of $V$ is determined by the signature condition, the product formula and the Hasse principle imply that $V$ is unchanged in its isometry class. Similarly, if two geometric points 
$s$ and $s'$ have the same image in $S$, then $V(\xi)\otimes\A_f^p$ and $V(\xi')\otimes\A_f^p$ are scales of one another  by the  scalar $c\in \big(\widehat{\Z}^p\big)^\times$ which compares the corresponding trivializations of the prime-to-$p$ roots of unity; the same applies to $V(\xi_0)\otimes\A_f^p$ and $V(\xi_0')\otimes\A_f^p$, hence $\Hom_\smallkay(V(\xi_0), V(\xi))\otimes \A_f^p\simeq\Hom_\smallkay(V(\xi_0'), V(\xi'))\otimes \A_f^p$,
and we conclude as before by the product formula and the Hasse principle. The same argument takes care of the change of the complex embedding. 
The independence of the image point in $S$ of the chosen geometric point is proved by using the specialization homomorphism from a generic geometric point. 

This defines the claimed disjoint decomposition in (ii), and also proves (i). 
\end{proof}

Over $\Spec \OK[\frac12]$,  a slight refinement of the decomposition of Proposition~\ref{disjsum} will be useful.
For a relevant hermitian space $V$, let $G^V_1 = \UU(V)$ be the isometry group. 
Then $G^V_1(\A_f)$ acts 
on the set of self dual lattices in $V$  by $g: L\mapsto V\cap\big(g(L\otimes \widehat{\Z})\big)$.  The orbit of a lattice $L$ is the $G_1^V$-genus of $L$; we denote it by $[[L]]$. 
The following result, due to Jacobowitz \cite{jacobowitz}, sections 9 and 10, describes the orbits. 
\begin{prop}\label{genera} {\rm (}\cite{jacobowitz}{\rm )} Suppose that $V_p$ is a non-degenerate hermitian space 
of dimension $n$ over $\kay_p/\Q_p$ and that $V_p$ contains a self-dual lattice.  
Then the unitary group $\UU(V_p)$ acts transitively on the set of self-dual lattices in $V_p$ except in the following cases: \hfb
(a) $p=2$, $\kay/\Q_p$ is ramified, $n=\dim V_p$ is even and $V_p$ is a split space. \hfb 
(b)  $p=2$, $\kay=\Q_p(\sqrt{\Delta})$ where $\ord_2(\Delta) =3$, $n=\dim V_p$ is even and $V_p$ is the sum of a $2$-dimensional 
anisotropic space and a split space of dimension $n-2$. \hfb
Then there are two $\UU(V_p)$-orbits of self-dual lattices in $V_p$.
\end{prop} 

\begin{remark}\label{genera.remark} (i)   Explicit representatives for the two orbits can be given as follows. 
Let $H(0)$ be the hyperbolic plane, i.e., $2$-dimensional 
space with $O_{\smallkay,2}$-basis $e$, $f$ and $(e,f)=1$, $(e,e)=(f,f)=0$.  
Note that the $\Z_2$-ideal generated by the set of values $(x,x)$, 
$x\in H(0)$ is then $\tr(O_{\smallkay,2}) = 2\Z_2$. Let $\dim V_2 = n=2k$. 
In case (a), the orbit representatives are $H(0)^k$ and $\diag(1,-1)\oplus H(0)^{k-1}$. 
In case (b), take $\kappa\in \Z_2^\times$ such that $\kappa$ is not a norm from $\kay_2$. 
There is then a unit $\l\in O_{\smallkay,2}^\times$ with $N(\l) -\kappa\in 4 \Z_2$. 
This element is unique modulo $2 O_{\smallkay,2}$. Write $\Delta = 4 d$ with $d\in 2\Z_2$. 
Then the  $O_{\smallkay,2}$-lattice $D(0)$ of rank $2$ with hermitian form defined by 
$$\begin{pmatrix} -d&\l^\s\\ \l & \frac{\kappa - N(\l)}{d}\end{pmatrix}$$
is unimodular, anisotropic, and has $(x,x) \in 2\Z_2$ for all $x\in D(0)$. 
The orbit representatives in case (b) are $D(0)\oplus H(0)^{k-1}$ and 
$\diag(1,-\kappa)\oplus H(0)^{k-1}$.\hfb
(ii) By analogy with the case of symmetric bilinear forms, \cite{serre.course}, we will call a 
self-dual (unimodular) hermitian lattice type II if the values 
$(x,x)$ for $x\in L$ are all even and type I otherwise. \hfb
(iii) By Proposition~10.4 of 
\cite{jacobowitz},  for a quadratic extension $E/F$ of local fields of characteristic zero and residue 
characteristic $2$, the isometry class of a unimodular hermitian lattice $L$ is determined by 
the rank,  the determinant $\det(L) \in F^\times/N(E^\times)$ and the $O_F$-ideal 
$\mu(L)$ generated by the values $(x,x)$ for $x\in L$. This ideal has the form
$\Cal P_F^{r}$ with $r\ge 0$ and contains the ideal $\tr(O_E) \supset 2 O_F$. 
Thus,  the number of isometry classes can grow with $\ord_F(2)$. 
\end{remark} 

\begin{cor}\label{global.genera} For a relevant hermitian space $V$ in $\Cal R_{(n-r,r)}(\kay)$, the number of $G_1^V$-genera 
of self-dual lattices in $V$ is $2$ or $1$ 
depending on whether or not one of the exceptional cases (a) and (b) occurs at the prime $p=2$.
\end{cor}

The following fact will also useful. 

\begin{lem}\label{genus.lemma}  Let 
$$G^V(\A_f)^0 = \{\ g\in G^V(\A_f)\mid \nu(g)\in \widehat{\Z}^\times\ \}.$$
Then the orbit of a self-dual lattice under the action of $G^V(\A_f)^0$ is the same as the orbit under 
$G^V_1(\A_f)$. 
\end{lem} 
\begin{proof} The only issue is to show that, in the case where there are two $G_1^V(\Q_2)$-orbits of self-dual lattices 
in $V_2$, the action of $G^V(\Q_2)^0$ preserves these orbits. But if $g\in G^V(\Q_2)^0$, and $L$ is a self-dual lattice in $V_2$, 
then the ideals $\mu(L)$ and $\mu(gL)$ in $\Z_2$ generated by the values $(x,x)$ for $x\in L$ (resp. $gL$) are the 
same. By  Proposition~10.4 of \cite{jacobowitz}, cf. Remark~\ref{genera.remark}, $L$ and $gL$ are isometric. 
\end{proof}

\begin{defn}
Let $\Cal R_{(n-r,r)}(\kay)^\sh$ be the set of isomorphism classes of pairs $V^\sh:=(V, \LLL)$ where $V$ is a relevant hermitian space
and $\LLL$ is a $G_1^V$-genus of self-dual hermitian lattices in $V$.  
\end{defn}
The $G_1^V$-genus is determined by its type, as defined 
in (ii) of the preceding remark.
Of course, if $n$ is odd or if $n$ is even and $2$ is unramified in $\kay$, all self-dual lattices are of 
type I and 
the natural map from $\Cal R_{(n-r,r)}(\kay)^\sh$ to $\Cal R_{(n-r,r)}(\kay)$ is a bijection. 

We write 
\begin{align*}
\M(n-r,r)[\frac12] &= \M(n-r,r) \times_{\Spec \OK} \Spec \OK[\frac12]\\
\noalign{\noindent and}
\M[\frac12] &= \M \times_{\Spec \OK} \Spec \OK[\frac12].
\end{align*}

\begin{prop}\label{disjsum.genera}  
(i) There is a natural disjoint decomposition of algebraic stacks
$$\M(n-r,r)[\frac12]= \coprod_{V^\sh\in \Cal R_{(n-r, r)}(\smallkay)^\sh/\text{\rm str.sim.}} 
\M(n-r,r)[\frac12]^{V^\sh}.$$
(ii) 
There is a natural disjoint decomposition of algebraic stacks
$$\M[\frac12]= \coprod_{V^\sh\in \Cal R_{(n-r, r)}(\smallkay)^\sh} 
\MHr[\frac12]^{V^\sh}.$$
For $n$ even, this decomposition is obtained by base change from that in (i).
\end{prop}
\begin{proof} Away from characteristic $2$, the Tate module $T_2(A)$ is a unimodular $O_{\smallkay,2}$-lattice 
where, as explained in the proof of Proposition~\ref{disjsum},  the hermitian form is well defined up to scaling 
by an element of $\Z_2^\times$. The norm, $\mu(T_2(A))$, of this lattice is thus well defined and hence 
so is the type of the genus of self-dual lattices in $V$ determined by $T^p(A)$. 
\end{proof}

We will frequently abuse notation and write $\M(n-r,r)^{V^\sh}$ and $\M^{V^\sh}$ instead of 
$\M(n-r,r)[\frac12]^{V^\sh}$ and $\M[\frac12]^{V^\sh}$ when working away from characteristic $2$. 

\subsection{} 

We next obtain some information about the support of the 
cycle $\ZZ(T)$.
\begin{lem}\label{generic.fiber} If $T\in \Herm_m(\OK)_{>0}$ for $m>n-r$, then
$\ZZ(T)_\Q=\emptyset$. 
\end{lem}
\begin{proof}
Obviously, it suffices to prove that $\ZZ(T)(\C)=\emptyset$, and this is part of Corollary~\ref{ZTcomplex}.
\end{proof}
We now assume that  the matrix $T$ is nonsingular
of rank $n$.  In this case more can be said.

\begin{lem}
Let $0<r<n$. Let  $T\in \Herm_n(\OK)_{>0}$. Then $\supp(\ZZ(T))$ is contained in the union over finitely many inert 
or ramified $p$ of the supersingular locus of $\M_p$.
\end{lem} 
\begin{proof}
Since $T$ is nonsingular, any point $(A, \iota, \l; E, \iota_0, \l_0; \xx)$ of $\ZZ(T)$ defines an $\OK$-isogeny $E^n\to A$. 
Over $\C$ such an isogeny cannot exist since the representations of $\kay\tt\C$ on $(\Lie E)^n$ and on $\Lie A$ are 
non-isomorphic by the determinant condition. Hence $\ZZ(T)$ is concentrated in the fibers at $p$ for finitely many $p$. 
Furthermore, if $p$ is inert or ramified, then $E$ is supersingular and hence so is $A$.\hfb
It remains to exclude the case of split primes $p$. Let $M$, resp. $M_0$, be the Dieudonn\'e module of $A$, resp. $E^n$, and let $N$, resp. $N_0$ be the corresponding rational Dieudonn\'e module. The action of 
$\OK\otimes\mathbb Z_p\simeq \mathbb Z_p\oplus \mathbb Z_p$ decomposes these Dieudonn\'e modules into the 
direct sum of Dieudonn\'e modules $M=M^1\oplus M^2$, and  
$M_0=M_0^1\oplus M_0^2$, and similarly for the rational Dieudonn\'e modules and the Lie algebras. An $\OK$-linear 
isogeny $E^n\to A$ induces  isomorphisms of rational Dieudonn\'e modules $N\simeq N_0$, and $N^i\simeq N_0^i$ for $i=1, 2$. 
Now ${\rm ord\, det}( V\vert N^i)=\dim M^i/VM^i=\dim\, (\Lie A)^i$, and ${\rm ord\, det}( V\vert N_0^i)=\dim M_0^i/VM_0^i=\dim\, (\Lie E^n)^i$.  
Hence $\dim\, (\Lie A)^i=\dim\,(\Lie E^n)^i$ for $i=1, 2$. Since $(\Lie E^n)^i=(0)$ for one $i$, this contradicts the nontriviality 
hypothesis made on the signature.
\end{proof}
In the case of signature $(n-1, 1)$ one can go a bit further. 
\begin{prop}\label{support.ZT} Let $r=1$ and suppose that $T\in \Herm_n(\OK)_{>0}$.
Let $V_T$ denote the positive-definite hermitian 
space $V_T=\kay^n$ with hermitian form given by $T$.  
Let $\Diff_0(T)$ be the set of primes $p$ that are inert in $\kay$  
for which $\ord_p(\det(T))$ is odd. \hfb
(i) If $|\Diff_0(T)|>1$, then $\ZZ(T)$ is empty. \hfb
(ii) If $\Diff_0(T)=\{p\}$, then 
$$\supp(\ZZ(T)) \subset  \MH^{V,\text{\rm ss}}_p, $$
where $V\in \Cal R_{(n-1,1)}(\kay)$ is the unique relevant hermitian space with $\inv_\ell(V) = \inv_\ell(V_T)$
for all finite primes $\ell\ne p$.\hfb
{\rm Here $\MH^{V,\text{\rm ss}}_p$    
denotes 
the supersingular locus in the fiber of $\MH^V$ at $p$. }  
\hfb
(iii) If $\Diff_0(T)$ is empty, then, for each $p\mid \Delta$, there is a unique
relevant hermitian space $V^{(p)}\in \Cal R_{(n-1, 1)}(\kay)$ for which $\inv_\ell(V^{(p)})=\inv_\ell(V_T)$ 
for all $\ell\ne p$. Then
$$\supp(\ZZ(T))\ \subset \ \bigcup_{p\mid \Delta} \MH^{V^{(p)},\text{\rm ss}}_p. 
$$
\end{prop}
\begin{proof}
Let $p\in\Diff_0(T)$, and  let $(A, \iota, \l; E, \iota_0, \l_0; \xx)\in \ZZ(T)(\bar\F_p)$. Let $V'=V'(A,E)$.  
Then $V'$ can be identified with $V_T$, and the natural map 
$$
V'\tt\A_f^p\to  V'_{\A_f^p}$$
is an isomorphism. On the other hand, $V'_{\A_f^p}$ can be identified with $\Hom_{\smallkay\tt\A_f^p}(T^p(E)^0, T^p(A)^0)$  (cf. section 
\ref{compWeil}), and therefore admits a self-dual 
$\hat{\Z}^p$-lattice. Hence  $\ord_\ell(\det(T))$ is even for all inert $\ell\neq p$. This proves (i), and (ii) and (iii) follow from this and the previous lemma. 
\end{proof}

\section{Uniformization of the complex points}\label{sectioncxunif}

In this section, we fix an embedding $\tau$ of $\kay$ into $\C$ and discuss the complex points of our moduli 
spaces  $\M(n-r,r)$ and of the special cycles.   

\subsection{} 
Let $V, (\ , \ )$ be a hermitian vector space over $\kay$ of signature $(n-r,r)$. Let $G=G^V=\text{\rm GU}(V)$ 
be the similitude group of $V$,  viewed as a reductive algebraic group 
over $\Q$. Thus, for any $\Q$-algebra $R$, 
$$G(R) = \{\ g\in \End_\smallkay(V)\tt_\Q R\mid gg^* =\nu(g) \in R^\times\ \},$$
where $*$ is the involution on $\End(V)$ determined by $(\ ,\ )$. 
Let $G_1 = \UU(V)$.

We note that
$$\nu(G(\R)) = \begin{cases} 
\R^\times&\text{for signature $(r,r)$,}\\
\R^\times_+&\text{for signature $(n-r,r)$, with $r\ne n-r$.}
\end{cases}
$$
This fact distinguishes the case $n-r\neq r$ from the case $n-r=r$. 

Let us first assume that $r\ne n-r$.  Let $\hh:\C\rightarrow \End_\smallkay(V)\tt\R$ be an $\R$-algebra homomorphism such that $\hh(z)^* = \hh(\bar z)$, 
and such that 
the form $\gs{\hh(i) x}{y}$ is symmetric and positive definite on $V\tt_\Q\R$. 
Note that $\hh(z) \hh(z)^* = \nu(\hh(z)) = |z|^2$, so that $\hh: \C^\times \rightarrow G(\R)$. 
Let
$J_0$ be the complex structure on $V_\R = V\tt_\Q\R$ given by multiplication 
by $\sqrt{\Delta}\tt |\Delta|^{-\frac12}$. 
Here $\sqrt{\Delta}\in \kay$ is the element for which $\tau(\sqrt{\Delta})$ has positive imaginary part. 
For a homomorphism $\hh$ as above, let $U$ be the 
subspace of $V_\R$ on which $\hh(i) = -J_0$. Then $U$ is a complex $r$-plane in $(V_\R, J_0)$
on which the hermitian form $(\ ,\ )$ is negative definite, and every such $r$-plane arises in this way.  
Note that $\hh(i) = J_0$ on the positive $n-r$-plane $U^\perp$. 
Thus, we can identify the space of homomorphisms $\hh:\mathbb S\rightarrow G(\R)$ of the above type
with the space $D = D(V)$  of negative $r$-planes in $(V_\R,J_0)$. 
The subspace of $V_\C$ on which $\hh(z)$ acts by $z$ is 
isomorphic, as a representation of $\kay$,  to 
$(n-r)\cdot \tau + r\cdot \bar\tau$.

 In the case of signature $(r,r)$, the form $\gs{\hh(i) x}{y}$ is only required to be symmetric 
 and definite on $V\tt_\Q\R$, and we define $\sgn(\hh)=\pm1$ so that $\sgn(\hh) \,\gs{\hh(i) x}{x}\ge 0$. 
 On the complex $r$-plane $U$ in $V(\R)$ on which $\hh(i) = -J_0$, the hermitian form 
 $\sgn(\hh)\,(\ ,\ )$ is negative definite, and we may identify the space of homomorphisms  
 $\hh:\mathbb S\rightarrow G(\R)$ of the above type with $D = D^+\cup D^-$, where $D^\e$ is the 
 space of $r$-planes that are negative for $\e\,(\ ,\ )$.

\subsection{}

Next, we describe the complex points of $\M(n-r,r)$.  Let 
$\Cal L_{(n-r,r)}(\kay)$ be the set of isomorphism classes of self-dual hermitian $\OK$-lattices of signature $(n-r,r)$. 
There is a natural surjective map $L\mapsto L\tt\Q$ from $\Cal L_{(n-r,r)}(\kay) $ to $ \Cal R_{(n-r,r)}(\kay)$. 
We fix representatives $V$ for 
$\Cal R_{(n-r, r)}(\kay)$ and a compatible set of representatives $L$ for $\Cal L_{(n-r,r)}(\kay)$. 
We will write $D(L) = D(L\tt\Q)$ for the corresponding space of negative $r$-planes. 
We also fix the trivialization $\A_f(1)\isoarrow \A_f $ 
given by the inverse of the exponential map $\Z/n\Z \isoarrow \mu_n(\C)$, $a\mapsto e(a/n)$, for $e(x) = \exp(2\pi i x)$. 

Suppose that $(A,\iota,\l)$ is an element of $\M(n-r,r)(\C)$, and  
let $H = H_1(A,\Z)$. Then, via $\iota$, $H$ is an $\OK$-lattice of rank $n$ 
with an 
alternating form $\gs{}{}_\l$ determined by the polarization $\l$.   
Since the adjoint with respect to $\gs{}{}_\l$ is given by $\iota(\aa)^*= \iota(\aa^\s)$, there is a
$\OK$-valued hermitian form $(\ ,\ )=(\ ,\ )_{\l,\smallkay}$ on $H$
such that $\gs{x}{y}_\l = \tr((x,y)/\sqrt{\Delta})$, 
as in (\ref{recover}). Since the polarization $\l$ is principal, the $\OK$-lattice $H$ is self dual with respect to 
$(\ ,\ )_{\l,\smallkay}$. By the signature condition and the canonical  isomorphism $H_\R = H\tt_\Z\R= \text{\rm Lie}(A)$, 
the hermitian lattice $H$ has signature $(n-r,r)$. 
Choose an isomorphism $j:H \rightarrow L$, where $L$ is one of our fixed representatives.  Via $j$ the complex structure on $H_\R$ corresponds to 
$\hh_z(i)$ for some $z\in D(L)$.  Now we eliminate the choice of $j$;  since any two choices of $j$ differ 
by an element of $\Gamma_L$,  the group of isometries of $L$, we obtain the following result.

\begin{prop} \label{C-points} There is an isomorphism of orbifolds
$$
\M(n-r,r)(\C) \isoarrow \coprod_L \big[\,\Gamma_L\back D(L)\,\big],
$$
where $L$ runs over $\Cal L_{(n-r, r)}(\kay)$.  
\end{prop}

Consider the special case $n=1$ and $r=0$. A fractional ideal $\frak a$ defines a self-dual hermitian lattice $L_{0,\frak a}$ 
in the 
space $V_{0,\frak a}= \kay$ with hermitian form $(x,y) = N(\frak a)^{-1}\,x y^\s$. This gives an isomorphism 
\begin{equation}
C(\kay) \isoarrow \Cal L_{(1,0)}(\kay), \qquad [\frak a] \mapsto [L_{0,\frak a}],
\end{equation}
where $C(\kay)$ is the ideal class group of $\kay$ and $[L_{0,\frak a}]$ denotes the isomorphism class of $L_{0,\frak a}$. 
Since $D(L_{0,\frak a})$ consists of a single point and $\Gamma_{L_{0,\frak a}} = \OK^\times$, we obtain
\begin{equation}
\Cal M_0(\C) \isoarrow \coprod_{[\frak a]\in C(\smallkay)} \big[\, \OK^\times\back D(L_{0,\frak a})\,\big] \simeq
\big[\, \OK^\times\back C(\kay)\,\big],
\end{equation}
where $\OK^\times$ acts trivially on $D(L_{0,\frak a})$ and $C(\kay)$. 

There is a slight variant of Proposition \ref{C-points}. Consider the  map $\Cal L_{(n-r,r)}(\kay) \lra  \Cal R_{(n-r,r)}(\kay)^\sh$ 
given by $L \mapsto (L\tt_\Z\Q, [[L]])$. 
By construction,  the fiber in $\Cal L_{(n-r,r)}(\kay)$ over $(V,\LLL)\in \Cal R_{(n-r,r)}(\kay)^\sh$ is in bijection with 
$$G_1^V(\Q)\back G^V_1(\A_f)/K^V_1,$$
where $K^V_1$ is the stabilizer in $G^V_1(\A_f)$ 
of a self-dual lattice in $\LLL$.   

\begin{cor}\label{altC-points} There is an isomorphism of orbifolds
$$\Cal M(n-r,r)(\C) \isoarrow \coprod_{V^\sh} \big[\,G_1^V(\Q)\back \big(\,D(V)\times G^V_1(\A_f)/K^{V^\sh}_1\,\big)\,\big],$$
where $V^\sh$ runs over $\Cal R_{(n-r,r)}(\kay)^\sh$.\qed
\end{cor}

When $n$ is odd, we will usually drop the $\sh$ here.  On the other hand, when there are two $G_1^V$-genera, 
the group $K_1^{V^\sh}$ depends on $V^\sh$, not just on $V$. 

\begin{remark}\label{pi0M0}
In the special case $n=1$ and $r=0$,  the decomposition 
$$\M_0(\C) \isoarrow \coprod_{V_0} \big[\,G_1^{V_0}(\Q)\back \big(\,D(V_0)\times G^{V_0}_1(\A_f)/K^{V_0}_1\,\big)\,\big]$$ 
in Corollary~\ref{altC-points}
corresponds to the decomposition of $C(\kay)$ according to genera.  More precisely, the isomorphism class of the 
hermitian space $V_{0,\frak a} = L_{0,\frak a}\tt\Q$ is determined by the values 
$\chi_p(\det V_{0,\frak a}) = \chi_p(N(\frak a)) = (N(\frak a),\Delta)_p$ as $p$ runs over the primes dividing $\Delta$. 
But these are just the values of the genus characters $\xi_p([\frak a]) = (N(\frak a),\Delta)_p$. Thus the fibers of the 
map $C(\kay)\simeq\Cal L_{(1,0)}(\kay) \rightarrow \Cal R_{(1,0)}(\kay)$ are the genera, i.e., the cosets of the subgroup $C(\kay)^2$. 
On the other hand, the inclusion of $\kay^1\back \kay^1_{\A_f}/\widehat{O}_{\smallkay}^1$ 
into $\kay^\times\back \kay^{\times}_{\A_f}/\widehat{O}_{\smallkay}^\times
=C(\kay)$ identifies 
$$G^{V_0}_1(\Q)\back G^{V_0}_1(\A_f)/K^{V_0}_1 = \kay^1\back \kay^1_{\A_f}/\widehat{O}_{\smallkay}^1$$
with $C(\kay)^2$. 
\end{remark}

\begin{remark} For a relevant hermitian space $\tilde V$ in $\Cal R_{(n-r,r)}(\kay)$,  the set of complex points
$\M^{\tilde V}(\C)$ of the summand $\M^{\tilde V}$ of $\M$,  in the sense of the 
decomposition in Proposition \ref{disjsum}, corresponds to the subset of the product $\M(n-r,r)(\C)\times \M_0(\C)$ 
indexed by pairs $(V^\sh,V_0)$ for which $\Hom_{\smallkay}(V_0,V) \simeq \tilde V$.  
\end{remark}

\subsection{}  Next we consider the complex points of the special cycles. 
Suppose that $(A,\iota,\l,E,\iota_0,\l_0;\xx)$ is a point of $\ZZ(T)(\C)$, and let $H= H_1(A(\C),\Z)$ and $H_0= H_1(E(\C),\Z)$. 
There are   
isomorphisms $j: H \isoarrow L$ and $j_0: H_0 \isoarrow L_0$, for relevant hermitian lattices $L$ and $L_0$ in our fixed sets of representatives, 
and we obtain data $z\in D(L)$ and, setting  $ \widetilde{L} := \Hom_{\OK}(L_0,L)$,
$$\tilde \xx =  j\circ \xx \circ j_0^{-1} \in \widetilde{L}^m.$$
Here we slightly abuse notation and write $\xx$ for the collection of homomorphisms from $H_0$ to $H$ 
induced by $\xx$. Note that the lattice $\widetilde{L}$ 
has a hermitian form $\tilde h$ coming from the hermitian forms $h_\l$ and $h_{\l_0}$ on 
$H$ and $H_0$ arising from the polarizations, as above. 
The pair $(z,\tilde \xx)$ satisfies the following incidence relations
\begin{enumerate}
\item[(1)]  $\tilde h(\tilde\xx,\tilde\xx) = T.$
\item[(2)]  $z\in D(L)_{\tilde \xx}$, where, for $\tilde\xx = [\tilde x_1,\dots,\tilde x_m]$, 
$$D(L)_{\tilde \xx} = \{ \, z\in D(L)\mid z\perp \tilde x_i(L_0)\, \text{\rm \ for all $i$}\,\}.$$
\end{enumerate}
Note that condition (2) 
corresponds to the fact that the $\kay$-linear maps $x_i: H_{0,\R}= \Lie E \rightarrow H_\R =\Lie A$ 
are holomorphic.

Let 
$$\text{\rm Inc}_\infty(T;L, L_0)\quad \subset\quad D(L)\times \widetilde{L}^m$$ 
be the set of pairs satisfying conditions (1) and (2). 
The complex uniformization of the special cycles is now given by the following proposition.
\begin{prop}\label{C-uni.ZT-I} Let $T\in \Herm_m(\OK)$. Then there is an isomorphism of orbifolds
$$\ZZ(T)(\C) \simeq \coprod_{[L]\in \Cal L_{(n-r,r)}}\coprod_{[L_0]\in \Cal L_{(1,0)}}
 \big[\, (\OK^\times \times \Gamma_L)\back \text{\rm Inc}_\infty(T;L, L_0)\,\big].$$
\end{prop}
Here $\OK^\times$ acts by scalar multiplication on the $\tilde\xx$-component of an element of $\text{\rm Inc}_\infty(T;L, L_0)$. 

This can be written more explicitly.  For a fractional ideal $\frak a$ 
and a lattice $L$ in $\Cal L_{(n-r,r)}(\kay)$, let 
$L_{\frak a} = \frak a\tt_{\OK} L$, with hermitian form 
\begin{equation}
(x,x)_{\frak a} = \frac{(x,x)}{N(\frak a)}.
\end{equation}
Then, $L_{\frak a}$ is again a self-dual lattice, 
\begin{equation}
L_{\frak a} \simeq \Hom_{\OK}(L_{0,\frak a^{-1}}, L)
\end{equation}
as hermitian lattices, and we have
\begin{equation}
\ZZ(T)(\C) \simeq \coprod_{[\frak a]\in C(\smallkay)}
\coprod_{[L]\in \Cal L_{(n-r,r)}} \bigg[\,(\OK^\times\times\Gamma_L) \back 
\coprod_{\substack{\xx\in L_{\frak a}^m \\ \snass (\xx,\xx)_{\frak a} = T}}
D(L)_{\xx}\ \bigg].
\end{equation}

If  $\tilde \xx\in (\tilde L\tt \R)^m$ with $\tilde h(\tilde \xx, \tilde \xx)= T>0$, then $m\le n-r$ and $D(L)_{\tilde \xx}$ has codimension $mr$ in $D(L)$.
\begin{cor}\label{ZTcomplex}  For $T\in \Herm_m(\OK)_{>0}$, and $m\le n-r$, the cycle $\ZZ(T)(\C)$ has codimension $mr$. If $m>n-r$, 
then $\ZZ(T)(\C)$ is empty. 
\end{cor}

\begin{remark} The cycles occurring here are essentially those studied in the joint work of the first author with John Millson. 
More precisely, let $\pr: D(L) \rightarrow \Gamma_L\back D(L)$ be the projection. Then the cycles
$$Z(T; L_{\frak a},\Gamma_L)=\coprod_{\substack{\xx\in L_{\frak a}^m \\ \snass (\xx,\xx)_{\frak a} = T\\ \snass \mod \Gamma_L}}
\pr(D(L)_{\xx})$$
of codimension $mr$ in $\Gamma_L\back D(L)$ 
are those introduced in \cite{kudlamillson}, \cite{kudlamillsonII}, \cite{KMihes}, and, in the case of signature $(2,1)$, in \cite{kudlaU21}.
\end{remark}

There is an alternative version, where, with the notation as above,  we fix isomorphisms $j_\Q: H_\Q \isoarrow V$ and 
$j_{0,\Q}:H_{0,\Q}\isoarrow V_0$. The lattice $j_\Q(H)$ determines a $G_1^V$-genus $[[L]]$ in $V$, i.e., 
an element $V^\sh$ in $\Cal R_{(n-r,r)}(\kay)^\sh$.  
The lattices $j_\Q(H)$ and $j_{0,\Q}(H_0)$ correspond to cosets $gK_1^{V^\sh}$ in $G_1^V(\A_f)$ and $g_0K_1^{V_0}$ in 
$G^{V_0}_1(\A_f)$, respectively, and so, we obtain a collection
$$(z,gK_1^{V^\sh},g_0K_1^{V_0}, \tilde \xx) 
\in D(V)\times G^V_1(\A_f)/K_1^{V^\sh}\times G^{V_0}_1(\A_f)/K_1^{V_0} \times \widetilde{V}(\Q)^m,$$
where 
$$\widetilde{V} = \Hom_{\smallkay}(V_0,V), \qquad \tilde\xx = j\circ \xx \circ j_0^{-1}.$$
The data $(z,gK_1^{V^\sh},g_0K_1^{V_0}, \tilde \xx)$ satisfies the following incidence relations. 
\begin{enumerate}
\item[(0)]  $\tilde\xx \in (g (\tilde L\otimes\widehat{\Z}) g_0^{-1})^m,$ where
$\tilde L  = \Hom_{\OK}(L_0,L)$.
\item[(1)]  $(\tilde\xx,\tilde\xx) = T.$
\item[(2)]  $z\in D(V)_{\tilde \xx}$, where 
$$D(V)_{\tilde \xx} = \{ \, z\in D(V)\mid z\perp \tilde x_i(V_0)\, \text{\rm \ for all $i$}\,\}.$$ 
\end{enumerate}

For given relevant spaces $V^\sh\in \Cal R_{(n-r,r)}(\kay)^\sh$ and $V_0\in \Cal R_{(1,0)}(\kay)$, let 
$${\rm Inc}_\infty(T; V^\sh, V_0)\  \subset \ D(V)\times G^V_1(\A_f)/K_1^{V^\sh}\times G^{V_0}_1(\A_f)/K_1^{V_0} \times \widetilde{V}(\Q)^m,$$
be the subset of collections satisfying conditions (0), (1) and (2). 
\begin{prop}\label{C-uni.ZT-II} Let $T\in \Herm_m(\OK)$. Then there is an isomorphism of orbifolds
$$\ZZ(T)(\C) \simeq \coprod_{V^\sh\in \Cal R_{(n-r,r)}^\sh}\coprod_{V_0^{\phantom{\sh}}\in \Cal R_{(1,0)}^{\phantom{\sh}}}
 \big[\, (G_1^V(\Q)\times G_1^{V_0}(\Q))\back \text{\rm Inc}_\infty(T;V^\sh, V_0)\,\big].$$\qed
\end{prop}

\section{Relation to Shimura varieties}\label{sectionrelshim}

In this section, we discuss the relation between our moduli space $\M(n-r,r)$ and Shimura varieties for 
unitary similitude groups. When $n>1$, we assume that $r(n-r)>0$, and we fix an embedding $\tau$ of $\kay$ into $\C$. 

\subsection{} 
We fix a hermitian vector space $V$ over $\kay$ of signature $(n-r,r)$ and write $G= G^V= \GU(V)$. 
For an open compact subgroup $K\subset G(\A_f)$, there is a Shimura variety $\Sh^V_K$ 
over\footnote{In the case where $n$ is even and $r=n-r$, the reflex field is $\Q$, and we take 
$\Sh^V_K$ to be the base change to $\smallkay$ of the usual canonical model.  } $\kay$ 
with 
\begin{equation}\label{Shim}
\Sh_K^V(\C)\simeq G(\Q)\backslash\big(D\times G(\A_f)/K\big). 
\end{equation}

\begin{remark}\label{strict-Shimura}
We note that two hermitian spaces $V, (\ , \ )$ and $V', (\ , \ )'$ which are strictly similar, i.e., 
$V\simeq V'$ with $ (\ , \ )'= c\  (\ , \ )$ for  $c\in \Q_+^\times$,  define the same Shimura varieties. 
\end{remark}

\subsection{}  

These Shimura varieties are related to the moduli stacks of an `up-to-isogeny' 
moduli problem  
\cite{Deligne-Bourb, kottwitzJAMS}.
For a hermitian space $V$ over $\kay$ with signature $(n-r, r)$
and a compact open subgroup $K$ in $G^V(\A_f)$, we define a groupoid $\SSh_K^V$ fibered over the category of locally noetherian 
$\kay$-schemes. For a locally noetherian $\kay$-scheme $S$,
the objects of $\SSh_K^V(S)$ are collections $(A,\iota, \l, \bar\eta)$, where
\begin{enumerate}
\item{} $A$ is an abelian scheme over $S$, up to  isogeny, 
with an action 
of $\kay$, $\iota:\kay\rightarrow\End^0(A)$, 
\item{} $\l$ is a polarization\footnote{
By this we mean that $\l$ is an isomorphism in the isogeny category, and that for some $a\in \Q^\times_+$, 
 $a\l$ is a polarization in the sense of Deligne,  \cite{Deligne-Bourb}, p.145.
Specifically, for an abelian variety $A$ up to  isogeny over a field $k$, a polarization is an element of $\text{\rm NS}(A)\tt\Q$ such that 
a positive multiple is defined by a projective embedding. Note that $\text{\rm NS}(A)\tt\Q$ 
can be identified with the symmetric elements in $\Hom(A,A^\vee)\tt\Q$. This identification 
is being made here. 
},
\item{} $\bar\eta$ is a $K$-level structure, i.e., a $K$-orbit of $\kay\tt\A_f$-linear isomorphisms
$$\eta: T(A)^0 \rightarrow V(\A_f)$$
such that the polarization form on $T(A)^0$ and the symplectic form $\gs{}{}$ on $V(\A_f)$ 
coincide up to a scalar in $\A_f^\times$. 
\end{enumerate}

In addition, the action $\iota$ is supposed to be compatible with the polarization, i.e., for $\aa\in \kay$, 
$\iota(\aa)^* = \iota(\aa^\s)$, where $*$ is the Rosati involution on $\End^0(A)$ determined 
by $\l$. Finally,  the action of $\kay$ is supposed to satisfy the determinant condition of type $(n-r,r)$ (Kottwitz condition):
\begin{equation}\label{detcond}
\det(T-\iota(\aa)\mid \Lie(A)) = (T-\ph(\aa))^{n-r}(T-\ph(\aa^\s))^{r} \in \Cal O_S[T],
\end{equation}
where $\ph:\kay\rightarrow \Cal O_S$ is the structure homomorphism.  

The morphisms between two objects $(A,\iota, \l, \bar\eta)$ and $(A',\iota', \l', \bar\eta^{\prime})$ 
in $\SSh^V_K(S)$ are 
given by 
$\kay$-linear  isogenies
$\mu: A\rightarrow A'$ carrying $\bar\eta$ to $\bar\eta^{\prime}$ and carrying $\l$ to a $\Q_{+}^\times$-
multiple of $\l'$. All such morphisms are isomorphisms.

\begin{remark}\label{careful.level}  Let us be more precise about (3). For this we recall the discussion on p.390-91 of \cite{kottwitzJAMS}. 
Suppose that $(A,\iota,\l)$ is an abelian scheme over a connected base $S$, 
up to isogeny, together with a polarization $\l$ and a compatible 
$\kay$-action. The rational Tate module of $A$ is then a smooth $\A_f$-sheaf on $S$. For a fixed geometric point $s$ of $S$,  the 
rational Tate module 
of $A$ is determined by the rational Tate module $H_1(A_s,\A_f)$ viewed as a $\pi_1(S,s)$-module. 
The polarization induces an alternating form on $H_1(A_s,\A_f)$ valued in $\A_f(1) = \widehat{\Z}(1)\tt_\Z\Q$, where
\begin{equation}
\widehat{\Z}(1) = \underset{\scr n }{\varprojlim}\  \mu_n(k(s)),
\end{equation} 
for $k(s)$ the residue field of $s$.
If we fix an isomorphism 
\begin{equation}\label{mutriv}
\widehat{\Z}(1) \isoarrow \widehat{\Z},
\end{equation}
we obtain an $\A_f$-valued alternating form on $H_1(A_s,\A_f)$ and a corresponding $\kay\tt\A_f$-valued 
hermitian form, since the polarization is compatible with the 
$\kay$-action. A change in the isomorphism (\ref{mutriv})
results in a scaling of these forms by an element of $\widehat{\Z}^\times$. A $K$-level structure is a $K$-orbit $\bar\eta$ in the set of $\kay\tt\A_f$-linear isomorphisms
$$\eta: H_1(A_s,\A_f) \isoarrow V(\A_f)$$
which preserve the hermitian forms up to a scalar in $\A_f^\times$. Finally, the group $\pi_1(S,s)$ 
acts on $H_1(A_s,\A_f)$ by similitudes of the alternating form, and it is required that the orbit $\bar\eta$ be 
fixed by this action. As a result, the notion of $K$-level structure is independent of the choice of the 
geometric point $s$ and of the trivialization (\ref{mutriv}). 
\end{remark}

A `$p$-integral version' of the following proposition is proved in \cite{kottwitzJAMS}, and the 
proof transposes easily to the situation considered here, comp.\ also \cite{Deligne-Bourb}. 

\begin{prop} $\SSh^V_K$ is a smooth Deligne-Mumford stack over $\Spec \kay$.
For $K$ sufficiently small, $\SSh^V_K$ is a quasi-projective
scheme over $\kay$, naturally isomorphic to the canonical model of the Shimura variety $\Sh^V_K$.\hfb
 In particular, if $K$ is sufficiently small,  the set of $\C$-points $\SSh^V_K(\C)$ of  $\SSh^V_K$, via the fixed embedding $\tau$ of 
 $\kay$ into $\C$,  is  canonically identified with 
 the double coset space $\Sh^V_K(\C)$ of (\ref{Shim}). 
 When $K$ is not sufficiently small, then $\SSh^V_K(\C)$ is canonically identified with the space
$$\big[\,G(\Q)\back D(V)\times G(\A_f)/K\ \big]$$ 
viewed as an orbifold. 
\end{prop}

\subsection{}

The relation between the Shimura variety $\SSh^V_K$  and our moduli stack $\M(n-r,r)$ is now given by the following result. 
\begin{prop}\label{identshim}
 Let $V^\sh\in \Cal R_{(n-r, r)}(\kay)^\sh$ 
 be a relevant hermitian space.  
Then there is an isomorphism of  stacks over $\kay$,
$$\M(n-r, r)^{V^\sh}\times_{\Spec \OK}\Spec \kay \ \simeq  \SSh^V_{K},$$
where $\M(n-r, r)^{V^\sh}$ is the component of $\M(n-r, r)$ associated to the strict similarity class of $V^\sh$ 
in (i) of Proposition~\ref{disjsum.genera} and $K$ is the stabilizer in $G^V(\A_f)$ of a self-dual lattice 
in the $G^V_1$-genus given by  $V^\sh$. 
\end{prop}
\begin{remark} Note  the stacks on both sides of this isomorphism depend only on the strict similitude class of $V$
and recall that, when $n$ is odd, there is only one such class. 
\end{remark}
\begin{proof} 
For a connected 
locally noetherian scheme $S$ over $\kay$,  let $\xi=(A,\iota,\l)$ be an object of $\M(n-r,r)^{V^\sh}(S)$. 
We view $A$ as an abelian scheme up to isogeny with  
polarization given by $\l$ and we extend $\iota$ to an action of $\kay$. 
In order to complete the definition of the corresponding 
object of $\SSh_K^V(S)$, we need to define the level structure $\bar\eta$. 
Since $\xi$ lies in $\M(n-r,r)^{V^\sh}(S)$, and after choosing a trivialization (\ref{mutriv}) of the roots of unity, 
we have an $\kay\tt\A_f$-linear similitude 
$$j_{\A_f}: T(A)^0 \isoarrow V(\A_f),$$
unique up to an element of $G^V(\A_f)^0$. 
The image $j_{\A_f}(T(A))$ is a self-dual lattice in $V$ in the genus $\LLL$ given by 
$V^\sh$, and so,  
adjusting by an element of $G(\A_f)^0$ if necessary, we can assume that 
$j_{\A_f}(T(A)) = L\tt\widehat{\Z}$. 
Let $\eta = j_{\A_f}$; the $K$-orbit $\bar\eta$ of $\eta$ 
is then uniquely determined, and the collection $(A,\iota,\l,\bar\eta)$ is an 
object in $\SSh_{K}^V(S)$. 

Conversely, if an object
$(B,\iota,\l,\bar\eta)$ of $\SSh_{K}^V(S)$ is given, the 
$\OK\tt\widehat{\Z}$-lattice 
$$(\eta)^{-1}(L\tt\widehat{\Z}) \subset T(B)^0$$
is independent of the choice of $\eta$ in the $K$-orbit. There is a unique abelian scheme 
$A$ over $S$, equipped with a quasi-isogeny with  $B$, 
such that 
$$T(A) = (\eta)^{-1}(L\tt\widehat{\Z}).$$
Moreover, there is a unique  $a\in \Q^\times_+$ such that 
$a\l$ is a principal polarization of $A$. To see this, note that, under $\eta$, 
the $\OK\tt\A_f$-valued hermitian forms on $T(B)$ and $V(\A_f)$ 
coincide up to a scalar in $(\A_f)^\times$. Of course,  this scalar is only 
well defined up to an element of $\widehat{\Z}^\times$. In any case, by passing to $a\l$, 
we can arrange it so that this scalar lies in $\widehat{\Z}^\times$ and hence 
$a\l$ defines a principal polarization of $A$, since $T(A)$ correponds to $L\tt\widehat{\Z}$
and hence is now self-dual. 
The given action of $\kay$ on $B$ defines an action of $\OK$ on $A$, 
since $T(A)$ is an $\OK$-lattice. The Kottwitz condition on $(B,\iota,\l,\bar\eta)$ 
implies the signature condition on $(A,\iota,\l)$. Furthermore, by construction $T(A)^0$ is 
similar to $V\tt\A_f$, so that $V(\xi)$ is strictly similar to $V$, again by the signature condition. 
Thus, we obtain a collection $(A,\iota,a\l)$ of $\M(n-r,r)^{V^\sh}(S)$. 
\end{proof}

\subsection{}

\newcommand{\Rkq}{\text{Res}_{\smallkay/\Q}}

In this section, we review the action of the Galois group $\Gal(\bar\Q/\kay)$ on the connected components 
of $\Sh^V_K$ and of $\M(n-r,r)$. Here we will work with the canonical model over the reflex field $E$, where $E =\Q$, if 
$r=n-r$, and $E=\kay$ otherwise. 

Let $T$ be the torus over $\Q$ given by 
$$T = \begin{cases} T^1\times \G_m&\text{if $n$ is even,}\\
\Rkq(\G_{m,\smallkay})&\text{if $n$ is odd, }
\end{cases}
$$
where $T^1 = \ker(N: \Rkq(\G_{m,\smallkay}) \rightarrow \G_m)$. 
For a fixed hermitian space $V$ of signature $(n-r,r)$ with $G = G^V= \GU(V)$, we define a 
surjective homomorphism
$\nu^o:G \rightarrow T$ by 
$$\nu^o(g) = 
\begin{cases} 
(\frac{\det(g)}{\nu(g)^k},\nu(g))&\text{if $n=2k$ is even,}\\
\nass
\frac{\det(g)}{\nu(g)^k}&\text{if $n=2k+1$ is odd.}
\end{cases}
$$
Note that, in the odd case $N(\nu^o(g)) = \nu(g)$. In particular, $\ker(\nu^o) = \text{\rm SU}(V)$ in both cases.
Then we have
\begin{equation}\label{pizero}
\pi_0(\Sh^V_K) \simeq T(\Q)\back T(\A)/\nu^o(K_\infty\times K) = T(\Q)^0\back T(\A_f)/\nu^o(K),
\end{equation}
where $T(\Q)^0 = T(\Q)\cap T(\R)^0T(\A_f)$ and $K_\infty$ is the centralizer of $\hh$ in $G(\R)$. 
Note that $\nu^o(K_\infty) = T(\R)^0$, the identity component of $T(\R)$. 

Next, we recall the standard description of the action of the Galois group on this set, cf.  \cite{PR.III}, 
for example.
For any $\hh:\mathbb S\rightarrow G(\R)$,  a simple calculation shows that 
$$\nu^o\circ \hh(z) =  \begin{cases} 
(\left(\frac{z}{\bar z}\right)^{k-r},z\bar z)&\text{if $n=2k$ is even,}\\
\nass
\left(\frac{z}{\bar z}\right)^{k-r}\,z&\text{if $n=2k+1$ is odd.}
\end{cases}
$$
Moreover, writing $\mathbb S = \mathbb S_\Q\times_\Q\R$ where $\mathbb S_\Q = \text{\rm R}_{E/\Q}(\G_{m,E})$, 
there is a homomorphism of the form (cf. \cite{PR.III}, 1.c), 
$$\rho= N_{E/\Q}\circ \nu^o\circ \mu_h: \mathbb S_\Q \lra T,$$
given by\footnote{ In \cite{PR.III}, eq. (1.15), the exponent $-1$
in the first factor should be eliminated.}
$$\rho(a) = \begin{cases}
(\left(\frac{a}{\bar a}\right)^{k-r}, a\bar a)&\text{if $n=2k$ is even and $r\ne n-r$,}\\
\nass
(1, a)&\text{if $n=2k$ is even and $r= n-r$,}\\
\nass
\left(\frac{a}{\bar a}\right)^{k-r}\,a&\text{if $n=2k+1$ is odd.}
\end{cases}
$$
Finally, the action of $\s\in \Gal (\bar \Q/E)$ on $\pi_0(\Sh^V_K)$ is given, on the right side of 
(\ref{pizero}), by multiplication by $\rho(x_\s)$, where $x_\s\in \A^\times_E$ is an element 
whose image under the Artin reciprocity map is $\s\vert E^{\text{\rm ab}}$.  As in \cite{PR.III}, we 
normalize this map so that a local uniformizer corresponds to the inverse of the Frobenius.

\vskip .5in

\centerline{\Large\bf Part II: The supersingular locus}

\section{Uniformization of the supersingular locus}\label{sectionpadicunif}

Let $p$ be a prime inert in $\kay$. In this section we review the $p$-adic uniformization of $\Cal M(n-r, r)$ along the supersingular locus of its reduction, 
following the procedure of Chapter 6 of \cite{RZ}.  

Let $\F=  \bar\F_p$ and let $W = W(\F)$ with a fixed embedding $\tau$ of $\kay_p$.

As in \cite{vollaard} and \cite{KRunitary.I}, we fix a supersingular $p$-divisible formal 
group $\X$ over $\F$ of dimension $n$ and height $2n$ with an action $\iota$ of $\OK$
satisfying the determinant condition of type $(n-r,r)$ and a $p$-principal polarization $\l_{\X}$ for which 
the Rosati involution $*$ satisfies $\iota(\aa)^*=\iota(\aa^\s)$. The collection $(\X,\iota,\l_\X)$ 
is unique up to isomorphism. 

Let
$\Cal N = \Cal N(n-r,r)$ be the functor on $\text{\rm Nilp}=\text{\rm Nilp}_W$ whose value on $S\in \text{\rm Nilp}$ 
is the set of isomorphism classes of collections $\xi=(X,\iota, \l_X,\rho_X)$ where 
$X$ is a $p$-divisible group over $S$ with an $\OK$-action satisfying the  
determinant condition of type $(n-r,r)$ and $\l_X$ is a $p$-principal polarization compatible with 
the $\OK$-action.  
Finally, 
$$\rho_X: X\times_W\F \lra  \X\times_\F \bar S$$
is an $\OK$-equivariant quasi-isogeny of height $0$ such that $\rho_X^\vee\circ \l_\X\circ \rho_X$
is a $\Z_p^\times$-multiple of $\l_X$ in $\Hom_{\OK}(X,X^\vee)\tt\Q$. Here $\bar S=S\times_W{\F}$. 
An isomorphism between two collections  $\xi=(X,\iota, \l_X,\rho_X)$ and
 $\xi'=(X',\iota', \l_{X'},\rho_{X'})$ is an $\OK$-linear isomorphism $\alpha: X\to X'$, compatible with  $\rho_X$ and  $\rho_{X'}$ over $\bar S$, such that 
 $\alpha^*(\l_{X'})\in \Z_p^\times\l_X$.

In fact, for our global construction, it will be convenient to choose $(\X,\iota,\l_\X)$ as follows. 
Suppose that $\xi^o=(A^o,\iota^o, \l^o)$ of $\Cal M(n-r, r)(\F)$ lies in the supersingular locus, and 
let  $(\X,\iota,\l_\X)=(X(A^o), \iota^o, \l_{X(A^o)})$ be the corresponding $p$-divisible group with its additional structure.  

We also fix a trivialization   of the prime-to-$p$ roots of unity over $\F$,
\begin{equation}\label{trivunits}
\widehat{\Z}^p(1)\simeq \widehat{\Z}^p.
\end{equation}
Then  the Weil pairing on $T^p(A^o)^0$ takes values in $\A_f^p(1)\simeq \A_f^p$, and there is a unique relevant space 
$V\in \Cal R_{(n-r, r)}(\kay)$ such that $T^p(A^o)^0$ is isometric to $V\tt\A_f^p$. 
If $p\ne 2$,  there is a unique 
$V^\sh= (V,\LLL)\in \Cal R_{(n-r, r)}(\kay)^\sh$ such that $T^p(A^o)^0$ is isometric to $V\tt\A_f^p$, 
and the Tate module $T_2(A^o)$ determines the type of the genus $\LLL$ of self-dual lattices,  cf.\ the proof of 
Proposition \ref{disjsum.genera}. 

For the rest of this section, we will discuss only the case $p\ne 2$ and leave it to the reader to make the slight modifications 
needed in the $p=2$ case. 

\begin{lem} For each $V^\sh= (V,\LLL)\in \Cal R_{(n-r, r)}(\kay)^\sh$, the supersingular locus in $\Cal M(n-r, r)^{V^\sh}(\F)$ is non-empty.
\end{lem}
\begin{proof} Fix a supersingular elliptic curve $E$ over $\F$ and an embedding $\OK \hookrightarrow \End(E) = O_B$, where $B$ is the 
quaternion algebra over $\Q$ ramified at $\infty$ and $p$. We assume that $\OK$ acts on $\Lie(E)$  via the 
standard map of $\OK$ to $\F$ and we give $E$ its canonical principal polarization.  The corresponding Rosati involution on $B$ is the 
main involution $b\mapsto b'$. 
Let $A = E^n$ so that $\End(A) = M_n(O_B)$ and let $\l_0$ be the product polarization. The Rosati involution is then $b\mapsto {}^tb'$. 
Define $\iota:\OK \lra \End(A)$ by 
$$\iota(a) = \diag(\underset{n-r}{\underbrace{a,\dots,a}},\underset{r}{\underbrace{\bar a,\dots, \bar a}}).$$
Then $(A,\iota,\l_0)$ gives a point in the supersingular locus of $\M(n-r,r)(\F)$. 
For our fixed trivialization (\ref{trivunits}), the procedure above yields a relevant hermitian space $V(\l_0)\in \Cal R_{(n-r, r)}(\kay)$.  
We next show that, by modifying the choice of the polarization and  performing an isogeny, we can obtain any given $V^\sh = (V,\LLL)$. 
Suppose that a relevant space $V^\sh = (V,\LLL)$ is given.
Any polarization $\l$  
on $A$ can be written as $\l=\l_0\circ \b$, for $\b\in M_n(O_B)\cap \GL_n(B)$ with ${}^t\b'=\b>0$,  and the corresponding Rosati involution 
is given by $b\mapsto \b {}^tb' \b^{-1}$. This will have the required compatibility with $\iota$ precisely when $\b$ centralizes 
$\iota(\OK)$.   By (\ref{Weil-form}), the hermitian forms $h_\l$ and $h_{\l_0}$ defined on $T^p(A)^0$ by $\l$ and $\l_0$ 
are related by 
\begin{equation}\label{shift.pol}
h_\l(x,y) = h_{\l_0}(x,\b y).
\end{equation}
Choose a prime $q\nmid \Delta$ such that $(\Delta,q)_\ell = \inv_\ell(V)\,\inv_\ell(V(\l_0))$ for all $\ell\mid \Delta$ 
and such that $q\equiv 1\!\!\mod 8$, if $2\nmid \Delta$.
In particular, since $V$ and $V(\l_0)$ both have signature $(n-r,r)$ and both contain a unimodular lattice, their invariants agree at infinity 
and at all unramified finite primes, and hence
\begin{equation}\label{genus.prod}
\prod_{\ell\mid \Delta} (\Delta,q)_\ell =1. 
\end{equation}
It follows that $(\Delta,q)_q=1$ as well so that $q$ is split in $\kay$.  Let $i(q) = \diag(q,1,\dots,1)\in \End(A)$, 
and let $\l = \l_0\circ i(q)$ be the resulting (non-principal) polarization. Let $V(\l)$ be the hermitian space of signature $(n-r,r)$
with an isometry $V(\l)\tt_{\Q}\A_f^p \simeq T^p(A)^0$ where the hermitian form on $T^p(A)^0$ is defined by $\l$. 
By (\ref{shift.pol}), the invariants of $V(\l)$ satisfy $\inv_\ell(V(\l)) = (\Delta,q)_\ell \,\inv_\ell(V(\l_0))=\inv_\ell(V)$ for all $\ell\ne p$ and hence
$\inv_p(V(\l)) = \inv_p(V)$ as well. 
Fixing an isometry $V\simeq V(\l)$, we obtain an isometry 
$\psi: V\tt_{\Q}\A_f^p \simeq T^p(A)^0$.  We can choose an abelian variety $B$ over $\F$, isogenous to $A$ by a prime to $p$ isogeny such that 
$T^p(B) \subset T^p(A)^0$ is the image of $L\tt_{\Z}\hat\Z^p$ under $\psi$. The collection $(B,\iota,\l)$ then defines an element of 
the supersingular locus of 
$\M(n-r,r)^{V^\sh}(\F)$. 
\end{proof}

For a given $V^\sh= (V,\LLL)\in \Cal R_{(n-r, r)}(\kay)^\sh$, we fix a base point $\xi^o=(A^o,\iota^o, \l^o)$ lying in the supersingular 
locus of $\Cal M(n-r, r)^{V^\sh}(\F)$.

We fix a self-dual lattice $L\in V$ in the given $G_1^V$-genus and  an isomorphism 
\begin{equation}\label{trivtate}
\eta^o:T^p(A^o)^0\isoarrow V(\A_f^p)
\end{equation}
such  that 
(i) $\eta^o$ is an isometry and (ii) $\eta^o(T^p(A^o)) = L\tt \widehat{\Z}^p$. 
Let $K^p$ be the stabilizer of $L$ in $G(\A_f^p)$, and note that $K^p$ is a subgroup of 
$$G(\A_f^p)^0 = \{\ g\in G(\A_f^p)\mid \nu(g) \in (\widehat{\Z}^p)^\times\ \}.$$

We choose a lift $\widetilde\X$ of $\X$ to $W$ and let $\widetilde A^o$ be the corresponding lift of $A^o$. 
Then there is a canonical isomorphism
\begin{equation}\label{liftlevel}
\widetilde{\eta}^o:T^p(\widetilde{A}^o)^0 \lra T^p(A^o)^0 \overset{\eta^o}{\lra} V(\A_f^p).
\end{equation}
Using these objects, we can define a morphism of functors on $\text{\rm Nilp}_W$
\begin{equation}
\Theta: \Cal N\times G(\A_f^p)^0 \ \lra \ \Cal M(n-r,r)
\end{equation}
as follows. 
For $S\in \text{\rm Nilp}_W$, let $\widetilde{A}^o_S = \widetilde{A}^o\times_W S$. 
Note that, over the special fiber $\bar S = S\times_W\F$, there is a canonical isomorphism
$$\widetilde{A}^o_S \times_W\F = A^o\times_\F \bar S.$$
Thus there is an $\OK$-action 
$$\iota^o_{\bar S}=\iota^o:\OK\lra  \End_{\bar S}(A^o\times \bar S)= \End(A^o),$$
and a polarization
$$\l^o_{\bar S} = \l^o\times 1_{\bar S}: A^o\times \bar S \lra {A^o}^\vee\times \bar S.$$
By `rigidity', there are unique extensions of these to $\widetilde{A}^o_S$, i.e., 
there is an $\OK$-action by quasi-isogenies, 
$$\widetilde{\iota}^o_{S}:\OK\lra  \End_{S}^0(\widetilde{A}^o_S)= \End^0(A^o),$$
and a (quasi-) polarization
$$\widetilde{\l}^o_{S}: \widetilde{A}^o_S \lra (\widetilde{A}^o_S)^\vee.$$
Finally, there is an isomorphism
$$\tilde\eta^o_S:T^p(\widetilde{A}^o_S)^0 \lra T^p(A^o)^0 \overset{\eta^o}{\lra} V(\A_f^p)$$
derived from (\ref{liftlevel}).   Note that $(\widetilde{A}^o_S,\widetilde{\iota}^o_S, \widetilde{\l}^o_S)$ need not be an element of 
$\M(n-r,r)(S)$, since $\widetilde{\iota}^o_S$ (resp. $\widetilde{\l}^o_S$) is only a quasi-action (resp. quasi-polarization).
\begin{remark}\label{sympliso}
 The choice of trivialization (\ref{trivunits})   of the prime-to-$p$ roots of unity over $\F$ made above gives a canonical choice of 
such a trivialization at all geometric points of any scheme $S\in \text{\rm Nilp}_W$.  
The isomorphism $\tilde\eta^o_S$ is always an isometry for this choice. 
\end{remark}
Note that the restriction to $\bar S$ of the $p$-divisible group of $(\widetilde{A}^o_S,\widetilde{\iota}^o_S, \widetilde{\l}^o_S)$ 
is canonically identified with $(\X_{\bar S},\iota,\l_{\X, \bar S})$. 
Also note that all of these constructions are functorial in $S$. 

\begin{prop}\label{RZkey}
For a given $S\in \text{\rm Nilp}_W$,  let 
$(\widetilde{A}^o_S,\widetilde{\iota}^o_S, \widetilde{\l}^o_S, \tilde\eta^o_S)$
be the collection of objects just defined. 
\hfb
For each object $\xi=(X,\iota, \l_X,\rho_X)$ in  $\Cal N(S)$
and coset $gK^p\in G(\A_f^p)^0/K^p$, 
there is an object $\Theta(\xi, gK^p)=(A,\iota, \l)$ of $\Cal M(n-r, r)(S)$ 
and a $\OK$-linear quasi-isogeny $\phi: A \rightarrow \widetilde{A}^o_S$ 
uniquely characterized by the following properties: \hfb
(i) The polarization $\l$ agrees with  $\phi^\vee\circ\widetilde{\l}^o_S\circ\phi$.
\hfb
(ii) Let 
$$\eta= \tilde\eta^o_S\circ \phi_*: T^p(A)^0 \lra V(\A_f^p).$$
(Note that the map $\eta$ is a symplectic isomorphism, by (i) and Remark \ref{sympliso}.) 
Then
$$\eta(T^p(A)) = g\cdot ( L\tt\widehat{\Z}^p).$$\hfb
(iii) Let $(X(A),\iota)$ be the  $p$-divisible group of $(A,\iota)$ with $\OK$-action. Then 
there is an isomorphism
$$i:(X(A),\iota)\isoarrow (X,\iota)$$ 
such that 
the quasi-isogeny $\phi$ induces $\rho_X$ over $\bar S$, i.e., the diagram
$$
\begin{matrix}
(X(A),\iota)\times_W\F &\isoarrow& (X,\iota)\times_W\F \\
\nass
{}&\scr\phi_*\searrow\phantom{\scr\phi_*}&\phantom{\scr\rho_X}\downarrow\scr\rho_X\\
\nass
{}&{}&(\X_{\bar S},\iota)
\end{matrix}
$$
commutes. 
The maps $\lambda_{X(A)}$ and $i^*(\lambda_X)$ agree up to a factor in $\Z_p^\times$. 

The  construction of $A, \phi$ is functorial in $S$ in the obvious sense. 
\end{prop}
\begin{proof} A detailed proof can be found in \cite{RZ}, \S 6.  The point is that we have the lattice
$$(\tilde\eta^o_S)^{-1}(gL\tt\widehat{\Z}^p)\quad \subset T^p(\widetilde{A}^o_S)^0,$$
while $\rho_X$ determines a $p$-divisible group  in the isogeny class of $X(\widetilde{A}^o_S)$, the $p$-divisible 
group of $\widetilde{A}^o_S$. Together these two determine a unique abelian scheme $A$ over $S$ 
with a quasi-isogeny to $\widetilde{A}^o_S$. The quasi-polarization and quasi-$\OK$-action on 
$\widetilde{A}^o_S$ determine a principal polarization and $\OK$-action on $A$, as required. 
\end{proof}

Let $I(\Q)=I^V(\Q)$ be the group of quasi-isogenies in $\End^0_{\OK}(A^o)$  that preserve the polarization  $\l^o$. Note that $I(\Q)$ 
is the group of $\Q$-points of an algebraic group defined over $\Q$. 
Any $\gamma\in I(\Q)$ induces a quasi-isogeny $\a_p(\gamma)$ of height $0$ of the $p$-divisible group $(\X,\iota,\l_\X)$ of $(A^o,\iota, \l^o)$
and hence acts on $\Cal N$ by sending $\xi=(X,\iota, \l_X,\rho_X)$ to $\a_p(\gamma)\xi =(X,\iota, \l_X,\a_p(\gamma)\circ\rho_X)$.  
The element $\gamma$ also induces an automorphism $\gamma_*$ of $T^p(A^o)^0$ and hence defines an element 
$$\a^p(\gamma) = \eta^{p,o}\circ \gamma_*\circ (\eta^{p,o})^{-1} \in G(\A_f^p),$$
of scale factor $1$.

\begin{lem} Any  $\gamma\in I(\Q)$ induces an isomorphism 
$$\Theta(\a_p(\gamma)\xi, \a^p(\gamma)gK^p) \simeq \Theta(\xi,gK^p),$$
as points in $\Cal M(n-r, r)(S)$. Conversely, any isomorphism 
$$ \Theta(\xi,gK^p)\simeq \Theta(\xi',g'K^p)$$
is induced by a  unique $\gamma\in I(\Q)$. 
\end{lem}
\begin{proof}
The quasi-isogeny $\gamma:A^o \rightarrow A^o$ lifts uniquely to a quasi-isogeny
$\widetilde{\gamma}_S:\widetilde{A}^o_S \rightarrow \widetilde{A}^o_S$
which commutes with the $\OK$-action and preserves the polarization $\widetilde{\l}^o_S$. 
Moreover, the diagram 
\begin{equation}\label{leveldiag}
\begin{matrix}
\tilde\eta^{p,o}_S:& T^{p}(\widetilde{A}^o_S)& \lra& T^{p}(A^o) &\overset{\eta^{p,o}}{\lra}& V(\A_f^p)\\
\nass
{}&\scr(\widetilde{\gamma}_S)_*\,\downarrow\,\phantom{\scr\widetilde{\gamma}_*}&{}&\scr\gamma_*\,\downarrow\,\phantom{\scr\gamma_*}&
{}&\scr\a^p(\gamma)\,\downarrow\,\phantom{\scr\a^p(\gamma)}\\
\nass
\tilde\eta^{p,o}_S:& T^{p}(\widetilde{A}^o_S)& \lra& T^{p}(A^o) &\overset{\eta^{p,o}}{\lra}& V(\A_f^p)
\end{matrix}
\end{equation}
commutes.  Suppose that   $\Theta(\xi,gK^p) = (A,\iota,\l)$ with quasi-isogeny $\phi:A\rightarrow \widetilde{A}^o_S$. 
Then the same collection $(A,\iota,\l)$ with quasi-isogeny $\widetilde{\gamma}_S\circ\phi$ satisfies the conditions of 
Proposition \ref{RZkey} for the pair $(\xi', g'K^p)=(\a_p(\gamma)\xi, \a^p(\gamma)gK^p)$.   Hence $\gamma$ 
induces an isomorphism $\Theta(\xi,gK^p)\simeq \Theta(\xi',g'K^p)$. Conversely, an isomorphism 
$\beta: \Theta(\xi,gK^p)\simeq \Theta(\xi',g'K^p)$ defines a quasi-isogeny $\phi^{-1}\circ\beta\circ\phi': \widetilde{A}^o_S\to \widetilde{A}^o_S$. 
This then defines an element $\gamma\in I(\Q)$ which induces the  isomorphism $\beta$.
\end{proof}

In the present situation, the uniformization theorem, Theorem~6.30 of \cite{RZ}, amounts to the 
following. It reflects the process of forgetting the quasi-isogeny $\phi$ in the construction 
of Proposition \ref{RZkey}. Recall that we fixed a trivialization (\ref{trivunits}) of the prime-to-$p$ roots of unity over $\F$. 
\begin{theo} Let $\widehat{\M}(n-r, r)^{\rm ss}$ denote the formal completion of \hfb $\M(n-r, r)\times_{\Spec \OK}\Spec W(\F)$ 
along its supersingular locus.   For a relevant space $V^\sh= (V,\LLL)$ in $\mathcal R_{(n-r, r)}(\kay)^\sh$, let $\widehat{\Cal M}(n-r, r)^{V^\sh,\text{\rm ss}}$ 
be the open and closed sublocus where the rational Tate module $T^p(A)^0$ is isomorphic to $V\tt\A_f^p$
and the type of the hermitian lattice $T_2(A)$ coincides with the type of the $G^V_1$-genus $\LLL$. 
Then the map $\Theta$ induces an isomorphism
$$\Theta:\big[\, I^V(\Q)\back (\Cal N\times G^V(\A_f^p)^0/K^p)\,\big] \ \isoarrow \widehat{\Cal M}(n-r, r)^{V^\sh,\text{\rm ss}}$$
of formal algebraic stacks over $W$, where $K^p$ is the stabilizer of $L$ in $G^V(\A_f^p)$. \qed
\end{theo}

Note that our space $\Cal N$ is (a slight variant of) the space $\breve{\Cal M}$ in \cite{RZ}
and that Theorem~6.30 of loc.\ cit.\ actually gives a stronger result, not needed here, involving the 
descent of both sides to $\Spf (\OK\tt\Z_p)$. 
\begin{remark}\label{compdec}
Recall from Proposition \ref{disjsum.genera} the disjoint decomposition of $\mathcal M(n-r, r)[\frac12]$ according to strict similarity classes of relevant 
hermitian spaces in $\mathcal R_{(n-r, r)}(\kay)^\sh$. 
Then, it is clear that by taking the disjoint sum of the spaces $\widehat{\Cal M}(n-r, r)^{V^\sh,\text{\rm ss}}$, as $V$ 
runs through the elements in the fixed strict similarity class $[V]$, we obtain the 
formal completion of $\M(n-r, r)^{[V]^\sh}\times\Spec W(\F)$ along its supersingular locus. 
\end{remark}

A special case of the previous result occurs for $\M_0$. In this case the formal scheme 
$\Cal N_0=\Cal N(1, 0)$ is trivial, i.e., is equal to $\Spf\ W$ ({\it canonical lifting}). 

Let us now combine the uniformization theorems for $\M(n-r,r)$ and for $\M_0$ to obtain a uniformization 
theorem for $\M=\M(n-r, r)\times_{\Spec \OK}\M_0$. Recall the disjoint sum decomposition of $\M[\frac12]$ 
according to relevant spaces $\tilde V^\sh\in \Cal R_{(n-r, r)}(\kay)^\sh$, cf.\ Proposition \ref{disjsum.genera}. 
We denote by $\widehat{\M}^{\tilde V^\sh,{\rm ss}}$
 the formal completion of $\M^{\tilde V^\sh}\times_{\Spec \OK} W(\F)$ along its supersingular locus. 
 Now there is a decomposition
\begin{equation}\label{finerss}
\widehat{\M}^{\tilde V^\sh,{\rm ss}} = \coprod_{(V^\sh,V_0)} \widehat{\M}^{(V^\sh,V_0),{\rm ss}},
\end{equation}
 where 
 \begin{equation}
 \widehat{\M}^{(V^\sh,V_0),{\rm ss}}= \widehat{\M(n-r,r)}^{V^\sh,{\rm ss}}\times \widehat{\M_0}^{V_0,{\rm ss}},
 \end{equation}
 and $(V^\sh,V_0)$ runs over pairs in $\Cal R_{(n-r,r)}(\kay)^\sh\times \Cal R_{(1,0)}(\kay)$ with $\Hom_\smallkay(V_0,V) \simeq \tilde V$. 
 Note that
 the indexing on the right hand side of (\ref{finerss}) is determined by our choice of trivialization  (\ref{trivunits}).
\begin{cor}\label{unifprod}
For a pair $(V^\sh,V_0)$,   
 with $\big(\M(n-r, r)^{V^\sh}\times \M_0^{V_0}\big)^{\rm ss}(\F)$ non-empty,  
there is an isomorphism of formal stacks over $\Spf\ W$
$$
\widehat{\M}^{(V^\sh,V_0),{\rm ss}}\simeq \big[\, (I^V(\Q)\times I^{V_0}(\Q))\backslash
\big(\,\Cal N\times \Cal N_0 \times G^V(\A_f^p)^0/K^{V^\sh,p}\times G_0^{V_0}(\A_f^p)^0/K_0^p\,\big)\big].
$$
 \qed
\end{cor} 

\section{Special cycles in the supersingular locus}\label{sectionspecialcyclesinss}

In this section, we utilize the uniformization described in the previous section 
to study the intersection of our special cycles with the supersingular locus. 
To lighten notation, we write $\Cal M(n-r, r)$ for $\Cal M(n-r, r)\times_{\Spec \OK} \Spec W$,  $\Cal M_0$ for  
$\M_0\times_{\Spec \OK} \Spec W$, and $\M$ for $\M\times_{\Spec \OK} \Spec W$.
We denote the supersingular locus by $\Cal M(n-r,r)^{\text{\rm ss}}$,  $\Cal M_0^{\text{\rm ss}}$, and  $\Cal M^{\text{\rm ss}}$, 
respectively.  We continue to treat the case $p\ne 2$, where it is (sometimes) necessary to keep track of the 
type of the $G^V_1$-genera. When $p=2$ is inert, the $G_1^V$-genera do not play a role. 

For $T\in \Herm_m(\kay)$, we write $\ZZ(T)$ for $\ZZ(T)\times_{\Spec \OK} \Spec W$, and 
introduce the fiber product:
\begin{equation}
\begin{matrix}
\widehat{\ZZ}^{(V^\sh,V_0), \ss}(T)&\lra&\widehat{\M}^{(V^\sh,V_0),{\rm ss}}\\
\nass
\downarrow&{}&\downarrow\\
\nass
\ZZ(T)&\lra & \M.
\end{matrix}
\end{equation}
Note that the formal stack $\widehat{\ZZ}^{(V^\sh,V_0), \ss}(T)$ is the formal 
completion of $\ZZ(T)$ along $\ZZ^{(V^\sh,V_0), \ss}(T)= \ZZ(T)\times_{\M}\M^{(V^\sh,V_0),{\rm ss}}$.

Having fixed a base point $(A^o,\iota^o,\l^o; E^o,\iota_0^o,\l_0^o)$ in $\M^{(V^\sh,V_0), \ss}(\F)$
and lifts $\widetilde{A}^o$ and $\widetilde{E}^o$ to $W$, etc., as in the previous section,  we have the uniformization 
of $\widehat{\M}^{(V^\sh,V_0),{\rm ss}}$ given by Corollary \ref{unifprod}.
Our goal is to give a similar uniformization of the formal stack $\widehat{\ZZ}^{(V^\sh,V_0), \ss}(T)$.

To the base point $(A^o,\iota^o,\l^o; E^o,\iota^o_0,\l_0^o)$ in $\big(\Cal M\times\Cal M_0\big)^{ \ss}(\F)$ there are 
associated two hermitian spaces, $\tilde V$ and $\tilde V'$. Here $\tilde V=\Hom_{\smallkay}(V_0, V)$; it may also 
be characterized as  the unique relevant space in $\Cal R_{(n-r, r)}(\kay)$ such that 
\begin{equation}\label{trivV}
\tilde V\tt\A_f^p\simeq\Hom_{\kay\tt\A_f^p}\big(T^p(E^o)^0, T^p(A^o)^0\big),
\end{equation}
i.e., $\tilde V$ is uniquely defined by the condition that $(A^o,\iota^o,\l^o; E^o,\iota^o_0,\l_0^o)\in\Cal M^{\tilde V, \ss}(\F)$.
Note that the isomorphism (\ref{trivV}) is determined by our choices of $\eta$ and $\eta_0$ as in (\ref{trivtate}).\hfb
 The space $\tilde V'$ is given as  
\begin{equation}\label{twisted}
\tilde V' = \Hom_{\OK}^0(E^o,A^o)
\end{equation}
with hermitian form defined by 
\begin{equation}
h'(x,y) = (\l_0^o)^{-1}\circ y^\vee\circ \l^o\circ x.
\end{equation}
Note that the natural action of the group $I^V(\Q)$ (resp.\ $I^{V_0}(\Q)$) on $V'$ 
by post-multiplication (resp. by  pre-multiplication)  preserves this hermitian form. 
For $x\in \tilde V'$, let
\begin{equation}
\und{x} = \eta^{o}\circ x\circ (\eta_0^{o})^{-1}\in \Hom_{\smallkay\tt\A_f^p}(V_0(\A_f^p),V(\A_f^p)),
\end{equation}
and let 
\begin{equation}
\und{\und{x}} \in \Hom_{\OK\tt\Z_p}(\X_0,\X) = \mathbb V
\end{equation}
be the corresponding homomorphisms. Note that there is a natural 
action of the group $I^V(\A_f^p)\times I^{V_0}(\A_f^p)$ (resp. $I^V(\Q_p)\times I^{V_0}(\Q_p)$) 
on the space $\Hom (V_0(\A_f^p), V(\A_f^p))$ (resp.\ $\mathbb V$), and the 
maps $x\mapsto \und{x}$ and $x\mapsto \und{\und{x}}$ are $I^V(\Q)\times I^{V_0}(\Q)$-equivariant. 
\begin{lem}\label{unitaryv'}
a) The hermitian spaces $\tilde V$ and $\tilde V'$ are isomorphic at all finite places $\ell\neq p$. 
At the archimedian place $\tilde V$ has signature $(n-r, r)$ and Hasse invariant $(-1)^r$,
whereas $\tilde V'$ has signature $(n, 0)$ and Hasse invariant $1$.  \hfb
b) $I^V(\Q) \simeq U(\tilde V')(\Q),$ and $I^{V_0}(\Q)\simeq \kay^1=\ker (\Nm_{\smallkay/\Q})$. 
\end{lem} 
\begin{proof} Recall that $(A^o,\iota^o,\l^o)\in \M(n-r,r)^{\tilde V}(\F)$ is supersingular; hence we can choose an $\OK$-equivariant  isogeny
$A^o\simeq (E^o)^n.$
This yields, in turn, compatible isomorphisms
$$\tilde V'(\Q)= \Hom_{\OK}^0(E^o,A^o) \simeq \End_{\OK}^0(E^o)^n = \kay^n,$$
and
$$\End_{\OK}^0(A^o) \simeq M_n(\End_{\OK}^0(E^o)) = M_n(\kay).$$ 
Hence also the natural map
\begin{equation}\label{trivV'}
\tilde V'\tt\A_f^p\to \Hom_{\smallkay\tt\A_f^p}\big(T^p(E^o)^0, T^p(A^o)^0\big)
\end{equation}
is an isomorphism. This proves a).\hfb
 For b) note that there are obvious homomorphisms $I^V \to U(\tilde V')$ and $I^{V_0}\to \ker (\Nm_{\smallkay/\Q})$  of 
 algebraic groups over $\Q$. Since these induce  isomorphisms $I^V(\A_f^p) \simeq U(\tilde V')(\A_f^p),$ 
 and $I^{V_0}(\A_f^p)\simeq (\kay\otimes{\A_f^p})^1$, these homomorphisms are isomorphisms, and induce on the $\Q$-points the isomorphisms in b).  
\end{proof}
In order to state our uniformization result for special cycles, we need the following definition of special cycles in $\Cal N\times \Cal N^0$, 
which is a slight variant of Definition~3.2 of \cite{KRunitary.I}.
\begin{defn}
For a collection $\und{\und{\xx}}\in \Hom_{\OK\tt\Z_p}(\X_0,\X)^m$, 
let $\ZZ(\und{\und{\xx}})$ be the subfunctor of $\Cal N\times \Cal N_0$, where $\ZZ(\und{\und{\xx}})(S)$ is the set of isomorphism classes of collections
$(X,\iota,\l_X,\rho_X;Y,\iota,\l_Y,\rho_Y)$ in $(\Cal N\times\Cal N_0)(S)$ such that the quasi-homomorphism
$$\rho_X^{-1}\circ\und{\und{\xx}}\circ \rho_Y: Y^m\times_S\bar S \lra X\times_S\bar S$$
extends to a homomorphism from $Y^m$ to $X$. Here  $S\in \Nilp_W$, and
$\bar S = S\times_W\F$ is the special fiber of $S$. \hfb
\end{defn}

\begin{prop}\label{padic.uni.cycles} Fix a base point $(A^o,\iota^o,\l^o; E^o,\iota^o_0,\l_0^o)$ in $\Cal M^{(V^\sh, V_0),  \ss}(\F)$, and define $\tilde V'$ by (\ref{twisted}). 
For $S\in \text{\rm Nilp}_W$, define  ${\rm Inc}_p(T; V^\sh, V_0)(S)$,  the {\rm incidence set}, inside 
$$
\big((\Cal N\times \Cal N_0)(S)\times (G^V(\A_f^p)^0/K^{V^\sh, p}\times G^{V_0}(\A_f^p)^0/K_0^p)\big)\times \tilde V'(\Q)^m$$
to be the subset of collections 
$(\xi, \xi_0, g K^{V^\sh, p}, g_0 K_0^p; \xx^o)$ 
determined by the following incidence relations:
\begin{enumerate}
\item[\it(a)] $h'(\xx^o,\xx^o) = T.$
\smallskip
\item[\it(b)]   
$g^{-1}\circ \und{\xx}^o\circ g_0\in \Hom_{\OK\tt \widehat{\Z}^p}(T^p(E^o),T^p(A^o))^m.$
\item[\it(c)] $(\xi,\xi_0)\in \ZZ(\und{\und{\xx}}^o)(S) \subset (\Cal N\times\Cal N_0)(S).$
\end{enumerate}
Then ${\rm Inc}_p(T; V^\sh, V_0)(S)$ is the set of $S$-points of the formal scheme
$${\rm Inc}_p(T; V^\sh, V_0) = \coprod_{\substack{(g K^{V^\sh, p}, \,g_0K_0^p)}}
\coprod_{\xx^o} \ZZ(\und{\und{\xx}}^o),$$
where $(g K^{V^\sh, p}, g_0K_0^p)$ runs over $G^V(\A_f^p)^0/K^{V^\sh, p}\times G^{V_0}(\A_f^p)^0/K_0^p$, and 
$\xx^o\in \tilde V'(\Q)^m$ runs over the set of $m$-tuples satisfying 
conditions (a) 
and (b). Moreover,  
there is an isomorphism of formal stacks over  
$W$, compatible with the uniformization isomorphism for $\widehat{\M}^{(V,V_0),{\rm ss}}$ in Corollary \ref{unifprod}, 
$$\big(I^V(\Q)\times I^{V_0}(\Q)\big)\back {\rm Inc}_p(T; V^\sh, V_0) \isoarrow \widehat{\ZZ}^{(V^\sh,V_0), \ss}(T) .$$ 
\end{prop}

\begin{proof} Let us recall our fixed trivialization (\ref{trivunits}) of the prime-to-$p$ units. Then 
in the notation of Proposition~\ref{RZkey}, the pair $(\xi, gK^{V^\sh, p})$, resp.\ $(\xi_0, g_0K_0^p)$ determines $(A, \iota, \l)$, resp. $(E, \iota_0, \l_0)$, 
and a quasi-isogeny $\phi: A\to \widetilde{A}^o_S$, resp.\ $\phi_0: E\to \widetilde{E}^o_S$. Let $\xx\in \Hom_{\OK}(E,  A)^m$ be
such that $(A,\iota,\l;E,\iota_0,\l_0;\xx)$ is in $\ZZ(T)(S)$. \hfb
Let 
\begin{equation}\label{defxo}
\xx^o = \phi\circ\xx\circ \phi_0^{-1}\in 
\Hom_{\OK}^0(\widetilde{E}^o_S, \widetilde{A}^o_S)^m = \Hom_{\OK}^0(E^o,A^o)^m
= \tilde V'(\Q)^m,
\end{equation}
be the corresponding collection of quasi-isogenies, and let 
$$\und{\xx}^o =  \eta^{o}\circ \xx^o\circ (\eta_0^{o})^{-1}\in 
\Hom_{\smallkay\tt\A_f^p}(V_0(\A_f^p),V(\A_f^p))^m,$$
and let 
$$\und{\und{\xx}}^o\in \Hom_{\OK\tt\Z_p}^0(\X_0, \X)^m$$
be the collection of quasi-isogenies of $p$-divisible groups induced by $\xx^o$. 
By (i) of Proposition~\ref{RZkey}, we have 
$$h'(\xx^o,\xx^o) = h(\xx,\xx) = T,$$
i.e., condition (a) holds.  Condition (b) follows immediately from (ii) of Proposition~\ref{RZkey}.
Finally, by (iii) of Proposition~\ref{RZkey}, 
$$(X(A),\iota, \l,(\phi_*)_{\bar S};X(E),\iota_0,\l_0,((\phi_0)_*)_{\bar S}, \und{\und{\xx}})\in \ZZ(\und{\und{\xx}}^o)(S),$$
where $\ZZ(\und{\und{\xx}}^o)$ is the cycle in $\Cal N\times \Cal N_0$ defined
above.  

Conversely, if a collection $\xx^o$ satisfying (a), (b) and (c) is given, 
the collection 
$$\xx = \phi^{-1}\circ \xx^o\circ \phi_0 \in \Hom_{\OK}^0(E,A)^m$$
actually lies in $\Hom_{\OK}(E,A)^m$ and satisfies $h(\xx,\xx)=T$. 

Finally, it is easy to check that dividing out by the action of $I^V(\Q)\times I^{V_0}(\Q)$ yields an 
identification with $\widehat{\ZZ}^{(V^\sh,V_0), \ss}(T)$. 
\end{proof}
\begin{remark}\label{thirdsp}
When $T\in \Herm_n(\OK)_{>0}$, 
so that $\ZZ(T)$ has support in the supersingular locus, there is a decomposition
\begin{equation}\label{ZTdecompo}
\ZZ(T) = \coprod_{(V^\sh,V_0)} \ZZ^{(V^\sh,V_0), \ss}(T). 
\end{equation}
Suppose  that $\ZZ(T)^{(V^\sh, V_0),{\rm ss}}\neq \emptyset$. As before, let $\tilde V=\Hom_\smallkay(V_0, V)$ and 
let  $\tilde V'$ be  the unique positive definite hermitian space such that $\inv_\ell(\tilde V')=\inv_\ell(\tilde V)$ for all 
$\ell\neq p$. Then $\tilde V'\simeq V_T$, where $V_T=\kay^n$ with  hermitian form given by $T$. 
Indeed, for any point $(A, \iota, \l; E, \iota_0, \l_0; \xx)\in \ZZ(T)^{(V^\sh, V_0),{\rm ss}}$,  
the last entry defines a $\kay$-linear map $\kay^n\to \tilde V'=\End_\kay^0(E, A)$ which is an isometry.   
Thus, the index set in (\ref{ZTdecompo}) runs over the pairs $(V^\sh,V_0)$ such that $\tilde{V}' \simeq V_T$. 
\end{remark}

\vskip.2in

\centerline{\bf\large Part III:  Eisenstein series }

In the next four sections, we review some material concerning theta integrals, Eisenstein series, the 
Siegel-Weil formula, etc. that will be needed in formulating our main results. We refer to \cite{HKS}, \cite{ichino}, 
and \cite{kudlaannals} for more details. 

\section{The theta integral}\label{sectionthetaint}

For the moment,  we shift notation and allow $V$ to be any nondegenerate hermitian space over $\kay$ of dimension $m$.
Let $G_1 = \UU(V)$ and let $H=\UU(n,n)=\UU(W_0)$, where $W_0$ is a split skew hermitian space of 
dimension $2n$.   
Let $W=V\tt_\smallkay W_0$, with symplectic form, 
$$\langle\gs{v_1\tt w_1}{v_2\tt w_2}\rangle= \tr_{\smallkay/\Q}((v_1,v_2)\gs{w_1}{w_2}),$$ 
as in \cite{HKS}. There is a homomorphism $G_1\times H\sra \Sp(W)$ and $(G_1,H)$ is a reductive dual pair. 

We fix a character $\eta$ of $\kay^\times_{\A}$ whose 
restriction to $\Q^\times_\A$ is $\chi^m$, where $m=\dim_\smallkay V$ and $\chi$ is the global quadratic 
character attached to $\kay$. 
As explained in \cite{HKS}, the choice of $\eta$ determines a homomorphism
$$G_1(\A)\times H(\A) \lra \Mp(W)(\A)$$
where $\Mp(W)(\A)$ is the metaplectic cover of $\Sp(W)(\A)$ and hence
a Weil representation\footnote{We also fix the standard additive character $\psi$ of $\A$, trivial on $\Q$ and on $\widehat{\Z}$
and such that, for $x\in \R$, $\psi(x)= e(x)=e^{2\pi ix}$.} $\o$ of the group $G_1(\A)\times H(\A)$ on the Schwartz space $S(V(\A)^n)$. 
We normalize this so that the action of $G_1(\A)$ is given by $(\o(g,1)\ph)(x) = \ph(g^{-1}x)$. 
The theta function attached to $\ph \in S(V(\A)^n)$ is then 
$$\theta(g,h;\ph) = \sum_{x\in V(\Q)^n} \o(h)\ph(g^{-1}x),$$
where $g\in G_1(\A)$ and $h\in H(\A)$. 

We now suppose that $m=n$ and that $\sig(V)=(n,0)$. The theta integral is then 
$$I(h;\ph) = \int_{G_1(\Q)\back G_1(\A)} \theta(g,h;\ph)\,dg,$$
where the Haar measure $dg$ is taken so that $\vol(G_1(\Q)\back G_1(\A))=1$. 
We take $\ph = \ph_\infty\tt\ph_f$ where
\begin{equation}\label{gaussian}
\ph_\infty(x) = e^{-2\pi \tr(x,x)}
\end{equation}
is the Gaussian,
and $\ph_f$ is the characteristic function of $(\LM\tt\widehat{\Z})^n$ 
for an $\OK$-lattice $\LM$ in $V$.  We take $K_1\subset G_1(\A_f)$ to be the stabilizer of $\LM\tt\widehat{\Z}$ and we note that 
$$\ph_\infty(g^{-1}x) = \ph_\infty(x)$$ 
for $g\in G_1(\R)$. 
Write
$$G_1(\A_f) = \coprod_j G_1(\Q)g_j K_1.$$
Then
$$
I(h;\ph) = \vol(G_1(\R)K_1)\cdot \sum_j |\Gamma_j|^{-1} \sum_{x\in V(\Q)^n}\o(h)\ph(g_j^{-1}x),
$$
where $\Gamma_j = G_1(\Q) \cap g_j K_1 g_j^{-1}$ is the group of isometries of the 
hermitian lattice $\LM_j = (g_j \,L\tt\widehat{\Z})\cap V(\Q)$, and the lattice $\LM_j$ runs over representatives for the 
classes in the $G_1$-genus of $\LM$. 
Note that 
$$1 = \vol(G_1(\Q)\back G_1(\A))= \vol(G_1(\R)K_1)\cdot \sum_j |\Gamma_j|^{-1},$$ 
so that 
\begin{equation}\label{mass.def}
\text{mass}(\LM): = \sum_j |\Gamma_j|^{-1} = \vol(G_1(\R)K_1)^{-1}
\end{equation}
is the classical mass of the genus of $\LM$. 

Taking $h_f=1$, we have
$$
I(h;\ph) = \text{mass}(\LM)^{-1}\,\sum_j |\Gamma_j|^{-1} \sum_{x\in L_j^n} \o_\infty(h)\ph_\infty(x).
$$
Let $D(W_0)$ be the space of negative $n$-planes in $W_0(\R)$, so that 
$$D(W_0) \simeq \{\, z\in M_n(\C)\mid v(z) := (2i)^{-1}(z - {}^t\bar z)>0\ \}.$$
Write $z = u(z)+i v(z)$, with $u(z) =2^{-1}(z+{}^t\bar z)$, and let 
\begin{equation}\label{hz-def}
h_z = \begin{pmatrix} 1_n&u(z)\\{}&1_n\end{pmatrix}\begin{pmatrix} a&{}\\{}&{}^t\bar{a}^{-1}\end{pmatrix} \quad \in  H(\R),
\end{equation}
where $a\in \GL_n(\C)$ with $v(z) = a{}^t\bar{a}$.  Note that $h_z(i 1_n)=z$. 
Now taking $h_\infty = h_z$, we have
\begin{equation}
\o_\infty(h_z)\ph_\infty(x) =  \eta_\infty(\det(a))\,\det(v(z))^{\frac{n}2}\,q^{(x,x)},
\end{equation}
where, for $T\in \Herm_n(\C)$, we write $q^T= e(\tr(Tz))$.
Thus,  we obtain the classical expression\footnote{This is called the analytic genus invariant in Siegel \cite{siegel}
and Braun \cite{braun}.}
\begin{align}\label{classicalFC}
I(z;\LM) :&= \eta_\infty(\det(a))^{-1}\,\det(v(z))^{-\frac{n}2}\,I(h_z;\ph)\\
\nass
{}& = 
\text{mass}(\LM)^{-1}\cdot \sum_{T\in \Herm_n(\OK)} r_{\text{\rm gen}}(T,\LM)\,q^T.\notag
\end{align}
Here the $T$-th Fourier coefficient is the representation number
\begin{equation}\label{r-gen}
r_{\text{\rm gen}}(T,\LM)= \sum_j |\Gamma_j|^{-1}\,|\O(T,\LM_j)|,
\end{equation}
where $\LM_j$ runs over the classes of lattices in the $G_1$-genus of $\LM$,
and
$$\O(T,\LM_j) = \{\,x\in \LM_j^n\mid (x,x) = T\, \}.$$

\section{The Siegel formula}\label{sectionSiegelfor}

Now we assume that $\det(T)\ne 0$ and we express $I_T(z;\LM)$, the $T$-th Fourier coeffiicient of 
the theta integral defined in (\ref{classicalFC})  in terms of local densities. 
This is done in detail in section 6 of Ichino, \cite{ichino}, as part of the proof of the regularized Siegel-Weil formula 
for unitary groups.  Here we specialize his formulas to the case where $\dim(V)=n$. 

We slightly shift notation and take $dg$ to be Tamagawa measure on $G_1(\A)$, defined with respect to 
 a gauge form $\nu_{G_1}$ on $G_1$.  
Since the Tamagawa number of $G_1$ is $2$, we include a factor of $\frac12$ in 
the definition of the theta integral.  
For general factorizable $\ph= \tt_v \ph_v\in S(V(\A)^n)$, Ichino obtains  
$$I_T(h;\ph) = \frac12 \, L(1,\chi)^{-1}\,\prod_v \l_v^{-1} \,\int_{\O_{T}(\Q_v)} \o(h_v)\ph_v(x)\,d\mu_{T,v}(x),$$
where the convergence factors are 
\begin{equation}\label{convergence.factors}
\l_v^{-1} = L_v(1,\chi_v),
\end{equation}
and where $\O_{T}(\Q_v)$ is the set of $\Q_v$-points of the variety 
$$\O_T = \{ x\in V^n\mid (x,x)=T\,\}.$$
The measures $d\mu_{T,v}$ are defined as follows.  
We choose $x\in V(\Q_v)^n$ with $(x,x)=T$, and define an isomorphism
$i_x:G_1(\Q_v) \isoarrow \O_{T}(\Q_v)$ via $g_v\mapsto g_v\cdot x$. The measure $d\mu_{T,v}$ on $\O_{T}(\Q_v)$ is 
obtained,  via this isomorphism, from the Tamagawa measure $dg_v$ on $G_1(\Q_v)$ determined by $\nu_{G_1}$. 
Note that, for a place $v\notin S$, for a sufficiently large finite set of places $S$ including the archimedean place, 
$$\int_{\O_{T}(\Q_v)} \o(h_v)\ph_v(x)\,d\mu_{T,v}(x) = \prod_{i=1}^n L_v(n-i+1,\chi_v^{n-i+1})^{-1},$$
so that \footnote{Here, as usual, the superscript $S$ indicates that the Euler factors for 
the primes in $S$ are omitted.}
\begin{align*}
I_T(h;\ph) &= \frac12 \, L(1,\chi)^{-1}\,\prod_{i=1}^{n-1} L^S(n-i+1,\chi^{n-i+1})^{-1}\\
\nass
{}&\qquad\qquad\qquad\qquad\times\prod_{v\in S}\l_v^{-1} \int_{\O_{T}(\Q_v)} \o(h_v)\ph_v(x)\,d\mu_{T,v}(x)\\
\nass
\nass
{}&= \frac12 \,\prod_{i=1}^{n} L^S(n-i+1,\chi^{n-i+1})^{-1}\,\prod_{v\in S} \int_{\O_{T}(\Q_v)} \o(h_v)\ph_v(x)\,d\mu_{T,v}(x).
\end{align*}

On the other hand, if $\P(s)$ is the Siegel-Weil section associated to a factorizable $\ph$, we have a factorization, 
\cite{ichino}  \cite{tan},
$$E_T(h,s,\P) =L^S(2s,n,\chi)^{-1}\cdot \prod_{v\in S} W_{T,v}(h_v,s,\P_v),$$
of the $T$-th Fourier coefficient of the Eisenstein series again,
for a sufficiently large finite set of places $S$, 
where, for convenience, we set
\begin{equation}\label{Lfactor}
L^S(2s,n,\chi) = \prod_{i=1}^n L^S(2s+n-i+1,\chi^{n-i+1}).
\end{equation}  
Here, for $\Re(s)>n$, the Whittaker function (discussed in more detail in section 10 below) is given by 
\begin{equation}\label{whittaker.def}
W_{T,v}(h_v,s,\P_v) = \int_{\Herm_n(\smallkay_v)} \P_v(w^{-1} n(b_v)h_v,s)\,\psi_v(-\tr(Tb_v))\,db_v,
\end{equation}
where $db_v$ is the self-dual measure on $\Herm_n(\kay_v)$ with respect to the pairing $\gs{b_1}{b_2} = \psi_v(\tr(b_1 b_2))$. 
By a standard argument due to Weil, \cite{weil.acta.II}, \cite{rallis.LN}, see also   \cite{KRYbook}, p.\ 127,  at $s=0$ we have
\begin{equation}\label{whittaker.value}
W_{T,v}(h_v,0,\P_v) = \gamma_v\,\int_{\O_{T}(\Q_v)} \o(h_v)\ph_v(x)\,d\mu_{T,v}(x),
\end{equation}
where $\gamma_v$ is a Weil index. Note that, in this identity, the gauge form $\nu_{G_1}$ defining the measure $d\mu_{T,v}$ is the one determined by 
the moment map (\ref{moment.map}). This point is discussed in more detail in section \ref{section.densities}, below. 
Since, for $S$ sufficiently large,  
$$\prod_{v\in S} \gamma_v=1,$$
this proves that 
\begin{equation}
2\,I_T(h;\ph) = E_T(h,0,\P).
\end{equation}

In particular, taking $h=h_z$ as in (\ref{hz-def}), we have
\begin{equation}\label{FC.one}
2\,I_T(z;\ph) = \eta_\infty(\det(a))^{-1}\,\det(v(z))^{-\frac{n}2}\,E_T(h_z,0,\P).
\end{equation}

Let $\P_\infty^n(s)$ be the Siegel-Weil section associated to the Gaussian (\ref{gaussian}), 
and, for each prime $\ell$, let $\P_\ell(s)$ be the Siegel-Weil section attached to 
$\ph_{\ell}$, the characteristic function of the set $(\LM\tt\Z_\ell)^n$.  We obtain the fundamental
identity
\begin{equation}\label{fund.identity}
2\,\text{mass}(\LM)^{-1}\,r_{\text{\rm gen}}(T,\LM)\,q^T
=L^S(0,n,\chi)^{-1}\cdot \prod_{\ell\in S_f} W_{T,\ell}(1,0,\P_\ell)\cdot W_{T,\infty}(h_z,0,\P_\infty^n),
\end{equation}
for a sufficiently large set of  primes $S= S_f \cup\{\infty\}$. 

\section{The incoherent case}\label{sectionincoherent}

We now turn to the incoherent case. Fix a hermitian space $V$ with signature $(n-1,1)$ and assume that $V$ contains a self-dual lattice
$L$, i.e., that $V$ is relevant in the sense of section 2.  
Take $\ph_f\in S(V(\A_f)^n)$ to be the characteristic function of $(L\tt\widehat{\Z})^n$.
The incoherent Eisenstein series is defined to be $E(h,s,\P)$ for $\P(s) = \P(s,L)=\P_\infty(s)\tt\P_f(s,L)$ where $\P_f(s, L)$ is the Siegel-Weil section 
associated to $\ph_f$ and $\P_\infty(s)=\P_\infty^n(s)$ is the Siegel-Weil section attached to the Gaussian of a hermitian space of 
signature $(n,0)$. 
As in the previous section, e.g., (\ref{FC.one}), we let 
\begin{equation} \label{semi.classical}
E(z,s,\P) = \eta_\infty(\det(a))^{-1}\,\det(v(z))^{-\frac{n}2}\, E(h_z,s,\P)
\end{equation}
be the corresponding  `classical' Eisenstein series. 
As explained in \cite{kudlaannals} (in the orthogonal case), $E(z,0,\P)=0$ and we are interested in the 
central derivative $E'(z,0,\P)$.

For $T\in \Herm_n(\kay)$ with $\det(T)\ne 0$, let 
\begin{equation}\label{Diff.def}
\Diff(T,V) = \{\, p<\infty\mid \chi_p(\det(T)) =-\chi_p(\det(V))\, \}.
\end{equation}
Note that only ramified or inert primes can occur in this set. Moreover, 
if $T>0$, then $1=\chi_\infty(\det(T)) = -\chi_\infty(\det(V))$, and hence  $\Diff(T,V)$ has odd cardinality, due to the 
product formula (\ref{product.formula}). 
Finally, since $V$ contains a self dual lattice, an inert place $p$ lies in $\Diff(T,V)$ if and only if $\ord_p(\det(T))$ is odd. 
By the analysis reviewed in the previous section, 
$$E_T(h,s,\P) = L^S(2s,n,\chi)^{-1}\cdot \prod_{v\in S} W_{T,v}(h_v,s,\P_v),$$
for a sufficiently large finite set  $S$ of places including the archimedean place. We assume that $S$ contains $\Diff(T,V)$. 
By (8.4),  
$$W_{T,p}(h_p,0,\P_p) =0,$$
for all $p\in \Diff(T,V)$, since $V_p$ does not represent $T$. This implies the following.

\begin{lem} (i)  If $T\in \Herm_n(\kay)_{>0}$ and $|\Diff(T,V)|>1$, then 
$$E'_T(h,0,\P) = 0.$$
(ii) If $T\in \Herm_n(\kay)_{>0}$ and $\Diff(T,V) = \{\,p\,\}$, then 
$$E'_T(h,0,\P) = W'_{T,p}(h_p,0,\P_p)\cdot L^S(0,n,\chi)^{-1}\cdot \prod_{\substack{v\in S\\  \snass v\ne p}} W_{T,v}(h_v,0,\P_v).$$
(iii) If $\Diff(T,V)$ is empty, so that $\sig(T) = (n-r,r)$ for $r$ odd, then 
$$E'_T(h,0,\P) = W'_{T,\infty}(h_\infty,0,\P_\infty^n)\cdot L^S(0,n,\chi)^{-1}\cdot \prod_{\substack{v\in S\\  \snass v\ne \infty}} W_{T,v}(h_v,0,\P_v).$$
\end{lem} 

These formulas can be made more explicit; we will only do this in a special case.  
We take $h=h_z$, as above and note that, by analogy with  Lemma~5.2.1 of \cite{KRYbook}, 
$E_T(h_z,s,\P)$ vanishes unless $T\in \Herm_n(\OK)$.  

\begin{remark} In this situation, we will write 
$W_{T,p}(s,\P_p)$  for $W_{T,p}(1,s,\P_p)$ for a finite prime $p$. 
\end{remark}

We suppose that $T\in \Herm_n(\OK)_{>0}$ with $\Diff(T,V) = \{\,p\,\}$ for an odd inert prime $p$. 
Let $V'= V_T$. Note that, up to isometry, $V'$ depends only on $V$ and $p$, since
it is the unique positive definite hermitian space with $\inv_p(V') = -\inv_p(V)$ and $\inv_\ell(V')=\inv_\ell(V)$ for all other finite primes. 
Fix an isomorphism $V'(\A_f^p) =V(\A_f^p)$, and let $L'$ be the lattice in $V'$ determined by 
$L'\tt\widehat{\Z}^p = L\tt \widehat{\Z}^p$ and $L'_p = \L^*$ where $\L\subset V'_p$ is a vertex of type $1$ and level $0$
as in \eqref{defvertex}.  Let $\ph'_p\in S((V'_p)^n)$ be the characteristic function of $(L'_p)^n$, and let 
$\ph'= \ph'_p\tt\ph^p$ where $\ph^p\in S(V(\A_f^p)^n)$ is the characteristic function of  $(L\tt \widehat{\Z}^p)^n$. 
The following facts will be proved in the next section.
\begin{prop}\label{whittaker.info} For $T\in \Herm_n(\OK)_{>0}$ with $\Diff(T,V) = \{\,p\,\}$ for an odd inert prime $p$, 
suppose that, under the action of $\GL_n(O_{\smallkay,p})$, 
$$T\sim \diag(1_{n-2},p^a,p^b), \qquad 0\le a<b.$$
Let 
$S = 1_n$ and $S'=\diag(1_{n-1},p)$. \hfb 
(i) Let $\P_p(s)$ be the Siegel-Weil section associated to the characteristic function $\ph_p$ of the set 
$(L\tt\Z_p)^n$ in $V(\Q_p)^n$. Then 
$$W'_{T,p}(0,\P_p) = \gamma_p(V)^n\,\a_p(S,S)\,\mu_p(T)\,\log p,$$
where 
$$\mu_p(T) =\frac12\sum_{\ell=0}^a p^{\ell}(a+b-2\ell+1).$$
and $\a_p(S,S)$ is the local density, cf. (\ref{local.density}).
Note that 
$\a_p(S,S) = L_p(0,n,\chi)^{-1}$,
where
$L_p(2s,n,\chi)$ 
is the local factor at $p$ of the $L$-function (\ref{Lfactor}).  \hfb
(ii) Let $\P'_p(s)$ be the Siegel-Weil section associated to the characteristic function $\ph'_p$ 
of the set $(L'_p)^n$ in $(V'_p)^n$, where $L'_p$ is as explained above. Then
$$W_{T,p}(0,\P'_p) =(-1)^n\gamma_p(V)^n\,p^{-n}\,\a_p(S',S').$$  
In particular, this quantity is nonzero and independent of $T$.
\end{prop}
Using the nonvanishing in (ii),  we can write
\begin{align*}
E_T'(h,0,\P) &=\frac{W'_{T,p}(0,\P_p)}{W_{T,p}(0,\P'_p)} \cdot  L^S(0,n,\chi)^{-1}\cdot 
\prod_{v\in S} W_{T,v}(h_v,0,\P'_v)\\
\nass
{}&=  \frac{W'_{T,p}(0,\P_p)}{W_{T,p}(0,\P'_p)} \cdot 2\,I_T(h,\ph').
\end{align*}
Then, using (\ref{fund.identity}) and Proposition~\ref{whittaker.info}, we obtain the expression
$$
E'_T(z,0,\P) = 
(-1)^n\, \mu_p(T)\,\log(p)\cdot C_p\cdot 2\,\text{\rm mass}(L')^{-1}\,r_{\text{\rm gen}}(T,L')\cdot q^T,$$
where
$$C_p = p^n \frac{\a_p(S,S)}{\a_p(S',S')}.$$

Finally, we write $G'_1 =U(V')$ and note that $G_1 = U(V)$ and $G'_1$ are inner forms of each other. 
Our fixed choice of a gauge form $\nu_1=\nu_{G_1}$ on $G_1$, together with an isomorphism of inner forms
$\Psi:G'\isoarrow G$ defined over $\bar\Q$,  determines a gauge form $\nu'_1=\nu_{G'_1}= \Psi^*(\nu_1)$ on $G'_1$. 
This form is again defined over $\Q$ and hence 
there are associated Haar measures on $G_1(\Q_p)$ and $G'_1(\Q_p)$, cf. \cite{kottwitz.tamagawa}, p. 631. 
Now, by (\ref{mass.def}) for the positive definite space $V'$, we have
$$2\,\text{\rm mass}(L')^{-1} = \vol(G'_1(\R)K'_1,d\nu'_1)$$
where the volume is taken with respect to Tamagawa measure and our fixed convergence factors (\ref{convergence.factors}). 
By (iii) of Lemma~\ref{measures.lemma}, we have
$$C_p\cdot \vol(K'_{1,p},d\nu'_1) =\vol(K_{1,p},d\nu_1).$$
Thus
$$C_p \cdot 2\,\text{\rm mass}(L')^{-1} = \vol(G'_1(\R), d \nu_1')\,\vol(K_1, d\nu_1).$$
As we will see in a moment,  the quantity 
$\vol(G'_1(\R),d\nu'_1)$
is independent of $p$.  For later convenience, we will now write 
\begin{equation}
E(z,s,L) = E(z,s,\P)
\end{equation}
to emphasize the dependence on the choice of the self-dual lattice $L$. 
\begin{cor}\label{FC.formula}  For $T\in \Herm_n(\OK)_{>0}$ with $\Diff(T,V) = \{\,p\,\}$ for an odd inert prime $p$, 
suppose that $T$ satisfies the condition of Proposition~\ref{whittaker.info}. Then 
$$
E'_T(z,0,L) = (-1)^n\,
C\cdot \mu_p(T)\,\log(p)\cdot r_{\text{\rm gen}}(T,L')\cdot q^T,$$
where $L'$ is  obtained from $L$ 
as explained just before Proposition~\ref{whittaker.info} and  
$$C=\vol(G'_1(\R),d\nu_1')\,\vol(K_1, d\nu_1).$$
\end{cor} 

Note that the factor $(-1)^n$ arises due to the switch from an incoherent to a coherent Eisenstein series, 
more precisely, from the sign change in the Weil invariant $\gamma_p(V'_p)= -\gamma_p(V_p)$ and the
presence of the 
factor $\gamma_p(V)^n$ in the formula for the local Whittaker function in Proposition~\ref{interpolate}.

Finally, the constant $C$ has the following nice
interpretation. 
\begin{lem}\label{EPconstant}
$$2\,C^{-1}  = (-1)^{n-1}\,\chi(\mathbb P^{n-1}(\C))^{-1}\cdot \chi_{\bullet}\big(\, G_1(\Q)\back (D\times G_1(\A_f)/K_1)\,\big ),$$
where, for an arithmetic group $\Gamma$ of isometries of the $n-1$ ball $D$, 
$$\chi_{\bullet}(\Gamma\back D):=\int_{\Gamma\back D}\Omega,$$
is the integral of the Gauss-Bonnet form $\Omega$.   
\end{lem}
\begin{remark}  If $\Gamma$ is torsion free, then $\chi_\bullet(\Gamma\back D)$ is the Euler characteristic of 
the manifold $\Gamma\back D$. Also note that $\mathbb P^{n-1}(\C)$ is the compact 
dual of $D$ and that $(-1)^{n-1}\,\chi_{\bullet}(\Gamma\back D)$ is positive.
\end{remark}
\begin{proof}
First note that, for Tamagawa measure,
$$2 = \vol(G_1(\Q) \back G_1(\A),d\nu_1) = \big(\sum_j \vol(\Gamma_j \back G_1(\R))\big) \cdot \vol(K_1),$$
where
$$G_1(\A) = \coprod_j G_1(\Q)G_1(\R) g_j K_1$$
and
$\Gamma_j = G(\Q) \cap g_j K_1 g_j^{-1}$.
Thus
$$2\,\vol(K_1)^{-1} =   \sum_j \vol(\Gamma_j \back G_1(\R)),$$
and so
$$2\,C^{-1} = \vol(G'_1(\R))^{-1} . \sum_j \vol(\Gamma_j \back G_1(\R)).$$
Here we are using the measure obtained from matching gauge forms on $G_1(\R) = U(n-1,1)$
and its compact dual $G'_1(\R) = U(n)$, as described above.
But then, the ratio is independent of the choice of this gauge form.
Now write the real Lie algebra of $G_1$ as
$\g = \frak k + \frak p$
where $\dim \frak p = 2n-2$ and
$\g' = \frak k + i\frak p$.
We can use the standard recipe, described in Serre's article \cite{serreEP}
on Euler-Poincare measures, pp.135-138,
to give a top degree form  $\Omega$ on $\frak p$   (degree $2n-2$)  such that, writing $\Omega$ for the corresponding $G_1(\R)$-invariant form on $D$, 
$$\int_{\Gamma \back D} \Omega = \chi(\Gamma\back D),$$
for $\chi(\Gamma\back D)$ the Euler characteristic of the manifold $\Gamma\back D$, when $\Gamma$ is torsion free, 
i.e., $\Omega$ gives the Gauss-Bonnet form on $D$, \cite{harder}.
We can extend this to $\g$ by wedging with a top degree form  $\eta$ on $\frak k$
which gives the corresponding maximal compact subgroup volume $1$.
According to \cite{serreEP} p.136, the resulting measure $\mu$ on $G_1(\R)$ is an Euler-Poincare measure.
Then,  since the form on $i\frak p$ corresponds to $(-1)^{n-1}$ times the Gauss-Bonnet form on 
$\mathbb P^{n-1}(\C)$, the  measure of $G'_1(\R)$ is
$$\vol(G'_1(\R),\mu') =(-1)^{n-1} \chi(\mathbb P^{n-1}(\C)).$$
\end{proof}

\section{Representation densities}\label{section.densities}

In this section, we derive the information about local Whittaker functions and their central derivatives 
summarized in Proposition~\ref{whittaker.info} from various explicit formulas for representation densities 
of hermitian forms. 

Recall that the classical representation densities are define as follows. 
For nonsingular matrices 
$S\in \Herm_m(\OKp)$ and $T\in \Herm_n(\OKp)$, let
$$A_{p^k}(S,T) = \{ \ x\in M_{m,n}(\OK/p^k\OK)\mid S[x] \equiv T\!\! \!\mod p^k\,\Herm_n(\OK)^\vee \ \},$$
where
$$\Herm_n(\OK)^\vee = \{ b\in \Herm_n(\kay)\mid \tr(b c) \in \OK \text{ for all $c\in \Herm_n(\OK)$}\ \}.$$
The representation density is then defined as the limit 
\begin{equation}\label{local.density}
\a_p(S,T) = \lim_{k\rightarrow \infty} (p^{-k})^{n(2m-n)}\,|A_{p^k}(S,T)|.
\end{equation}
An explicit formula for  $\a_p(S,T)$ has been given by Hironaka, \cite{hironaka}, in the case of an inert 
prime $p$. 

These quantities are related to the Whittaker functions $W_{T,p}(s,\P_p)$ where 
$\P_p(s)$ is the Siegel-Weil section determined by the characteristic function $\ph_p\in S(V_p^n)$ of  $(L_p)^n$
for a lattice $L_p$ on which the hermitian form has matrix $S$. We sketch the argument, which is analogous to that 
given for quadratic forms in \cite{kudlaannals}, in order to be precise about the various constants involved. 

Recall that 
\begin{equation}\label{whittaker.again}
W_{T,p}(s,\P_p) = \int_{\Herm_n(\smallkay_p)} \P_p(w^{-1} n(b),s)\,\psi_p(-\tr(Tb))\,db.
\end{equation}
where\footnote{Note that the conventions here differ slightly from those of \cite{HKS}, (6.5).} 
$$w = \begin{pmatrix}{}&1_n\\-1_n&{}\end{pmatrix}, \qquad n(b) = \begin{pmatrix}1_n&b\\ {}&1_n\end{pmatrix},$$
and $db$ is the self-dual measure with respect to the pairing
$$\gs{b}{c} = \psi_p(\tr(bc))$$
on $\Herm_n(\kay_p)$. In particular, note that 
$$\vol(\Herm_n(\OKp),db) = |\Delta|_p^{n(n-1)/4}.$$

\begin{prop}\label{interpolate}  Let 
$$S_r = \begin{pmatrix}{}&{}&1_r\\{}&S&{}\\1_r&{}&{}\end{pmatrix}.$$
Then
$$
W_{T,p}(r,\P_p)=\gamma_p(V)^n\, |N(\det S)|_p^{\frac12n}\,|\Delta|_p^{e}\, \a_p(S_r,T),
$$
where $\gamma_p(V)$ is an eight root of unity given by (\ref{gammapV}) 
and 
$e=\frac12 n(m+2r) +\frac14 n(n-1)$.
\end{prop}
\begin{proof}
We first use Rallis's interpolation trick. 
Let $V_p^{[r]} = V_p\oplus V_{r,r}$ where $V_{r,r}$ is the split space of signature $(r,r)$, and let
$$\ph_p^{[r]} = \ph_p\tt \ph_{r,r} \in S((V^{[r]}_p)^n)$$
where $\ph_{r,r}$ is the characteristic function of $(L_{r,r})^n$ for a self-dual lattice $L_{r,r}$ in $V_{r,r}$. 
Then, as in \cite{kudlaannals}, p.642, 
$$\P_p(h,r) = \o(h)\ph_p^{[r]}(0),$$
where $\o$ is the Weil representation of $H(\Q_p)$ on $S((V^{[r]}_p)^n)$.  Note that we use the character $\eta$  fixed above
to define the Weil representation for all $r$. 
Recall,  \cite{splitting}, that for $w=w_n\in \UU(n,n)$ as above, 
$$\o(w^{-1})\ph_p(x) = \gamma_p(V)^n\int_{V_p^n} \psi_p(-\tr_{\smallkay/\Q_p}\tr(x,y))\,\ph_p(y)\, dy,$$
where
\begin{equation}\label{gammapV}
\gamma_p(V) = (\Delta,\det(V_p))_p\,\gamma_p(-\Delta,\psi_{p,\frac12})^{m}\,\gamma_p(-1,\psi_{p,\frac12})^{-m},
\end{equation}
where $\gamma_p(a,\psi_{p,\frac12})$ is the Weil index, \cite{splitting} with the additive character 
$\psi_{p,\frac12}$ given by $\psi_{p,\frac12}(x) = \psi_p(\frac12 x)$, and $dy$ is the self-dual measure with respect to the pairing 
$$\gs{x}{y} = \psi_p(\tr_{\smallkay/\Q_p}\tr(x,y))$$
on $V_p^n$.  
Inserting this in (\ref{whittaker.again}) and writing $N$ for $\Herm_n(\kay_p)$, we have
\begin{align*}
W_{T,p}(r,\P_p) & = \int_N \psi_p(-tr(Tb))\,\P_p(w^{-1} n(b),r)\,db\\
\nass
{}&=\int_N \psi_p(-tr(Tb))\, \gamma_p(V)^n\,\int_{(V_p^{[r]})^n}\,\psi_p(\tr(b(x,x)))\,\ph_p^{[r]}(x)\,dx\,db\\
\nass
{}&=\gamma_p(V)^n\,|\Delta|_p^{n(n-1)/4}\,\lim_{k\rightarrow \infty} p^{kn^2} \int\limits_{\substack{ x\in (V_p^{[r]})^n \\ \snass (x,x) - T\in p^k \Herm_n(\OKp)^\vee}}
\ph_p^{[r]}(x)\,dx.
\end{align*}
Here, in the last step, we have computed the integral over $N$ as the limit of the integrals over the sets $p^{-k}\,\Herm_n(\OKp)\subset N$.

Take $\ph_p=\ph_{L_p}$ and choose an integral basis for $L_p\oplus L_{r,r}$ for which the hermitian form has matrix $S_r$, 
as above.
Then $(L_p\oplus L_{r,r})^n = M_{m+2r,n}(\OKp)$, and 
we have
$$\vol(M_{m+2r,n}(\OKp),dx) = |N(\det S_r)|_p^{\frac{n}2}\,|N(\d_p)|_p^{\frac{n(m+2r)}2} = |\det S_0|_p^n\,|\Delta|_p^{\frac{n(m+2r)}2},$$
where $|\ |_p$ is the norm on $\Q_p$ and $\d_p$ is the different of $\kay_p/\Q_p$.
We break the integral up into cosets of $p^k M_{m+2r,n}(\OK)$ in $M_{m+2r,n}(\OK)$.

This gives
\begin{align*}
W_{T,p}(r,\P_p) &=
\gamma_p(V)^n\, |\det S|_p^n\,|\Delta|_p^{e}\,\lim_{k\rightarrow \infty} (p^k)^{n^2 -2(m+2r)n}
|A_{p^k}(S_r,T)|\\
\nass
{}&=\gamma_p(V)^n\, |\det S|_p^n\,|\Delta|_p^{e}\, \a_p(S_r,T).
\end{align*}
where $e=\frac{n(m+2r)}2 +n(n-1)/4$.
\end{proof}

\begin{proof}[Proof of Proposition~\ref{whittaker.info}]
In the case of an odd inert prime, the formula above reduces to 
$$W_{T,p}(r,\P_p)=\gamma_p(V)^n\, |\det S|_p^n\, \a_p(S_r,T).$$
Now, by Proposition~9.1~of \cite{KRunitary.I} (which uses Nagaoka's formula \cite{nagaoka}), we have
\begin{align*}
W'_{T,p}(0,\P_p) &= \gamma_p(V)^n\, |\det S|_p^n\,\a_p'(S,T)\,\log(p)\\
\nass
{}& =\gamma_p(V)^n\, |\det S|_p^n\,\a_p(S,S)\,\mu_p(T)\,\log(p).
\end{align*} 
This proves (i), where we note that (by Proposition 9.1 of \cite{KRunitary.I})
$$\a_p(S,S) = \prod_{i=1}^n(1-(-1)^ip^{-i}) = L_p(0,n,\chi)^{-1}.$$

To prove (ii), let $S' = \diag(1_{n-1},p)$ and consider 
$$W_{T,p}(0,\P'_p) = \gamma_p(V'_p)^n\,p^{-n}\,\a_p(S',T),$$
where, by (\ref{gammapV}), $\gamma_p(V'_p) = - \gamma_p(V_p)$. 
Note that $S'$ is the matrix for the hermitian form of $V'_p$ on the lattice $\L^*$, 
where $\L$ is a vertex of type $1$ and level $0$ in $V'_p$.  
The analogue
of the reduction formula of Proposition~9.3 of \cite{KRunitary.I}, which follows immediately from Corollary~9.12  of \cite{KRunitary.I},  implies that  
$$\a_p(S',T) = \a_p(S',1_{n-2})\,\a_p(S'',T''),$$
where $S'' = \diag(1,p)$ and $T''=\diag(p^a,p^b)$. 

We will give the proof of the following result after our discussion of gauge forms. 
\begin{prop} \label{S''T''} The quantity $\a_p(S'',T'')$   
is independent of $a$ and $b$.
\end{prop}
 
\begin{cor}  For $n\ge 2$, 
$$\a_p(S',T) = \a_p(S',S').$$  
\end{cor}

This finishes the proof of Proposition~\ref{whittaker.info}. 
\end{proof}

Next, we need to obtain some information about the dependence  on $p$ of the 
quantity $\text{\rm mass}(L')$. To do this, we first
record a few facts about the moment map and gauge forms
\begin{equation}\label{moment.map}
h:V^n\lra \Herm_n, \qquad x \mapsto (x,x),
\end{equation}
for a hermitian space $V$. 
Let 
$$V^n_{\text{\rm reg}}= \{\ x\in V^n\mid\det(h(x))\ne 0\,\}.$$
Let $G_1=\UU(V)$ and note that for any $x\in V^n_{\text{\rm reg}}$ with $h(x) =T$, the map 
$$i_x:G_1\isoarrow \O_T=h^{-1}(T), \qquad g\mapsto g\cdot x$$
is an isomorphism. 

Fix basis elements $\a\in \wedge^{\text{\rm top}}(V^n)^*$ 
and $\b\in\wedge^{\text{\rm top}}(\Herm_n)^*$ and write $\a$ and $\b$ for the translation invariant forms
they define. Here note that the exterior products are taken for the  
rational vector spaces $V$ and $\Herm_n$; in particular the forms $\a$ and $\b$ have degrees $2 n^2$ and $n^2$ respectively. 
Then there is a form 
$\nu$ of degree $n^2$ on $V^n_{\text{\rm reg}}$ such that 
(i) $\a = h^*\b\wedge \nu$, (ii) $\nu$ is invariant under the action of 
$G_1\times \GL(n)$ on $V^n$ and (iii) for all points $x\in V^n_{\text{\rm reg}}$, the restriction of $\nu$ to $\ker(dh_x)$
is nonzero. In particular, the pullback $\nu_{G_1}=i_x^*\nu$ is a gauge form on $G_1$ which is 
independent of the choice of $x$. It determines a Tamagawa measure on $G_1(\A)$ via a choice of 
convergence factors $\l_v$, as in (\ref{convergence.factors}) above. The analogous facts in the
orthogonal case are described in detail in Lemma~5.3.1 of \cite{KRYbook}.

In the local situation, we have the following useful observations.
The analogue of Lemma~4.2 of \cite{rallis.LN} implies that the distribution
$$\ph \mapsto W_{T,p}(0,\P_{\ph})$$
on $S(V_p^n)$ is proportional to the orbital integral
$$O_{T,p}(\ph) = \int_{\O_T(\Q_p)} \ph(x)\,d_{T,p}(x),$$
where $d_{T,p}(x)$ is the measure on $\O_T(\Q_p)$ determined by the restriction of 
$\nu$ to $\O_T$. 
More precisely, by the same argument as in pp.~121-127 of \cite{KRYbook}, we have
$$W_{T,p}(0,\P_\ph) = C_p(V,\a,\b,\psi)\cdot O_{T,p}(\ph),$$
where 
$$C_p(V,\a,\b,\psi) = \frac{\gamma_p(V)^n\,c_p(\b,\psi)}{c_p(\a,\psi)}.$$
Here 
$$d_{\a,p}x = c_p(\a,\psi)\,dx ,$$
where
$d_{\a,p}x$ is the measure on $V^n_p$ determined by the gauge form $\a$ and $dx$ is the 
self-dual measure with respect to the pairing $\gs{x}{y} = \psi_p(\tr_{\smallkay_p/\Q_p}\tr(x,y))$. 
Similarly, 
$$d_{\b,p}b = c_p(\b,\psi)\,db$$
where
$d_{\b,p}x$ is the measure on $\Herm_n(\kay_p)$ determined by the gauge form $\b$ and $db$ is the 
self-dual measure with respect to the pairing $\gs{b}{c} = \psi_p(\tr(bc))$. 
Taking $\o_p(h)\ph$ in place of $\ph$, we get
$$W_{T,p}(h,0,\P_\ph) = C_p(V,\a,\b,\psi)\cdot O_{T,p}(\o(h)\ph).$$

Note that, globally, 
$$\prod_{p\le \infty}C_p(V,\a,\b,\psi) = 1.$$
 
\begin{lem}\label{measures.lemma}
(i) Let $L_p\subset V_p$ be an $\OKp$-lattice with basis $\{\,v_1,\dots,v_n\,\}$, and let 
$S = ((v_i,v_j))$ be the matrix for the hermitian form on $L_p$ with respect to this basis. 
Let $K_p\subset U(V_p)=G_1(\Q_p)$ be the stabilizer of $L_p$. Then 
$$h^{-1}(S) \cap L_p^n  = K_p\cdot \bold v,$$
where $\bold v = [v_1,\dots,v_n]\in V_p^n$. \hfb
(ii) Let $dg$ be the measure on $G_1(\Q_p)$ determined by the 
gauge form $\nu_{G_1}=i_{\bold v}^*\nu$. Let $\bold v\tt \bold e$ be the $\OKp$-basis 
$v_i\tt e_j$ for $L_p^n = L_p\tt_{\Z_p}\Z_p^n$, where  $e_j$  is the  standard basis of $\Z_p^n$. 
Then, for  a $\Z_p$-basis  $1, \xi_p$ for $\OKp$, $\bold v\tt\bold e,\xi_p\bold v\tt\bold e$ is a $\Z_p$-basis for $L_p^n$.  
Let $\bold w = [w_1, \dots, w_{n^2}]$ for a $\Z_p$-basis $w_j$ for $\Herm_n(O_{\smallkay,p})$. 
Then
$$\vol(K_p,dg) = \frac{|\a(\bold v\tt\bold e,\xi_p\bold v\tt\bold e)|_p}{|\b(\bold w)|_p}\cdot |\Delta|_p^{\frac12n(n-1)}\cdot\a_p(S,S),$$
 \hfb
(iii) 
$$\frac{\vol(K'_p,dg)}{\vol(K_p,dg)} = \frac{ |\det(S')|_p^n\,\a_p(S',S')}{|\det(S)|_p^n\,\a_p(S,S)} $$
\end{lem}

\begin{proof}
To prove (i), take $\xx\in h^{-1}(S) \cap L_p^n$, and write $\xx = g\cdot \bold v = \bold v\cdot a$, 
with $g\in G_1(\Q_p)$ and $a\in \GL_n(\kay_p)\cap M_n(\OKp)$. Let $L'_p\subset L_p$ be the $\OKp$-lattice 
spanned by the components of $\xx$. Then $L'_p=g\cdot L_p$ and $|L_p:L'_p| = |N(\det a)|_p^{-1}$. 
But 
$$S=(\xx,\xx) = (\bold v a,\bold v a) = {}^ta(\bold v,\bold v)\bar a = {}^ta S\bar a,$$
so that $N(\det a)=1$ and $gL_p= L_p$, i.e., $g\in K_p$.  Conversely, if $k\in K_p$, then 
$\xx = k\cdot \bold v \in h^{-1}(S)\cap L_p^n$. 

To prove (ii), note that by (i),
$$\vol(K_p,dg) = O_{S,p}(\ph_p),$$
where $\ph_p$ is the characteristic function of $L_p^n$. 
Thus, 
\begin{align}\label{temp.volK}
\vol(K_p,dg)& = C_p(V,\a,\b,\psi)^{-1} \,W_{S,p}(0,\P_{\ph_p}) \notag\\
\nass
{}&= C_p(V,\a,\b,\psi)^{-1} \,\gamma_p(V)^{n}\, 
|\det S|_p^n\,|\Delta|_p^{e}\, \a_p(S,S)\\
\nass
{}&= \frac{c_p(\a,\psi)}{c_p(\b,\psi)}\, 
|\det S|_p^n\,|\Delta|_p^{e}\, \a_p(S,S).\notag
\end{align}
For an $\OKp$-basis of $L_p$, $\bold v$  and a $\Z_p$-basis $1, \xi_p$ for $\OKp$, write $\bold v \tt \bold e$ for the $\OKp$-basis 
$v_i\tt e_j$ for $L_p^n= L_p\tt \Z_p^n$, and note that $\bold v\tt\bold e, \xi_p\bold v\tt\bold e$ is a $\Z_p$-basis for $L_p^n$. 
Then 
$$\vol(L_p^n,d_{\a,p}x) =  |\a(\bold v\tt\bold e,\xi_p\,\bold v\tt\bold e)|_p = c_p(\a,\psi)\,|\det S|^n_p|\Delta|^{\frac12n^2}_p,$$
so that 
$$c_p(\a,\psi) =|\a(\bold v\tt\bold e,\xi_p\,\bold v\tt\bold e)|_p\,|\det S|^{-n}_p|\Delta|^{-\frac12n^2}_p.$$
Similarly, if $\bold w= [w_1,\dots, w_{n^2}]$ is a $\Z_p$-basis for $\Herm_n(\OKp)$, then 
$$\vol(\Herm_n(\OKp),d_{\b,p}b) = |\b(\bold w)|_p = c_p(\b,\psi) \,|\Delta|_p^{\frac14n(n-1)}.$$
Using these expressions in (\ref{temp.volK}), we obtain (ii). 

Finally, to prove (iii), note that there exists a $\kay$-linear isomorphism $\gamma: V'(\bar\Q) \isoarrow V(\bar\Q)$ of $\bar\Q$-vector spaces
such that $(\gamma(x), \gamma(y))_{V} = (x,y)_{V'}$, for all $x$, $y\in V'(\bar\Q)$. 
This isomorphism induces an isomorphism $\text{\rm Ad}(\gamma): G'_1 \isoarrow G_1$, $g'\mapsto \gamma g' \gamma^{-1}$  
of algebraic groups  over $\bar\Q$.
Note that the function $\s\mapsto \psi_\s = \s(\gamma)\circ \gamma^{-1}\in G_1(\bar\Q)$ 
on $\Gal(\bar\Q/\Q)$ defines the $1$-cocycle relating the inner forms $G_1$ and $G'_1$. 

Let $\a' = \gamma^*(\a)$ be the top degree form on $(V')^n$ obtained by pulling back $\a$.  Since $G_1(\bar Q)$ acts trivially on 
$\wedge^{\text{\rm top}}(V)$, this form is rational over $\Q$. Moreover,  the gauge form $\nu_{G'_1}$ on $G'_1$ determined by $\a'$ and $\b$ 
is the pullback via $\text{\rm Ad}(\gamma)$ of the gauge form $\nu_{G_1}$ on $G_1$ determined by $\a$ and $\b$. 

With respect to the $\kay$-bases $\bold v$ and $\bold v'$ 
there exists a matrix $g=a+\delta b$ with $a$, $b\in M_n(\bar\Q)$ such that 
$$\gamma(\bold v') = \bold v \cdot g.$$
Then 
$$S' = (\bold v',\bold v')_{V'} = (\gamma(\bold v'),\gamma(\bold v'))_V = {}^tg (\bold v,\bold v)_V \bar g = {}^tg S\bar g,$$ 
where $\bar g= a-\delta b$.  In particular, 
$$N(\det(g)) = \det(S')\,\det(S)^{-1}.$$
On the other hand,  
$$\a'(\bold v'\tt\bold e,\xi_p\bold v'\tt\bold e) =
\a(\gamma(\bold v')\tt\bold e,\xi_p\gamma(\bold v')\tt\bold e) = N(\det(g))^n\,\a(\bold v\tt \bold e,\xi_p\bold v\tt\bold e).$$
Using these in the ratio of the expressions in (ii), we obtain (iii). 
\end{proof}

\begin{proof}[Proof of Proposition~\ref{S''T''}]  Let $V''_p$ be the $2$-dimensional hermitian space over $\kay_p$ with $\OKp$-lattice $L''_p$ spanned by 
the components of $\bold v''= [v''_1,v''_2]$ with $(\bold v'', \bold v'') = S''$.  
Let $\a''$ (resp. $\b''$)  be a top degree translation invariant form on $(V''_p)^2$ ( resp. $\Herm_2$), and let 
$\nu'' = \a''/h^*(\b'')$ be the corresponding $4$-form on $(V''_p)^2$, as above. 
Let 
$$X= \{\bold x\in (L_p)^2\mid {}^t\bold x'' S'' \bar{\bold x}'' = T\ \}.$$
Then 
$$\a_p(S'',T'') = \vol(X,d\nu'')\cdot \frac{|\b''(\bold w'')|_p}{|\a''(\bold v''\tt \bold e, \xi_p \bold v''\tt \bold e)|_p},$$
where $\bold w'' = [w_1,\dots, w_4]$ is a $\Z_p$-basis for $\Herm_2(\OKp)$.  Note that we are now assuming that $\kay_p/\Q_p$ is unramified. 
In particular, the space $V''_p$ is anisotropic. It is then easy to check that, for any $T=\diag(p^a,p^b)$ with $a$, $b\in \Z_{\ge0}$ with $a+b$ odd, 
the set $X$ is non-empty and that  
$$X = G''_p\cdot \bold x''_0\lisoarrow G''_p$$ 
for any $\bold x''_0\in X$, where $G''_p= U(V''_p)$. Thus, cf. \cite{gan.yu}, p.500, for example, 
$$\a_p(S'',T'') = \vol(G''_p,d\nu_{G''_p})\cdot \frac{|\b''(\bold w'')|_p}{|\a''(\bold v''\tt \bold e, \xi_p \bold v''\tt \bold e)|_p},$$
where, as explained above, the gauge for $\nu_{G''_p}$ is independent of $\bold x''_0\in (V_p'')^2$. 
In particular, $\a_p(S'',T'') = \a_p(S'',S'')$, as claimed. 
\end{proof}

\begin{remark}  Although we do not need it, the value 
$$\a_p(S',S') = p\cdot (1+p^{-1}) \prod_{r=1}^{n-1} ( 1- (-1)^r p^{-r})$$
is obtained from the general formula in \cite{gan.yu}, Theorem 7.3.
\end{remark}

\vskip.2in

\centerline{\bf \large Part IV. Arithmetic intersection numbers}

From now on, we restrict to the case of signature $(n-1,1)$. Accordingly, we abbreviate $\M(n-1,1)$ to $\M(n)$, 
and we write $\M$ for the base change $\M(n)\times_{Spec(\OK)} \M_0$ of $\M(n)$ to $\Cal M_0$. 

\section{The main theorem and a conjecture}\label{sectionintersection}

We first recall from the introduction the following definition.

\begin{defn} A hermitian matrix  
$T\in \Herm_m(\OK)_{>0}$ (and the corresponding special cycle $\ZZ(T)$) is called {\it non-degenerate} if
$\ZZ(T)$ is of pure dimension $n-m$.  \hfb
\noindent In particular, in the special case $m=1$ any $T=t\in \Z_{>0}$ is non-degenerate.
\end{defn}
 The following proposition gives a partial characterization of non-degenerate $T\in  \Herm_n(\OK)_{>0}$; it follows from \cite{KRunitary.I} and Proposition 
 \ref{padic.uni.cycles} above.  Recall that for a nonsingular $T$ in $\Herm_n(\OK)$,  $\Diff_0(T)$ is the set of inert primes $p$ such that 
$\ord_p\det(T)$ is odd.
\begin{prop}
For $T\in \Herm_n(\OK)_{>0}$, suppose that  $\ZZ(T)\neq\emptyset$. 

(i) Suppose that $\Diff_0(T)=\{p\}$ for $p>2$. Let $r^0(T)=n-{\rm rank}(\red(T))$ be the dimension of the radical of the hermitian form $\red(T)$ over $\OK/p\OK$. Then $\ZZ(T)$ is equidimensional of dimension 
$$
\dim \ZZ(T)=\left[\frac{r^0(T)-1}2\right]. 
$$
In particular,  $T$ is non-degenerate if and only if $T$ is $\GL_n(\OKp)$-equivalent to
 $\diag(1_{n-2},p^a,p^b)$ with $0\le a<b$.
 
(ii) Conversely, suppose that $T\in \Herm_n(\OK)_{>0}$ and that,  for an odd unramified prime $p$,
 $\ZZ(T)\cap \M_p$ 
is a nonempty $0$-cycle. Then $T$ is non-degenerate with $\Diff_0(T)=\{p\}$. 
\end{prop} 
\begin{proof}
By Proposition~\ref{support.ZT}, the cycle $\ZZ(T)$ is empty unless
$|\Diff_0(T)|\le 1$.  When $|\Diff_0(T)|=0$, then the cycle $\ZZ(T)$ is supported 
in the fibers at ramified primes. Thus, under either of the hypotheses in (i) or (ii),  
it follows that $p$ is an odd inert  prime and that $\Diff_0(T) =\{p\}$.
Moreover, by Remark \ref{thirdsp},   
$\supp(\ZZ(T))$ is contained in the  component $\M^{\tilde V,\text{\rm ss}}_p$, 
where $\tilde V\in \Cal R_{(n-1, 1)}(\kay)$   is locally  isomorphic to $V_T$ at all finite places away from $p$.
Now the completion of $\ZZ(T)$ along its supersingular locus breaks up into a disjoint sum 
(\ref{ZTdecompo}) of $\widehat{\ZZ}^{(V^\sh,V_0), \ss}(T)$, where the pairs $(V^\sh, V_0)$ run over 
$\Cal R_{(n-1, 1)}(\kay)^\sh\times \Cal R_{(1, 0)}(\kay)$ such that $\Hom_\smallkay(V_0, V)\simeq \tilde V$. 
By Proposition~\ref{padic.uni.cycles}, $\widehat{\ZZ}^{(V^\sh,V_0), \ss}(T)$ is the stack quotient of the 
formal scheme
\begin{equation}\label{Iset}
{\rm Inc}_p(T; V^\sh, V_0) = \coprod_{\substack{(g K^{V^\sh, p}, g_0 K_0^p)}}
\coprod_{\xx^o} \ZZ(\und{\und{\xx}}^o),
\end{equation}
by the group $I^V(\Q)\times I^{V_0}(\Q)$. 
It is shown in \cite{KRunitary.I} that 
$\ZZ(\und{\und{\xx}}^o)$ 
is equidimensional of dimension $[(r^0(T)-1)/2]$, and hence is a $0$-cycle if and only if 
$T$ is $\GL_n(\OKp)$-equivalent to $\diag(1_{n-2},p^a,p^b)$ with $0\le a<b$. 
\end{proof}

\begin{remark} To obtain a complete characterization of non-degenerate $T$'s of maximal size it would be necessary to determine what 
happens for ramified primes and for $p=2$, e.g., to extend the results of 
\cite{vollaard}, \cite{vollwedhorn} and \cite{KRunitary.I} to such primes. 
\end{remark}

For positive integers $m_1, \ldots , m_r$ with $\sum\nolimits_{i=1}^r m_i=n$,  we fix $T_i\in \Herm_{m_i}(\OK)$, with 
corresponding special cycles $\ZZ(T_i)$. The fiber product of these cycles over $\M$ decomposes as 
\begin{equation}
\ZZ(T_1)\times_\M\times\ldots\times_\M\ZZ(T_r)=\coprod_T 
\ZZ(T),
\end{equation}
where $T$ runs over elements of  $\Herm_n(\OK)_{\ge 0}$ with diagonal blocks $T_1,\ldots,T_r$. The decomposition is according
 to the {\it fundamental matrix} which, to an $S$-valued point $(A,\iota,\l; E, \iota_0, \l_0; {\bf x}_1,\ldots,{\bf x}_r)$
 for a connected scheme $S$,  attaches the matrix $T = h'(\xx,\xx)$,  where $\xx = [\xx_1,\dots,\xx_r]$.  
 Here $h'$ is the hermitian form (\ref{fundherm}) on $\Hom_{\OK}(E, A)$. 

\begin{defn}
For  $T\in \Herm_n(\OK)_{>0}$ with non-degenerate diagonal blocks $T_1,\ldots,T_r$, let  
\begin{equation}\label{EPT}
\langle \ZZ(T_1), \ldots, \ZZ(T_r)\rangle_T=\sum_p\chi \big(\ZZ(T)_p, 
\Cal O_{\ZZ(T_1)}\tt^{\mathbb L}\ldots\tt^{\mathbb L}\Cal O_{\ZZ(T_r)}\big)\log p. 
\end{equation}
\end{defn}
Here $\ZZ(T)_p$ denotes the part of $\ZZ(T)$ with support in the fiber at $p$. Note that by Proposition  \ref{support.ZT}, 
$\ZZ(T)$ has proper support in the special fiber of at most finitely many primes $p$, so that the sum and  
the Euler-Poincar\'e characteristics  appearing here are finite.  
Note that, since $\ZZ(T)$ is a stack and not a scheme, we have to use here the
`stacky' definition of the Euler-Poincar\'e characteristic, cf.~\cite{deligne.rapoport}, VI~4.1.  
\begin{remark}\label{autond}
Note that, if $T\in \Herm_n(\OK)_{>0}$ is non-degenerate,
then the diagonal blocks $T_1,\ldots, T_r$ are automatically non-degenerate. 
\end{remark}

 By  
Proposition  \ref{support.ZT}, if  $|\Diff_0(T)|>1$, the cycle $\ZZ(T)$ is empty and the   Euler-Poincar\'e characteristic is zero. 
On the other hand, if $\Diff_0(T)=\{p\}$, then $\ZZ(T)$ is supported in the fiber over $p$ of $\M$, and  
\begin{equation}
\langle \ZZ(T_1), \ldots, \ZZ(T_r)\rangle_T=\chi \big(\ZZ(T), \Cal O_{\ZZ(T_1)}\tt^{\mathbb L}\ldots\tt^{\mathbb L}\Cal O_{\ZZ(T_r)}\big)\log p. 
\end{equation}
Finally, if $\Diff_0(T)$ is empty, then $\ZZ(T)$ is either empty or supported in the fibers for $p$ ramified in $\kay$, 
and the sum in (\ref{EPT}) runs over such primes. 

We expect  that $\langle \ZZ(T_1), \ldots, \ZZ(T_r)\rangle_T$ is the contribution of $T$ to the 
arithmetic intersection number of the non-degenerate special cycles $\ZZ(T_1), \ldots \ZZ(T_r)$.

\begin{prop}
Let $T\in \Herm_n(\OK)_{>0}$ with diagonal blocks $T_1,\ldots,T_r$  and assume that $T$ is non-degenerate. Then  
$$\Cal O_{\ZZ(T_1)}\tt^{\mathbb L}\ldots\tt^{\mathbb L}\Cal O_{\ZZ(T_r)}=\Cal O_{\ZZ(T_1)}\tt\ldots\tt\Cal O_{\ZZ(T_r)};$$
more precisely, $\Cal O_{\ZZ(T_1)}\tt\ldots\tt\Cal O_{\ZZ(T_r)}$ represents the left hand side in the derived category. 
\end{prop}
\begin{proof}
We use the $p$-adic uniformization, Proposition   \ref{padic.uni.cycles} of special cycles.  We are then reduced to 
the corresponding statement concerning closed formal subschemes of $\Cal N\simeq\Cal N\times \Cal N_0$: we are given
$$\ZZ({\und{\und \xx}}_i)\subset \Cal N, \quad  {\und{\und \xx}}_i\in \Hom_{\OK\tt\Z_p}(\X_0, \X)^{n_i}, 
$$ for $i=1,\ldots,r$. By assumption $\cap_{i=1,\ldots,r}\ZZ({\und{\und \xx}}_i)$ is of dimension $0$ (and in fact, by \cite{KRunitary.I}, 
Corollary 4.7  reduced to a single point $\xi$). But by Proposition 3.5  of \cite{KRunitary.I}, for any  $x\in \Hom_{\OK\tt\Z_p}(\X_0, \X)$, 
the formal subscheme $\ZZ(x)$ of $\Cal N$ is defined by a single equation. It follows that, denoting by $x_i^j \ (j=1,\ldots,n_i)$ the 
components of ${\und{\und \xx}}_i$, each $\ZZ(x_i^j)$ is a formal divisor in $\Cal N$, and $\ZZ({\und{\und \xx}}_i)$ is the 
proper intersection of these divisors. Hence the local equations $f_i^j\  (j=1,\ldots, n_i)$ of the divisors $\ZZ(x_i^j)$ form 
a regular sequence in the regular local ring ${\Cal O}_{\Cal N, \xi}$. This implies that ${\Cal O}_{\ZZ({\und{\und \xx}}_i)}$ represents the derived tensor product 
$$
\Cal O_{\ZZ(x_i^1)}\tt^{\mathbb L}\ldots\tt^{\mathbb L}\Cal O_{\ZZ(x_i^{n_i})}. 
$$
Indeed,  each ${\Cal O}_{\ZZ(x_i^j)}$ is represented by the complex
$$
[{\Cal O}\overset{f_i^j}{\lra}  {\Cal O}]. 
$$
Hence the derived tensor product of all ${\Cal O}_{\ZZ(x_i^j)}$ is represented by the Koszul complex 
\cite{Bourb} of the elements $f_i^j \ ( j=1,\ldots n_i)$, which, since the $f_i^j$ form a regular sequence,  has as its only homology group the module 
${\Cal O}_{\ZZ({\und{\und \xx}}_i)}$. \hfb
Similarly, since all formal divisors $\ZZ(x_i^j)\ (i=1,\ldots,r; j=1,\ldots , n_i)$ intersect properly, an analogous reasoning 
proves that all $f_i^j$ form a regular sequence in ${\Cal O}_{\Cal N, \xi}$ and hence the derived tensor product of 
all ${\Cal O}_{\ZZ(x_i^j)}$ is represented by the Koszul complex  of the elements $f_i^j$, which  has as its only 
homology group the module ${\Cal O}_{\cap \ZZ(x_i^j)}$. 
\end{proof}

\begin{cor} 
For a non-degenerate $T\in \Herm_n(\OK)_{>0}$ with diagonal blocks $T_1,\ldots,T_r$, 
$$\langle \ZZ(T_1), \ldots, \ZZ(T_r)\rangle_T = \text{\rm length}(\ZZ(T)) \cdot \log p =: \degh(\ZZ(T))$$
is the arithmetic degree of the $0$-cycle $\ZZ(T)$. In particular, this quantity is independent 
of the choice of the sizes of the blocks $T_i$ on the diagonal of $T$.  
\end{cor}
Note that the blocks $T_i$ here are automatically nondegenerate, cf.~Remark~\ref{autond}.

The proof of the following explicit formula for the length will be given in the next section. 
Before stating the formula, we recall that, for $T\in \Herm_n(\OK)_{>0}$ with $\Diff_0(T)= \{p\}$, 
an $\OK$-lattice $M$  in the positive definite hermitian space  $V_T$
determined by $T$ is called {\it nearly self-dual} if $M\subset M^*$ with $M^*/M \simeq \OK/p\OK$. 
If $M$ is such a lattice, then $r_{\rm gen}(T,M)$, the representation number of $T$ by the $\UU(V_T)$-genus of $M$, 
is defined by  (\ref{r-gen}). Finally, let 
$$r_{\rm gen}(T,V_T) = \sum_{[[M]]} r_{\rm gen}(T,M),$$
 In the sum, 
 $[[M]]$ runs over the genera of nearly self-dual lattices in $V_T$. 

\begin{theo}\label{length}
Let $T\in \Herm_n(\OK)_{>0}$ be non-degenerate with $\ZZ(T)\neq\emptyset$ and with $\Diff_0(T)=\{p\}$ for some  $p>2$. Then 
$$
\text{\rm length}(\ZZ(T))=\mu_p(T)\cdot \frac{h_{\smallkay}}{w_{\smallkay}}\cdot r_{\rm gen}(T,V_T),
$$
with 
$$\mu_p(T) = \frac12 \sum_{l=0}^a p^l(a+b+1-2l),$$
where $T$ is $\GL_n(\OKp)$-equivalent to
 $\diag(1_{n-2},p^a,p^b)$ with $0\le a<b$.
\end{theo}

Note that the factor $\frac{h_{\smallkay}}{w_{\smallkay}}$ is the degree of the stack $\M_0 = \big[\,\OK^\times \back \Spec O_H\,\big]$ 
over $\Spec \OK$. 

We may now formulate our main result.  For each self-dual lattice $L$ in a relevant hermitian space $V\in \Cal R_{(n-1, 1)}(\kay)$  
there is an incoherent Eisenstein series $E(z,s,L)$ defined by (\ref{semi.classical}); it depends only on the $G_1^V$-genus $\LLL$ of $L$. 
Let 
\begin{equation}\label{Eis.avg}
E(z,s,V) = \sum_{\LLL}  E(z,s,L),
\end{equation}
be the sum of these series over the $G^V_1$-genera of self-dual lattices in $V$, cf. Corollary~\ref{global.genera}. 
We also consider the Eisenstein series
$$E(z,s) = \sum_{V\in \Cal R_{(n-r,r)}(\smallkay)} E(z,s,V)$$
obtained by summing over all self-dual genera. 
\begin{theo}\label{main}
Let $T\in \Herm_n(\OK)_{>0}$ be non-degenerate with $\ZZ(T)\neq\emptyset$ and with $\Diff_0(T)=\{p\}$ for some  $p>2$. 
Let $V_T$ be the positive definite hermitian space of dimension $n$ determined by the matrix $T$.  
Then
$$E'_T(z,0) = E'_T(z, 0, V)=C_1\cdot\degh(\ZZ(T))\cdot q^T,$$ 
where 
$V\in \Cal R_{(n-1, 1)}(\kay)$ is the unique relevant hermitian space that is locally 
isomorphic to $V_T$ at all primes other than
$\infty$ and $p$, 
and $E(z, s, V)$ 
is the corresponding incoherent Eisenstein series defined by (\ref{Eis.avg}).
Here 
$$C_1 =  (-1)^n\frac{w_{\smallkay}}{h_{\smallkay}}\,\vol(G_1^{V_T}(\R))\,\vol(K_1),$$
for $K_1\subset G_1^V(\A_f)$ the stabilizer of a self-dual $\OK$-lattice $L$ in $V$ 
and for the measure normalized as in sections \ref{sectionthetaint} and \ref{sectionSiegelfor}.
\end{theo}
Note that the factor $\vol(G_1^{V_T}(\R))\,\vol(K_1)$ is a (stacky) Euler characteristic, cf.\ Lemma~\ref{EPconstant}.
\begin{proof} This result is an immediate consequence of the formulas for the two sides 
given in Corollary~\ref{FC.formula} and 
Theorem~\ref{length} respectively. 
\end{proof}

Note that this identity can be rephrased as saying that, for $T$  as in the theorem, 
$$E'_T(z, 0, V)=C_1\cdot \langle \ZZ(T_1), \ldots, \ZZ(T_r)\rangle_T \, q^T,$$
where $T_1, \dots, T_r$ are diagonal blocks of $T$, as above. 

We have the following conjecture. 
\begin{conj} \label{conjint}
 Let $T\in \Herm_n(\OK)_{>0}$ with non-degenerate diagonal blocks $T_1,\ldots,T_r$  and 
assume that $\Diff_0(T)=\{p\}$ for $p>2$. Suppose that $\ZZ(T)\neq\emptyset$. 
Let $V_T$,  $V$, and 
$E(z, s,V)$ be as in Theorem~\ref{main}. Then
$$E'_T(z,0)=E'_T(z, 0,V)=C_1\cdot \langle \ZZ(T_1), \ldots, \ZZ(T_r)\rangle_T\cdot q^T.$$ 
\end{conj}
 
The point here is that, when the cycle $\ZZ(T)$ has positive dimension, there is no intrinsic 
definition of  an arithmetic degree $\degh\,\ZZ(T)$. Instead, as in Conjecture \ref{conjint}, we assign to $T$  the 
following quantity. Let $t_1, \ldots , t_n\in \Z_{>0}$ be the diagonal entries of $T$. Then 
consider 
\begin{equation}\label{arithdef}
\degh\,\ZZ(T):=\langle \ZZ(t_1), \ldots, \ZZ(t_r)\rangle_T .
\end{equation}
The conjecture includes the assertion that $\langle \ZZ(T_1), \ldots, \ZZ(T_r)\rangle_T$ should always 
be equal to this quantity,  no matter what ordered partition of $n$ is used to divide $T$ up into blocks, 
as long as the blocks  $T_1,\ldots , T_r$ on the diagonal are non-degenerate.
 Also note that due to the invariance property of 
the Fourier coefficients $E'_T(z, 0, V)$ under the action of $\GL_n(\OK)$
coming from the transformation law, the conjecture also implies that the following invariance property 
of the arithmetic intersection numbers (\ref{arithdef}) should hold. If $T'= gT^t\!{\bar g}$ for some $g\in \GL_n(\OK)$, then $\degh\,\ZZ(T)=\degh\,\ZZ(T')$.
The picture is consistent with what is proved in \cite{KRinventiones} in the case of degenerate intersections of 
special divisors on Shimura curves
and in \cite{terstiege} and \cite{terstiege2} in the case of degenerate intersections of special divisors on Hilbert modular surfaces.

\section{The arithmetic degree in the non-degenerate case}\label{sectionarith.degree}

In this section we prove Theorem~\ref{length}. Recall 
that there is a decomposition
\begin{equation}\label{ZThatdecompo}
\widehat{\ZZ}(T) = \coprod_{(V^\sh,V_0)} \widehat{\ZZ}^{(V^\sh,V_0),{\rm ss}}(T).
\end{equation}
We use the $p$-adic uniformization of the special cycle $\widehat{\ZZ}^{(V^\sh,V_0),{\rm ss}}(T)$ given in 
Proposition~\ref{padic.uni.cycles}. 
As in \cite{KRunitary.I},  we identify $\ZZ(\uuxx^o)$ with a closed formal subscheme of $\N$.
By Theorem~5.1 of \cite{KRunitary.I},  the length of $\ZZ(\und{\und{\xx}}^o)$ depends only on $T$ and is equal to 
\begin{equation}\label{lengthzx}
\text{\rm length}( \ZZ(\und{\und{\xx}}^o)) = \mu_p(T) = \frac12\sum_{l=0}^a p^l\,(a+b+1-2l).
\end{equation}
 Therefore, we have 
\begin{equation}\label{ardeg}
\text{\rm length}(\widehat{\ZZ}^{(V^\sh,V_0),{\rm ss}}(T)) = \mu_p(T)\cdot |\big[(I^V(\Q)\times I^{V_0}(\Q))\back {\rm Inc}_p(T;V^\sh,V_0)(\F)\big]|,
\end{equation}
where
the second factor on the right side is the (stack) cardinality of the quotient. 

Recall from Lemma \ref{unitaryv'}, a) and Remark~\ref{thirdsp} that the hermitian space $\tilde V' = \Hom_{\OK}^0(E^o,A^o)$ is positive definite
and has 
invariants
$\inv_p(\tilde V')=-\inv_p(\tilde V) = -1$, and $\inv_\ell(\tilde V')=\inv_\ell(\tilde V)$ for $\ell\neq p$, 
where we recall that $\tilde V = \Hom_{\smallkay}(V_0,V)$.  
Let $G' = \GU(\tilde V')$ and $G'_1=\UU(\tilde V')$, and recall from Lemma \ref{unitaryv'}, b) that $I^V(\Q) \simeq G'_1(\Q)$
and $I^{V_0}(\Q) = \kay^1$.   
 
We also recall that, by Theorem~4.5 of 
\cite{KRunitary.I},
\begin{equation}\label{point.component}
\ZZ(\uuxx^o)(\F) = \Cal V(\L)(\F),
\end{equation}
is a single point, where $\L$ is the unique vertex of level $0$ and type $1$ 
in\footnote{Here we are rephrasing the result of \cite{KRunitary.I} in terms of $\tilde V'_p$. 
Note that, by Lemma~3.9 in \cite{KRunitary.I},  the space $C$, $\{\,\ \}$ coincides with the space $\mathbb V = \tilde V'_p$ but with the 
hermitian form scaled by $p$, i.e., $\{\ ,\ \} = p\,(\ ,\ )$. Since $\ord_p\det C$ is $0$ if $n$ is odd and $1$ if $n$ is even, 
we have $\inv_p(\tilde V') = -1$, as required. }\,   $\tilde V'_p$ such that $\L^*$ contains the components of  $\uuxx^o$
and where $\Cal V(\L)$ is the stratum associated to $\L$ in $\Cal N$. 
Recall here from \cite{KRunitary.I} that a {\it vertex  of level $j$ and type $t$} is a lattice $\L$ in $\tilde V'_p$ such that 
\begin{equation}\label{defvertex}
p^{j}\L^* \overset{t}{\subset} \L \subset p^{j-1}\L^*,
\end{equation}
where\footnote{In \cite{KRunitary.I}, this relation is expressed in terms of 
$$p^{-1}\L^* =\L^\vee = \{\, x\in C\mid \{\L,x\}\subset \OKp\ \}.$$
}
$$\L^* = \{\, x\in \tilde V'_p\mid (\L,x) \subset \OKp\ \},$$
and where the notation $p^{j}\L^* \overset{t}{\subset} \L$ means that $\L/p^j\L^* \simeq (\OK/p\OK)^t$. 
Also note that, if $g\in G'(\Q_p)$ with $\ord_p\nu(g)=r$, then 
$\nu(g)(g\L)^* = g(\L^*)$ and hence $g\L$ has level $j+r$ and type $t$. In particular the group $G'(\Q_p)^0$ (the subgroup of 
elements with scale factor a $p$-adic unit) acts on the set of lattices of level $0$ and type $t$.  The following fact is 
easily checked, cf. \cite{jacobowitz}, section 7.

\begin{lem} Let $\Cal L_{t}$ be the set of lattices in $\tilde V'_p$ of level $0$ and type $t$. Then $G'(\Q_p)^0$ acts transitively on $\Cal L_t$.\qed
\end{lem}

We fix a lattice $\L\in \Cal L_{1}$ and let $K'_p$ be its stabilizer in $G'(\Q_p)$. Note that $K'_p\subset G'(\Q_p)^0$. 

Recall that $K^p$ is the stabilizer of the $\widehat{\Z}^p$-lattice $\Hom_{\OK\tt\widehat{\Z}^p}(T^p(E^o), T^p(A^o))$, where we write
$K^p=K^{V^\sh, p}$ to simplify the notation. 
Thus, we must compute the stack  cardinality of the quotient of 
\begin{equation}\label{incF}
{\rm Inc}_p(T;V^\sh,V_0)(\F)=\coprod_{\substack{(gK^p, \, g_0K_0^p)}}\ 
\coprod_{g'_pK'_p\in  G'(\Q_p)^0/K'_p }\ \coprod_{\xx^o}\ \{{\rm pt}\}
\end{equation}
by the action of $G'_1(\Q)\times I^{V_0}(\Q)$.  Here, the last union runs over the set of $\xx^o\in V'(\Q)^n$ such that 
\begin{enumerate}
\item[(i)] $h'(\xx^o,\xx^o) = T$,
\item[(ii)] 
$g^{-1}\circ\und{\xx}^o\circ g_0\in \big(\Hom_{\OK\tt\widehat{\Z}^p}(T^p(E^o), T^p(A^o))\big)^n,$
\item[(iii)] $\uuxx^o\in g'_p\L^*$.
\end{enumerate}

The combination of (\ref{trivV}) and  (\ref{trivV'}) determines an isometry $\tilde V(\A_f^p)\simeq \tilde V'(\A_f^p)$ 
compatible with the action of $I^V(\Q)$ on both sides. 
Define a lattice $\tilde L'$ in $\tilde V'$ by taking $\tilde L'_p = \L^*$ and
\begin{equation*}\label{ident}
\tilde L'\tt\widehat{\Z}^p = \Hom_{\OK\tt\widehat{\Z}^p}(T^p(E^o), T^p(A^o)). 
\end{equation*} 
The stabilizer of $\tilde L'\otimes\widehat{\Z}$ in $G'(\A_f)$ is  $K'=K'_pK'^p$, where $K'^p=K^p$ 
under the identification of $V(\A_f^p)$ with $V'(\A_f^p)$. Note that this lattice is nearly self-dual, i.e., 
\begin{equation}\label{almostsd}
(\tilde L')^*/\tilde L' \simeq \OK/p\OK.
\end{equation}
Next observe that 
$$\nu(K') = \nu(G'(\A_f)^0),$$
so that 
$$G'(\A_f)^0/K' \simeq G'_1(\A_f)/K'_1$$
for $K'_1 = K'\cap G'_1(\A_f)$. 
The cosets $g'K'_1$ in $G'_1(\A_f)/K'_1$ correspond to the lattices $\tilde L''= g'\tilde L'$ in the $G'_1$-genus of $\tilde L'$, via 
$$ 
\tilde L''=\tilde V'\cap \big(g'\cdot (\tilde L'\tt\widehat{\Z})\big). 
$$
For any such lattice, $\tilde L''\subset \tilde V'$, corresponding to $g'K'_1$, 
we  let
$$\O(T,\tilde L'') = \{\, \xx\in (\tilde L'')^n\ \mid\ h'(\xx,\xx) =T\,\}$$
and 
$$\Gamma(\tilde L'') = G'_1(\Q)\cap g' K'_1 (g')^{-1}.$$
Note that this is a finite group.

Observe  that since the components of any $\xx\in \O(T,\tilde L'')$ span $\tilde V'$, 
the stabilizer $\Gamma(\tilde L'')_{\xx}$ is trivial. This implies that 
$$\big[\, G'_1(\Q)\back \coprod_{\substack{gK^p}}\ 
\coprod_{g'_pK'_p }\ \coprod_{\xx^o}\ \{{\rm pt}\}\,\big]  = \coprod_{\tilde L''} \Gamma(\tilde L'')\back \O(T, \tilde L'')$$
as orbifold quotients. Here, on the right hand side, $\tilde L''$ runs over the classes of lattices in the $G'_1$-genus of $\tilde L'$.  Thus
$$\big|\, \big[\,G'_1(\Q)\back \coprod_{\substack{gK^p}}\ 
\coprod_{g'_pK'_p }\ \coprod_{\xx^o}\ \{{\rm pt}\}\,\big]\,\big| = \sum_{\tilde L''} |\Gamma(\tilde L'')|^{-1}\,| \O(T, \tilde L'')\,|=
 r_{\text{\rm gen}}(T,\widetilde{L}').$$

This is not yet the cardinality of the stack quotient of (\ref{incF}); we have to take into account the role of the 
cosets $g_0K^p_0$ and the action of $I^{V_0}(\Q)$.  Here $g_0K^p_0$ runs over $G^{V_0}(\A_f^p)^0/K_0^p \simeq G^{V_0}_1(\A_f^p)/K_{0,1}^p$. 
For any $g_0\in G^{V_0}_1(\A_f^p) = \kay^1_{\A_f^p}$, as the lattice $\tilde L''$ runs over a set of representatives for the 
$G'_1$-genus of $\tilde L'$,  so does the lattice $\tilde L'' g_0^{-1}$. 
Note that 
$$\big| \,\big[\, I^{V_0}(\Q)\back G^{V_0}_1(\A_f^p)/K^p_{0,1}\,\big]\, \big| =\frac{1}{w_\smallkay}\cdot |\kay^1\back \kay^1_{\A_f}/\widehat{O}_{\smallkay}^{\times}|
=\frac{h_{\smallkay}}{2^{\delta-1}\,w_\smallkay}.$$

Thus, we obtain the following  explicit formula for the stack cardinality.  
\begin{lem}\label{gen-lem}
$$ \big|\,\big[\,(I^V(\Q)\times I^{V_0}(\Q))\back {\rm Inc}_p(T;V^\sh,V_0)(\F)\,\big]\,\big|= \frac{h_{\smallkay}}{2^{\delta-1}\,
w_\smallkay}\cdot r_{\text{\rm gen}}(T,\tilde L').$$
\end{lem} 
Note that the right side here depends only on $\tilde V = \Hom_{\smallkay}(V_0,V)$
and the type of the genus in $V^\sh= (V,\LLL)$.  Moreover, $r_{\text{\rm gen}}(T,\tilde L')$ vanishes unless
$\tilde V$ represents $T$, i.e., $\tilde V\simeq V_T$. 
The number of pairs $(V,V_0)$ occurring in the decomposition 
(\ref{ZThatdecompo}) 
such that $\tilde V = \Hom_{\smallkay}(V_0,V) \simeq V_T$ 
is $2^{\delta-1}$. 
As we vary the choice of $\LLL$ in $V$ for a given such pair $(V,V_0)$, $\tilde L'$ runs over the $G_1^{\tilde V}$-genera 
of nearly self-dual lattices in $\tilde V$. 
Taking these facts into account,  Theorem~\ref{length} then follows from Lemma~\ref{gen-lem}
together with the formulas (\ref{lengthzx}) and (\ref{ardeg}). 
\vskip .5in

\centerline{\bf \large Part V. Examples}

\section{Level structures}\label{sectionlevel}

In this section we define variants of our moduli problem $\M(\kay; n-r, r)$, and discuss how to extend our main results 
to these cases. These variants 
actually show up in the examples discussed in the following sections. 

\subsection{Special parahoric level structures at ramified primes}   For every $p\mid\Delta$, we fix  an {\it even}
 integer $t(p)$ with $0\leq t(p)\leq n$ (the {\it type} of $p$). We denote by ${\bf t}$ the function $p\mapsto t(p)$ 
 on the set of divisors of $\Delta$.  We then introduce the following variant $\M(\kay, {\bf t}; n-r, r)^{*, \rm naive}$ 
 of the stack $\M(\kay; n-r, r)^{\rm naive}$ over $({\rm Sch}/\Spec \OK)$ of section 2. It parametrizes objects $(A, \iota, \l)$ 
 as in the definition of $\M(\kay; n-r, r)^{\rm naive}$, except that the condition that the polarization $\l$ be principal 
 is replaced by the following condition.
We require that  $ \ker\lambda\subset A[\Delta]$, so that $\OK/(\Delta)$ acts on $\ker\lambda$. In addition, we require that 
this action factors  through the factor ring $\prod_{p\mid \Delta}\F_p$ of $\OK/(\Delta)$, and that the height of $ \ker\lambda$ 
be equal to $\prod\nolimits_{p\mid\Delta}p^{t(p)}$. Note that if ${\bf t}=0$, then  $\M( \kay, {\bf t}; n-r, r)^{*, \rm naive}$ coincides 
with $\M(\kay; n-r, r)^{\rm naive}$. When $\kay$ and $\bf t$ are understood, we simply write $\M( n-r, r)^{*, \rm naive}$ for this stack.  

Then  $\M( n-r, r)^{*, \rm naive}$ is a Deligne-Mumford stack over $\Spec \OK$ which is smooth of relative dimension $(n-r) r$ over 
$\Spec \OK[\Delta^{-1}]$, provided it is not empty (this may happen, see below).   At the  primes $p\mid \Delta$, the stack is
 not flat in general. We define $\M( n-r, r)^{*}$ to be the 
flat closure of $\M( n-r, r)^{*, \rm naive}\times _{\Spec O_\smallkay}\Spec\OK[\Delta^{-1}]$ in $\M( n-r, r)^{*, \rm naive}$. 
\begin{prop}\label{structurestar}
Let $r=1$ and assume $2\nmid \Delta$. 

(i) If ${\bf t}=0$, then $\M( \kay, {\bf t}; n-1, 1)^{*}=\M(\kay; n-1, 1)$.

(ii)  At all primes $p\mid\Delta$ with $t(p)=0$, $\M( \kay, {\bf t}; n-1, 1)^{*}$ is Cohen-Macaulay, normal, and regular outside
 finitely many points.  If $n\geq 3$, the special fibers  at such primes  $p$  are irreducible, reduced, normal and with isolated rational singularities.
 If $n=2$, then $\M( \kay, {\bf t}; n-1, 1)^{*}$ has semi-stable, but non-smooth, reduction at such primes. 

(iii) If $n$ is even, then $\M( \kay, {\bf t}; n-1, 1)^{*}$ is smooth at all primes $p\mid\Delta$ with $t(p)=n$. 

(iv) If $n$ is odd, then $\M( \kay, {\bf t}; n-1, 1)^{*}$ is smooth at all primes $p\mid\Delta$ with $t(p)=n-1$. 

\end{prop}
\begin{proof} Statement (i) is Pappas' Proposition \ref{pappasflatness}, which is  \cite{pappas.JAG}, 
Theorem 4.5, and statement (ii) is also contained in loc.~cit.; statement (iii) is \cite{PR.III}, \S 5.c., 
and statement (iv) is \cite{Arz}, Prop.~4.16.
\end{proof}

 We extend  the definition of the special cycles in the obvious way:  for a hermitian matrix $T\in\Herm_m(\OK)$, we define the 
 Deligne-Mumford stack 
$\ZZ(T)^{*}$ equipped with a natural  morphism to $\M^{*}=\M(n-r, r)^{*}\times\M_0$ as the stack that parametrizes
 collections $(A, \iota, \l; E, \iota_0, \l_0; {\bf x})$, where ${\bf x}=[x_1, \ldots , x_m]\in \Hom_{O_\smallkay}(E, A)^m$
  is an $m$-tuple of homomorphisms such that $h'({\bf x}, {\bf x})=T$. 
If $S$ is not connected, we require these conditions on each connected component. 

We expect that with these definitions all results of the previous sections can be transposed to these cases. Let us
 illustrate this by discussing the complex uniformization of $\M(\kay, {\bf t}; n-r, r)^{*}$. 

\begin{defn}  An $\OK$-module $L$ equipped with a $\kay$-valued hermitian form of signature $(n-r, r)$ is called\footnote{
Hopefully, there will be no confusion between this notion of type and the type of a $G^V_1$-genus which occurs 
in earlier sections.} {\it of type}
 ${\bf t}$ if $L\subset L^\vee\subset \Delta^{-1}L$ and $L^\vee/L\simeq \prod\nolimits_{p\mid \Delta}\F_p^{t(p)}.$  
\end{defn} 
In particular, an $\OK$-module of type ${\bf t}=0$ is a self-dual hermitian $\OK$-module.  We note that for any $\OK$-module
 of type ${\bf t}$, the hermitian space $V=L\otimes\Q$ is relevant. Indeed, to check the existence of a self-dual lattice in $V$, 
 it suffices to check this locally at any inert prime $p$. But for such $p$ a self-dual lattice in   $V_p$ is given
  by  $L\otimes\Z_p=L^\vee\otimes\Z_p$. We also note that,
   if $2\nmid \Delta$,  lattices of type ${\bf t}$  always exist in a given relevant hermitian space $V$, except in the case when $n$
    is even and $t(p)=n$ for some $p\mid\Delta$ with ${\rm inv}_p(V)=-1$. Indeed, this is a local question at primes $p$
     dividing $\Delta$. Fix a self-dual lattice $\Lambda$ in $V_p$. Then,  up to  conjugacy   under the unitary group
      ${\rm U}(V_p)$, the  lattices $L$ in $V_p$ with $L\subset L^\vee\subset \pi^{-1}L$  such that $L^\vee/L\simeq \F_p^{t(p)}$
       correspond to the totally isotropic subspaces of $\Lambda/\pi\Lambda$ of dimension $t(p)/2$ with respect to the
        non-degenerate symmetric form induced by the hermitian form. Such subspaces always exist except in the case
         when $n$ is even and $t(p)=n$, in which case they exist if and only if ${\rm inv}_p(V)=1$. 

Let $\mathcal L_{(n-r, r)}(\kay, {\bf t})$ be the set of isomorphism classes of hermitian $\OK$-modules of signature $(n-r, r)$
 and   type ${\bf t}$. Now the complex uniformization of $\M( \kay, {\bf t}; n-r, r)^{*}$ is given by the following  analogue of Proposition \ref{C-points}.

\begin{prop}
There is an isomorphism of orbifolds
$$
\M( \kay, {\bf t}; n-r, r)^{*}(\C)\isoarrow \coprod_L\, \big[\Gamma_L\back D(L)\big] ,
$$
where $L$ runs over $\mathcal L_{(n-r, r)}(\kay, {\bf t})$. \qed
\end{prop}
\begin{remark}
By our remarks above, the set $\mathcal L_{(n-r, r)}(\kay, {\bf t})$ is either empty or  is in bijective correspondence with  $\mathcal L_{(n-r, r)}(\kay)$,
provided that $2\nmid\Delta$. 
\end{remark}

\subsection{Level structures at unramified primes } Now we discuss how to introduce level structures at unramified primes.
\begin{defn}
Let $N$ be an odd positive integer prime to $\Delta$. A {\it level $N$-structure} on an object $(A, \iota, \l)$ of $\M( \kay; n-r, r)(S)$
 is an $\OK$-linear isomorphism of  finite group schemes 
$$A[N]\isoarrow (\OK/N\OK)^n_S,$$
 compatible with the hermitian form associated to  the Riemann form corresponding to $\l$  on the LHS and the 
 standard hermitian form on the RHS, up to a scalar in 
$(\Z/N\Z)^\times$. 
\end{defn}
Here the standard hermitian form $h$ on $(\OK/N\OK)^n$ with values in $\OK/N\OK$ is given in terms of the
 canonical basis by $h(e_i, e_{n-j+1})=\delta_{i j},\forall i, j=1,\ldots,n$. It induces a similar form on the constant group scheme over $S$. 
 Note that a level $N$-structure can only exist if $N$ is invertible on $S$. The compatibility condition is independent of the choice
  of trivialization of the $N$-th roots of unity needed for the comparison with the
 standard form. We also note that if $L$ is a self-dual $\OK$-module, then $L/NL$ is isomorphic to $(\OK/N\OK)^n$ with the standard form.

 We obtain a Deligne-Mumford stack $\M(\kay; n-r, r)_{(N)}$ over $\M(\kay; n-r, r)[N^{-1}]$
 which parametrizes 
 collections $(A, \iota, \l, \eta)$ where $\eta$ is a level $N$-structure. It is smooth
  of relative dimension $(n-r)r$ over $\Spec\OK[(N\Delta)^{-1}]$. Note that if $N\geq 3$, then, by Serre's Lemma, $\M(\kay; n-r, r)_{(N)}$ is a scheme. 
  
  \begin{remark}
  A variant of the preceding definition arises from a subgroup $\bar K$ of ${\rm GU}\big((\OK/N\OK)^n\big)$, 
  defining a level $\bar K$-structure to be  an isomorphism 
  between $A[N]$ and $(\OK/N\OK)_S^n$,  given modulo $\bar K$.  In this way we obtain a 
  Deligne-Mumford stack $\M(\kay; n-r, r)_{(N),  \bar K}$ over $\Spec \OK[N^{-1}]$.  
  For instance, if $\bar K$ is the subgroup preserving the standard flag spanned by $e_1, e_2, \ldots,e_n$, then 
  a level $\bar K$-structure corresponds to  a complete  flag  of primitive $\OK$-submodules in $A[N]$, which is self-dual for the Riemann form. 
  
 There are   adelic variants of these definitions. Let $\Q_N=\prod\nolimits_{\ell\mid N}\Q_\ell$
  and $\Z_N=\prod\nolimits_{\ell\mid N}\Z_\ell$. Then a  {\it $\Gamma(N)$-level structure} on an object $(A, \iota, \l)$ is  
  a class of $\kay$-linear isomorphisms which respect the hermitian forms up to a scalar in $\Q_N^\times$,
 $$
 \eta:\, T_N(A)^o\isoarrow \kay^n\otimes\Q_N ,
 $$
 given modulo $\Gamma(N)$. Here $T_N(A)^o=\prod\nolimits_{\ell\mid N}T_\ell(A)^o$ is the product of the 
 rational Tate modules for $\ell\mid N$, and $\Gamma(N)$ denotes the 
 principal congruence subgroup of level $N$ of $\Gamma(1)={\rm GU}\big(\OK^n\otimes\Z_N)$. A level $N$-structure is 
 equivalent to a $\Gamma(N)$-level structure $\eta$  such that  $\eta\big( T_N(A)\big)= \OK^n\otimes\Z_N$.  More 
 generally, if $K$ is an open compact subgroup of 
 ${\rm GU}(k^n\otimes\Q_N)$, one defines the notion of a {\it $K$-level structure} as the datum of $\eta$ as above, given  
 modulo $K$. For a subgroup $\bar K\subset {\rm GU}\big((\OK/N\OK)^n\big)$,   a level $\bar K$-structure is equivalent to giving a $K$-level structure $\eta$ 
 with $\eta\big( T_N(A)\big)= \OK^n\otimes\Z_N$, where  $ K$ is the  inverse image  of $\bar K$ in $\Gamma(1)$.
  \end{remark}
  The  definition  of our special cycles can now be extended  as follows to the cases with level structures. 
  Let $K_0\subset {\rm GU}\big(\kay\otimes\Q_N\big)$ and 
  $K\subset {\rm GU}\big(\kay^n\otimes\Q_N\big)$ be open compact  subgroups, with associated moduli 
  stacks $\M(\kay; 1, 0)_{K_0}$ and $\M(\kay; n-r, r)_{K}$. Let $m$ be a positive integer, and   fix a compact open subset 
  $$\o\subset \Hom_\smallkay\big(\kay\otimes\Q_N, \kay^n\otimes\Q_N\big)^m$$
stable  under the action of $K_0\times K$. Then for $T\in\Herm_m(\OK[N^{-1}])$, we define the stack $\ZZ(T, \o)$ over 
   $\M(\kay; n-r, r)_{K}\times\M(\kay; 1, 0)_{K_0}$ as parametrizing  the collections 
   $(A, \iota, \l, \eta, E, \iota_0, \l_0, \eta_0; {\bf x})$, where ${\bf x}\in \big(\Hom_{\OK}(E, A)\otimes\OK[N^{-1}]\big)^m$ satisfies the following two conditions
   
   (i) $h'({\bf x},{\bf x})=T$
   
   (ii) $\eta\circ {\bf x}\circ\eta_0^{-1}\in \o .$
   
   This definition is analogous to the definition of special cycles adopted in our previous papers \cite{kudlaannals, KR1, KR2}. 
   Note that if $K_0$ and $K$ are trivial, i.e., equal to 
   $\Gamma(1)_0$ and $\Gamma(1)$ resp., so that  $\M(\kay; n-r, r)_{K}\times\M(\kay; 1, 0)_{K_0}=\M(\kay; n-r, r)\times\M(\kay; 1, 0)$, 
   and if $\o=\Hom_{\OK}\big(\OK\otimes\Z_N, \OK^n\otimes\Z_N\big)^m$ and $T\in\Herm_m(\OK)$, then the stack  $\ZZ(T, \o)$ is 
   equal to the previously defined stack $\ZZ(T)$.

  \begin{remark}
 Of course, one may also mix the two kinds of level structures and define in this way stacks $\M(\kay, {\bf t}; n-r, r)_{K}^*$ 
 over $\Spec\OK[N^{-1}]$, and corresponding special cycles. 
  \end{remark}
  
    We expect that the results of the previous sections can be extended to the moduli stacks $\M(\kay, {\bf t}; n-r, r)^*_{K}$. 
    For the analogues of Theorems \ref{length} and \ref{main}, one has to assume that $\Diff_0(T)=\{p\}$, where $p\nmid N$ is odd.

\section{The case $n=2$}\label{sectionGrossZ}

In this section, we illustrate our general results in the case $n=2$. In particular, 
we explain the relation of our special cycles to the Heegner points defined in the 
classic paper of Gross and Zagier, \cite{grosszagier}. 

Let $\Cal E$ be the moduli stack of elliptic curves over $\Spec \Z$.  For a scheme $S$, an object $E\in \Cal E(S)$, 
and an $\OK$-module $M$, we obtain an abelian scheme $M\tt_\Z E$ over $S$ by the Serre construction, \cite{serre.CM}. 
Then $M\tt_\Z E$ has a natural $\OK$-action, and the dual abelian scheme is given by 
$$(M\tt_\Z E)^\vee \simeq M^\vee\tt_\Z E^\vee,$$
where $M^\vee = \Hom_\Z(M,\Z)$ with its natural $\OK$-action.  If we take a fractional ideal $M=\frak a$, then the abelian scheme 
$\frak a\tt_\Z E$
has relative dimension $2$ over $S$ and the $\OK$-action satisfies the $(1,1)$-signature 
condition
$$\cha(T, \iota(a)\mid \Lie A) = T^2 -\tr(a)\,T+N(a)\ \in \Z[T],\qquad a\in \OK.$$
Note that there is an $\OK$-antilinear isomorphism
$$\d^{-1}\frak a \isoarrow \frak a^\vee, \qquad a\mapsto N(\frak a)^{-1}\,(\cdot ,a)_{\smallkay}, $$
where $(x,y)_{\smallkay} = \tr(x y^\s)$ is the trace form of $\kay/\Q$, and $\d^{-1}$ is the inverse different. 
The canonical principal polarization $\l_E: E\isoarrow E^\vee$ gives a polarization
$$\l: \frak a\tt E \isoarrow \frak a\tt E^\vee \lra \d^{-1}\frak a \tt E^\vee \simeq \frak a^\vee\tt E^\vee,$$
where middle arrow is induced by the inclusion $\frak a \subset \d^{-1}\frak a$, and hence has degree $|\Delta|$. 
Thus, we obtain a functor
\begin{equation}\label{defja}
j_{\frak a}:\Cal E \lra \M^{\text{\rm spl}}(1,1)^*,\qquad E\mapsto (\frak a\tt_\Z E, \iota, \l).
\end{equation}
where $\M^{\text{\rm spl}}(1,1)^*$ is a moduli stack that we now explain. 

Let $\M(1,1)^*$ be the moduli stack over $\Spec \Z$ for triples $(A,\iota,\l)/S$, where $A$ is an abelian 
scheme over $S$, $\iota$ is an action of $\OK$ on $A$ satisfying the $(1,1)$-signature condition, 
and $\l: A\lra A^\vee$ is a polarization with corresponding Rosati involution satisfying $\iota(a)^*= \iota(a^\s)$. 
Finally, we require that
$\ker(\l) = A[\d]$. If $\Delta = N(\d)$ is odd, this moduli problem was introduced in the previous section, and corresponds
to the function ${\bf t}(p)=2$, for all $p\mid \Delta$. 
Since the polarization $\l$ is principal away from $\Delta$, there is a decomposition
$$\M(1,1)^* = \coprod_{V} \M^V(1,1)^*,$$
where $V$ runs over isomorphism classes $\Cal R_{(1,1)}(\kay)$ of relevant hermitian spaces.
These can be described as follows.  For a quaternion algebra $B/\Q$ with an embedding $\kay\rightarrow B$, 
write $B = \kay\oplus  B_-$,  where $B_-$ is the set of $b\in B$ such that $b a = a^\s b$ for all $a\in \kay$. 
The space $V = V^B = B$, viewed as a left vector space over $\kay$, has a hermitian form
$(x,y) = (x y^\iota)_+,$
where $b\mapsto b^\iota$ is the main involution of $B$, and where for $x\in B$, we denote by $x_+\in \kay$ 
its component in the above direct sum decomposition.  Let $D(B)$ be the product of the primes $p$ for which $B\tt_\Q\Q_p$ is a division 
algebra. 
\begin{lem} 
$$\Cal R_{(1,1)}(\kay) = \{\ [V^B]\mid B\  \text{\rm indefinite}, \  
D(B)\mid \Delta\ \}.
$$\qed 
\end{lem}
Note that the quaternion algebra $M_2(\Q)$ always occurs here and that $V^{\text{\rm spl}}:=V^{M_2(\Q)}$ is the 
split hermitian space.  We write
$$\M^{\text{\rm spl}}(1,1)^* = \M^{V^{\text{\rm spl}}}(1,1)^*.$$

When $2\mid \Delta$ and we work over $\Spec\OK[\frac12]$, the `type' of the hermitian form on the $2$-adic Tate module $T_2(A)$ gives 
a finer decomposition.  Let $\pi$ be a uniformizer of $O_{\smallkay,2}$, and note that $T_2(A)$ is then 
a $\pi^i$-modular lattice for $i = \ord_2(\Delta)$, in the sense of the lemma below.  When $2$ is ramified, the isometry types of such lattices are given as follows. 
\begin{lem} Let $L$ be an $O_{\smallkay,2}$-lattice of rank $2$ that is $\pi^i$-modular, i.e., 
$L^\vee = \pi^{-i}L$ as lattices in $V_2=L\tt_{\Z_2}\Q_2$, where $i=\ord_2(\Delta)$. \hfb
(a) Suppose that $V_2$ is isotropic. Then representatives for the isometry types of $\pi^i$-modular lattices are given by 
$$2 \,\begin{pmatrix}{}&1\\1&{}\end{pmatrix}, \qquad  2 \,\begin{pmatrix}1&{}\\{}&-1\end{pmatrix} $$
for $i=2$, and 
$$2\,\begin{pmatrix}0&\pi\\ \pi^\s&0\end{pmatrix}, \qquad 2\,\begin{pmatrix}2&\pi \\ \pi^\s& 0\end{pmatrix},$$
for $i=3$. \hfb
(b) Suppose that $V_2$ is anisotropic. Then representatives for the isometry types of $\pi^i$-modular lattices are given by
$$2 \,\begin{pmatrix}1&{}\\{}&-1\end{pmatrix},$$
for $i=2$, and 
$$2\,\begin{pmatrix}2&\pi \\ \pi^\s& 4\end{pmatrix},$$
for $i=3$. 
\end{lem} 

\begin{remark}  We will refer to the first of the two lattices in each line of the lists in case (a) as  `type II' lattices. 
\end{remark}

It is easily checked that, for $p\mid \Delta$ with $p\ne 2$,  and $\pi_p$ a uniformizer of $\OKp$, there is a unique isometry class of 
$\pi_p$-modular hermitian lattices of rank $2$. 

We denote by $\Cal R_{(1,1)}(\kay)^*$ the isomorphism classes of pairs $V^* = (V,[[L]])$, where $V$ is a relevant hermitian space and 
$[[L]]$ is a $G^V_1(\A_f)$-genus of $\d$-modular lattices in $V$.  By the previous lemma, the map from $\Cal R_{(1,1)}(\kay)^*$ to $\Cal R_{(1,1)}(\kay)$
has fibers of cardinality $1$ or $2$. The latter occurs precisely when $2\mid \Delta$ and for those $V = V^B$ where $B$ is split at $2$. 

Then there is a decomposition
$$\M(1,1)^*[\frac12] = \coprod_{V^*} \M^{V^*}(1,1)^*[\frac12],$$
where $V^*$ runs over $\Cal R_{(1,1)}(\kay)^*$. Here $\M^{V^*}(1,1)^*[\frac12]$ is the open and closed substack where 
the $2$-adic Tate module is of the type determined by $V^*$. 

\begin{prop}\label{Edecompo} (i) The morphism $j_{\frak a}$ of (\ref{defja}) has image contained in $\M^{\text{\rm spl}}(1,1)^*$. \hfb
(ii) If $2 \nmid \Delta$, the morphisms $j_{\frak a}$ induce an isomorphism 
$$\coprod_{[\frak a]\in C(\smallkay)}j_{\frak a} : \, 
\coprod_{[\frak a]\in C(\smallkay)} \mathcal E \isoarrow \M^{\text{\rm spl}}(1, 1)^*.$$
(iii) If $2\mid \Delta$, then, over $\Spec \OK[\frac12]$, the morphisms $j_{\frak a}$ induce an isomorphism 
$$\coprod_{[\frak a]\in C(\smallkay)}j_{\frak a} : \, 
\coprod_{[\frak a]\in C(\smallkay)} \mathcal E[\frac12]\isoarrow \M^{\text{\rm spl, II}}(1, 1)^*[\frac12].$$
where  $\M^{\text{\rm spl, II}}(1, 1)^*[\frac12]$ denotes the locus in $\M^{\text{\rm spl}}(1, 1)^*[\frac12]$ 
where the $2$-adic Tate module has type II.  
\end{prop}
\begin{proof}
Part (i) follows from the fact that for the rational Tate module we have $T^p(\frak a\otimes E)^0\simeq \kay\otimes_\Q T^p(E)^o\simeq {\rm M}_2(\A_f^p)$. 

(ii) Each morphism $j_{\frak a}$ is proper, as one checks easily using the valuative criterion for properness. Also, $j_{\frak a}$ is a bijective map 
onto its image, in the stack sense. Indeed,  for a pair $E, E'$ of elliptic curves over a base scheme $S$, any 
isomorphism $\frak a\otimes E\isoarrow \frak a\otimes E'$ 
induces an isomorphism $E\isoarrow E'$,
after choosing a $\Z$-basis of $\frak a$, which allows us to identify $\frak a\otimes E$ with $E\times E$, and $\frak a\otimes E'$ with $E'\times E'$. 
By Proposition \ref{structurestar}, (iii), the stack $\M^{\text{\rm spl}}(1, 1)^*$ is regular; hence we may apply  Zariski's Main theorem, and 
the morphism $j_{\frak a}$ 
induces an isomorphism of  $\mathcal E$ with a union of connected components of  $\M^{\text{\rm spl}}(1, 1)^*$. 

When $2\mid \Delta$, over $\Spec \OK[\frac12]$, a simple calculation of the hermitian form on 
$T_2(j_{\frak a}(E)) \simeq \frak a\tt_\Z T_2(E)$ 
shows that   this hermitian lattice has type II and hence the image of $j_{\frak a}(\Cal E[\frac12])$ lies in $\M^{\text{\rm spl,II}}(1, 1)^*[\frac12]$.

Since $\mathcal E$ is 
connected, to prove the isomorphism in (ii), it suffices to prove that 
$\pi_0(\M^{\text{\rm spl}}(1, 1)^*)\simeq C(\kay)$, or also, since $\M^{\text{\rm spl}}(1, 1)^*$ is regular, that 
$\pi_0(\M^{\text{\rm spl}}(1, 1)^*_\C)\simeq C(\kay)$. This follows from
Proposition~\ref{identshim} and (\ref{pizero}), where we note that $\nu(K) = \widehat{\Z}$ and that the image of $K$ 
in first factor of $T(\A_f) \simeq \kay^1_{\A_f}\times \Q^\times_{\A_f}$ is the image $U$ of $\widehat{O}_{\smallkay}^\times$ under the map
$u\mapsto u \bar u^{-1}$.  Also note that this last map induces an isomorphism 
$C(\kay)= \kay^\times\back \kay^\times_{\A_f}/\widehat{O}_{\smallkay}^\times \simeq \kay^1\back \kay^1_{\A_f}/U$.

The proof of (iii) is similar and omitted. 
\end{proof}

Now we turn to the special cycles.  To lighten notation, we now denote by  $\mathcal E$ (resp. $\M^{\text{\rm spl}}(1, 1)^*$) 
the base change from 
$\Spec \Z$ to $\Spec \OK$, and write 
$$\M^* = \M^{\text{\rm spl}}(1,1)^*\times \M_0,$$
where the fiber product is taken over $\Spec \OK$. 
For a positive integer $m$, we define the special cycle
$$\ZZ^*(m) \lra \M^*$$
as the stack of collections $(A,\iota,\l, E_0,\iota_0,\l_0;x)$ where
$(A,\iota,\l)$ is a object of $\M^{\text{\rm spl}}(1,1)^*(S)$, $(E_0,\iota_0,\l_0)$ is an object of $\M_0(S)$, and 
$$x\in \Hom_S((E_0,\iota_0),(A,\iota))\tt\Q$$
is an $\OK$-linear quasi-homomorphism such that 
$$x^* = \l_0^{-1}\circ x^\vee \circ \l\in \Hom_S((A,\iota),(E_0,\iota_0))$$
is a homomorphism\footnote{In the case of a principal polarization, 
the integrality of $x^*$ is equivalent to the integrality of $x$.  In general, $x$ need not be integral, so that we are 
slightly extending the earlier definition of section~\ref{sectionlevel}.} with 
$h(x,x) = x^*\circ x= \l_0^{-1}\circ x^\vee\circ\l\circ x = m$.  

We want to determine the pullback of $\ZZ^*(m)$ under $j_{\frak a}\times 1$, i.e., the fiber product
$$\begin{matrix} (j_{\frak a}\times 1)^* \ZZ^*(m) & \lra & \ZZ^*(m)\\
\nass
\downarrow &{}&\downarrow \\
\nass
\Cal E\times \M_0&\lra &\M^*.
\end{matrix}
$$
Let $\Cal T(m)$ be the stack for which the objects of $\Cal T(m)(S)$ are triples $(E,E',\psi)$, where $E$ and $E'$ are elliptic schemes over $S$ 
and $\psi: E\rightarrow E'$ is an $m$-isogeny.  Let $s: \Cal T(m) \rightarrow \Cal E$ (resp. $t:\Cal T(m)\rightarrow \Cal E$) 
be the morphism defined by sending $(E,E',\psi)$ to $E$ (resp. $E'$). 
Consider the fiber product
$$\begin{matrix}
\Cal T(m)_{\Delta,\frak a}&\lra&\Cal T(m)\\
\nass
\downarrow&{}&\phantom{\scr(s,t)}\downarrow {\scr(s,t)}\\
\nass
\Cal E\times \M_0&\overset{1\times i_{\frak a}}{\lra} &\Cal E\times \Cal E.
\end{matrix}
$$
where $i_{0}:\Cal M_0\rightarrow \Cal E$ is the morphism that sends $(E_0,\iota_0)$ to $E_0$ 
and $i_{\frak a} = i_0\circ t_{\frak a}$ where $t_\frak a: \M_0\rightarrow \M_0$ sends $E_0$ to $\frak a^{-1}\tt_{\OK} E_0$. 
\begin{prop}\label{comphecke}
There is a natural isomorphism 
$$(j_{\frak a}\times 1)^* \ZZ^*(m) \isoarrow \Cal T(m)_{\Delta,\frak a}$$
over $\Cal E\times \M_0$. 
\end{prop}
\begin{proof}  An object of $(j_{\frak a}\times 1)^* \ZZ^*(m)(S)$ is a collection $(E,E_0,\iota_0,\l_0; x)$ 
where $x: E_0 \lra \frak a\tt_\Z E$ is an $\OK$-linear quasi-homomorphism with $x^*$ integral and with $h(x,x) = x^* \circ x = m$. Here 
$x^* = \l_0^{-1}\circ x^\vee\circ \l: \frak a\tt_\Z E\rightarrow E_0$. 
Let $x^*_0: E\rightarrow \frak a^{-1}\tt_{\OK}E_0$ be the image of $x^*$ 
under the natural isomorphism
\begin{equation}\label{serre.adj}
\Hom_{\OK}(\frak a\tt_\Z E, E_0) \isoarrow \Hom_{\Z}(E, \frak a^{-1}\tt_{\OK} E_0)
\end{equation}
arising from the Serre construction. On the other hand, 
an object of $\Cal T(m)_{\Delta,\frak a}(S)$ is a collection $(E,E_0,\iota_0,\l_0;y_0)$ where $y_0: E\rightarrow \frak a^{-1}\tt_{\OK}E_0$ 
is an isogeny of degree $m$.  The proposition is then an immediate consequence of the following result. 
\begin{lem}  If $y\in \Hom_{\OK}(\frak a\tt_\Z E, E_0)$ and $y_0\in \Hom_{\Z}(E, \frak a^{-1}\tt_{\OK} E_0)$
correspond under (\ref{serre.adj}), then 
$h(y,y) =  y\circ y^*=\deg(y_0)$.
\end{lem} 
\begin{proof}  For an isogeny $y_0\in \Hom_{\Z}(E, \frak a^{-1}\tt_{\OK} E_0)$, the corresponding $y\in \Hom_{\OK}(\frak a\tt_\Z E, E_0)$
is given by  
\begin{equation}\label{yy0}
y:  \frak a\tt_\Z E \overset{1\tt y_0}{\lra} \frak a\tt_\Z(\frak a^{-1}\tt_{\OK}E_0)\lra E_0,
\end{equation}
where the second map arises by multiplication $a\tt (b\tt x) \mapsto ab \cdot x$.  
To compute the desired relation between $\deg y_0$ and $h(y,y) = y\circ y^*$, we can pass to the rational Tate modules at a prime 
$\ell$ different from the characteristic.   On the rational Tate modules $V_\ell(E)=T_\ell(E)^0$, $V_\ell(\frak a^{-1}\tt_{\OK}E_0)$ and $V_\ell(E_0)$, there are 
nondegenerate $\Q_\ell$-valued alternating forms 
$\form_E$, $\form_{\frak a^{-1}\tt_{\OK}E_0}$,  and $\form_{E_0}$ induced by the canonical principal polarizations 
and a fixed trivialization $\Q_\ell(1) \simeq \Q_\ell$. For example, $\rho_E(z_1,z_2) = e_E(z_1,\l_E(z_2))$, where $e_E$ is the Weil pairing 
on $V_\ell(E)\times V_\ell(E^\vee)$ and $\l_E$ is the canonical polarization. 
The pullback of $\form_{\frak a^{-1}\tt_{\OK}E_0}$ under the isomorphism 
$$y_{0,\ell}: V_\ell(E) \isoarrow V_\ell(\frak a^{-1}\tt_{\OK}E_0)$$
is $\deg(y_0)\cdot \form_E$, while the pullback of $\form_{E_0}$ under the isomorphism
$$m_{\frak a, \ell}: V_\ell(\frak a^{-1}\tt_{\OK} E_0)  \isoarrow V_\ell(E_0)$$
associated to the quasi-isogeny $\frak a^{-1}\tt E_0 \lra E_0$, given by multiplication, 
is $N(\frak a)^{-1}\cdot \form_{\frak a^{-1}\tt_{\OK}E_0}$. 
Similarly, setting $A= \frak a\tt_\Z E$, the non-degenerate $\Q_\ell$-valued alternating pairing 
on $V_\ell(A) = V_\ell(\frak a\tt_\Z E) \simeq \kay\tt_\Q V_\ell(E)$ determined by the polarization $\l=\l_A$ 
is given by 
$$\form_A(a \tt z, a'\tt z') = N(\frak a)^{-1}\,(a,a')_{\smallkay} \,\form_E(z,z').$$
From (\ref{yy0}), we have
\begin{equation}\label{yy0ell2}
y_\ell:\kay\tt_\Q V_\ell(E)\overset{\xi}{ \isoarrow} V_\ell(E_0)\oplus \overline{V_\ell(E_0)} \lra V_\ell(E_0),
\end{equation}
$$ a \tt z \mapsto \big(a \cdot m_{\frak a,\ell}\circ y_{0,\ell}(z), \bar a \cdot m_{\frak a,\ell}\circ y_{0,\ell}(z)\big) 
\mapsto a \cdot m_{\frak a,\ell}\circ y_{0,\ell}(z).$$
Here, to arrive at the middle entry, we have used the isomorphism
$$V_\ell(\frak a\tt_\Z(\frak a^{-1}\tt_{\OK} E_0)) \simeq \kay\tt_\Q (\kay\tt_\smallkay V_\ell(E_0)) \simeq V_\ell(E_0)\oplus (\kay\tt_{\smallkay,\s}V_\ell(E_0)).$$
A short calculation shows that the pullback of the diagonal form $\form_{E_0} \oplus \form_{E_0}$ on the middle term under the isomorphism 
$\xi$ is $\deg y_0\cdot \form_A$.  
If we identify $V_\ell(A)$ and $V_\ell(E_0)\oplus \overline{V_\ell(E_0)}$ via $\xi$, 
then $y^*_\ell = \deg y_0\cdot \text{\rm inc}_1$, where $\text{\rm inc}_1$ is the inclusion of $V_\ell(E_0)$ into the first factor and the adjoint $y^*_\ell$ is 
defined with respect to $\form_A$. 
Hence $y\circ y^*$ is multiplication by $\deg y_0$ and the lemma is proved. 
\end{proof}
This completes the proof of  Proposition \ref{comphecke}. 
\end{proof}

For comparison with the classical theory, we add a little more level.  We fix  a positive integer $N$ and work over $\Spec \Z[N^{-1}]$. 
Let $\Cal E_0(N)$
be the moduli stack over $\Spec \Z[N^{-1}]$ for pairs of elliptic curves $(E,E',\phi)$ with a cyclic $N$-isogeny $\phi$.  
Again applying the Serre construction, we obtain a morphism
$$j_{\frak a, N}: \Cal E_0(N) \lra \M^{\text{\rm spl}}(1,1)^*_0(N),$$
where  $\M^{\text{\rm spl}}(1,1)^*_0(N)$ is the following moduli stack. 
For a locally noetherian base $S$, an object in $\M^{\text{\rm spl}}(1,1)^*_0(N)(S)$  is a collection $(\xi, \xi'; \phi)$
where $\xi=(A,\iota_A,\l_A)$ and $\xi'=(A', \iota_{A'},\l_{A'})$ are objects in $\M^{\text{\rm spl}}(1,1)^*(S)$, and
$\phi: A\rightarrow A'$ is an $\OK$-linear isogeny with $\phi^*(\l_{A'})= N \l_A$ and such that locally in the fppf topology, $\ker(\phi) \simeq (\OK/N\OK)_S$. 

As before we tacitly effect a base change from $\Spec \Z[N^{-1}]$ to $\Spec \OK[N^{-1}]$, and  write 
$$ \M^{**} = \M^{\text{\rm spl}}(1,1)^*_0(N) \times \M_0.$$
We define the special cycle 
$\ZZ^{**}(m)\rightarrow \M^{**}$ as the stack of collections $(\xi, \xi',\phi; E_0,\iota_0,\l_0;x)$, where
$x:E_0\rightarrow A$ is a quasi-homomorphism with $x^*$ integral, etc. just as before, but with the 
additional requirement that, fppf locally,  
$x^*( \ker(\phi))$  is isomorphic to $(\Z/N\Z)_S$. 

To describe the pullback $(j_{\frak a,N}\times 1)^*(\ZZ^{**}(m))$ over $\Cal E_0(N)\times \M_0$, let
$\Cal T(m)_0(N)$ be the stack whose objects over $S$ are collections $(E,E',\phi,E'', \psi)$, where
$(E,E',\phi)$ is an object of $\Cal E_0(N)$ and $\psi:E\rightarrow E''$ is an isogeny of degree $m$
such that the intersection $\ker(\psi)\cap \ker(\phi)$ is trivial. 
There are again source and target morphisms to $\Cal E_0(N)$, where the target 
is the object $(E'',E''/\psi(\ker \phi),\phi')$ with $\phi'$  the quotient map. 

Finally, as in \cite{grosszagier}, we make the assumption that all primes dividing $N$ split in $\kay$ and 
choose an integral ideal $\frak n$ with $N(\frak n) = N$.  There is a resulting morphism
$$i_{0,\frak n}: \M_0 \lra \Cal E_0(N),\qquad E_0 \mapsto (E_0 \overset{\phi}{\lra} E_0/E_0[\frak n]),$$
and its twist $i_{\frak a,\frak n} = i_{0,\frak n}\circ t_{\frak a}$. 
\begin{prop} There is a natural isomorphism 
$$(j_{\frak a,N}\times 1)^*(\ZZ^{**}(m)) \isoarrow \coprod_{\frak n} \Cal T(m)_0(N)_{\Delta,\frak a,\frak n},$$
over $\Cal E_0(N)\times \M_0$, where 
$$\begin{matrix}
\Cal T(m)_0(N)_{\Delta,\frak a,\frak n}&\lra&\Cal T(m)_0(N)\\
\nass
\downarrow&{}&\phantom{\scr(s,t)}\downarrow {\scr(s,t)}\\
\nass
\Cal E_0(N)\times \M_0&\overset{1\times i_{\frak a,\frak n}}{\lra} &\Cal E_0(N)\times \Cal E_0(N),
\end{matrix}
$$
is the fiber product. \qed
\end{prop}

\begin{remark}
This proposition can be interpreted as follows.  First suppose that $m=1$. Then, over a base $S$ on which 
$N$ is invertible, an object of $\Cal T(1)_0(N)_{\Delta,\frak a,\frak n}(S)$ is a collection
$(E\overset{\phi}{\lra} E', E_0, \psi)$, where $(E\overset{\phi}{\lra} E')$ is an object of $\Cal E_0(N)(S)$ and 
$\psi$ is an isomorphism
$$(E\overset{\phi}{\lra} E') \isoarrow (\frak a^{-1}\tt_{\OK}E_0 \lra \frak a^{-1}\tt_{\OK}E_0/(\frak a^{-1}\tt_{\OK}E_0)[\frak n]).$$
Hence, $\Cal T(1)_0(N)_{\Delta,\frak a,\frak n}$ can be viewed as the graph of the morphism of stacks
\begin{align*}
i_{\frak a,\frak n}: \M_0 &\lra \Cal E_0(N),\\
\nass
E_0 &\mapsto (\frak a^{-1}\tt_{\OK}E_0 \lra \frak a^{-1}\tt_{\OK}E_0/(\frak a^{-1}\tt_{\OK}E_0)[\frak n])
\end{align*}
Passing to the coarse moduli schemes, for each $\frak a$ and $\frak n$, we have a morphism
$$i_{\frak a, \frak n}:\Spec(O_H[N^{-1}]) \lra \und{\Cal E_0(N)}$$
where $O_H$ is the ring of integers in $H$, the Hilbert class field of $\kay$. 
Finally, passing to the generic fiber, we obtain morphisms
\begin{equation}\label{Heegner.point}
i_{\frak a, \frak n,\Q}:\Spec(H) \lra \und{\Cal E_0(N)}_\Q = X_0(N).
\end{equation}
\begin{cor} (i)
The point (\ref{Heegner.point}) in $X_0(N)(H)$ determined by $\Cal T(1)_0(N)_{\Delta,\frak a,\frak n}$
is the Heegner point  associated to the ideal class $\frak a$ and 
the primitive ideal $\frak n$ of norm $N$ as in \cite{grosszagier}, p.227.\hfb
(ii) For $m\ge 1$, the set of points on $X_0(N)$ similarly determined by $\Cal T(m)_0(N)_{\Delta,\frak a,\frak n}$
is the image of this Heegner point under the $m$-th Hecke operator. 
\end{cor}
Thus, in this special case, our cycles $\ZZ^{**}(m)$ provide an integral version of  the images under the Hecke operators of the Heegner points
considered in \cite{grosszagier}. 
\end{remark}

So far we have discussed the case $B=M_2(\Q)$. This case corresponds to the condition that all prime divisors of 
$N$ are split in $\kay$, imposed by Gross and Zagier to ensure that their Heegner points lie on the modular curve. 
As observed in \cite{grosszagier} p.313, if this condition is relaxed,  the Heegner points lie on Shimura curves, as we now explain from our point of view. 
Suppose that $B$ is an indefinite division quaternion algebra 
with a maximal order $O_B$ and 
with $D(B)\mid |\Delta|$. Fix an embedding $\OK\rightarrow O_B$.  
Let $\M^B$ be the Drinfeld stack over $\Spec \Z$ parametrizing  two-dimensional abelian schemes 
with a special $O_B$-action, cf.~\cite{boutotcarayol}. Choose an element $\xi\in O_B$ with $\xi^2 = - D(B)$, \cite{boutotcarayol}, and 
define the involution $b\mapsto b': = \xi b^\iota \xi^{-1}$ where $b\mapsto b^\iota$ is the main involution of $B$. 
Given an object $(A,\iota_B)$ of $\M^B(S)$, there is a unique principal polarization $\l_A$ of $A$ for which the Rosati involution 
induces the involution $'$ on $O_B$, \cite{boutotcarayol}, Prop.~(3.3), p.134. Since the element $\sqrt{\Delta}\,\xi^{-1}\in O_B$
is invariant under $'$, the composition $\l:= \l_A\circ \iota_B(\sqrt{\Delta}\,\xi^{-1})$ is a polarization 
of $A$ whose Rosati involution $*$ satisfies $\iota(a)^* = \iota(a^\s)$ for all $a\in \OK$.  Here we have written $\iota$ for the restriction of $\iota_B$ to
$\OK$. 
Then the collection $(A,\iota, \l)$ 
is an object of the moduli space $\M(1,1)_{D(B)}^*$ which parametrizes two-dimensional abelian schemes with 
$\OK$-action of signature $(1,1)$ 
and polarization $\l$ of degree $|\Delta|/D(B)$.
 If $\Delta$ is odd, $\M(1,1)_{D(B)}^*$ coincides with the moduli problem 
$\M({\bf t};1,1)^*$ defined in the previous section where
$$t(p) = \begin{cases} 2 &\text{if $p\mid \Delta$, $p\nmid D(B)$,}\\
\nass
0&\text{if $p\mid \Delta$, $p\mid D(B)$.}
\end{cases}
$$
Thus we obtain a morphism 
$j^B:\M^B \rightarrow \M(1,1)_{D(B)}^*$
and twists 
$$j^B_{\frak a}:\M^B \rightarrow \M(1,1)_{D(B)}^*, \qquad j^B_{\frak a} = \frak a\tt_{\OK} j^B,$$
for any fractional ideal $\frak a$.  When $\Delta$ is odd, these give an isomorphism
\begin{equation}\label{quatemb}
\coprod_{[\frak a]\in C(\smallkay)} \M^B \isoarrow \M^{V}(1,1)^*_{D(B)},
\end{equation}
where $V= V^B$, generalizing that in (ii) of Proposition~\ref{Edecompo} where $B=M_2(\Q)$. 
As in the case of a split quaternion algebra, we can now pull back $\ZZ^*(m)$ to $\M^B\times\M_0$, via $j^B_{\frak a}\times 1.$ 
This parametrizes 
tuples  $(A, \iota_B, E_0, \iota_0, x)$ where $x\in \Hom_{O_{\smallkay}} (E_0, \frak a\otimes_{O_{\smallkay}} A)\otimes \Q$  is 
such that $x^*$ is integral with $x^*\circ x=m$. Of course, here the adjoint $x^*$ is formed using the canonical polarization $\lambda$ of $\frak a\otimes_{O_{\smallkay}} A$ constructed above. And we may add level structures as in the case of the split quaternion algebra.

We end this section by explaining the relations among the  $\M({\bf t}; 1, 1)^*$ for varying ${\bf t}$. 
\begin{prop}
Suppose that ${\bf t}$ is given and that $t(p)=0$ for some prime $p\ne2$ with 
$p\mid \Delta$. Define ${\bf t}'$ by $t'(p')=t(p')$ for $p\neq p'$ and $t'(p)=2$. Then  there exists an \'etale   Galois covering  of degree $2$
\begin{equation*}
\M({\bf t}; 1, 1)^{*, p}\lra \M({\bf t}; 1, 1)^* ,
\end{equation*}
equipped with a  proper  morphism
\begin{equation*}
\varphi:\, \M({\bf t}; 1, 1)^{*, p}\lra \M({\bf t}'; 1, 1)^* ,
\end{equation*}
which is finite of degree $p+1$ outside the fibers over $p$ and which contracts some lines in the special fiber of 
$\M({\bf t}; 1, 1)^{*, p}$ at $p$, {\rm (cf. Remark~\ref{contracts} below)}. 
\end{prop}
\begin{proof}
Let $\M({\bf t}; 1, 1)^{*, p}$ be the stack of objects $\xi=(A, \iota_A, \l_A)$ of $\M({\bf t}; 1, 1)^{*}$, together with an abelian
scheme ${A'}$ with $\OK$-action of signature $(1, 1)$ and a $\OK$-linear isogeny $\mu: {A'}\lra A$ of degree $p$ such that the pullback polarization
$\l_{A'}=\mu^*(\l_A)$ has kernel contained in ${A'}[\sqrt\Delta]$. The morphism $\varphi$ maps an object $(A, \iota_A, \l_A, \mu)$ to 
$({A'}, \iota_{A'}, \l_{A'})$ and is obviously proper. Let us determine the fiber of $\M({\bf t}; 1, 1)^{*, p}$
 over a geometric point $\xi=(A, \iota_A, \l_A)\in \M({\bf t}; 1, 1)^*(k)$. 
If ${\rm char}\, k\neq p$, then $\mu: {A'}\lra A$ is given by the $p$-adic Tate module of ${A'}$ which is an $O_{\smallkay_{ p}}$-stable 
submodule $\Lambda$ of $T_p(A)$ such that $\Lambda \subset T_p(A)\subset \Lambda^\vee=\pi^{-1}\Lambda$, where $\pi$ is a uniformizer in
$\kay_p$. 
However, the hermitian form $h_A$
on $T_p(A)$ arising from the polarization $\l_A$ induces a 
non-degenerate $\F_p$-valued {\it symmetric} form  on $\pi^{-1}T_p(A)/T_p(A)$. The set of $\Lambda$ as above 
corresponds in a one-to-one way to the set of isotropic lines in this two-dimensional 
$\F_p$-vector space, via $\Lambda\mapsto \Lambda^\vee/T_p(A)$. Since $p\ne 2$, there are precisely
two such lines. If ${\rm char}\,  k=p$, we use 
Dieudonn\'e modules instead of Tate modules. The Dieudonn\'e
module of $A$ is a module $M$ over $O_{\smallkay_{ p}}\tt W(k)$, and $\pi^{-1}M/M$ is equipped 
with a non-degenerate $k$-valued symmetric form. Then
$\mu: {A'}\lra A$ is given by the Dieudonn\'e module $\Lambda$ of ${A'}$ which is a $O_{\smallkay_{ p}}\tt W(k)$-submodule 
with $\Lambda \subset M\subset \Lambda^\vee=\pi^{-1}\Lambda$. Just as before $\Lambda$ corresponds to 
one of the two isotropic lines in $\pi^{-1}M/M$. 

Now let us determine the fiber of $\varphi$ over a geometric point $({A'}, \iota_{A'}, \l_{A'})\in \M({\bf t}'; 1, 1)^* (k)$. If ${\rm char}\,  k\neq p$, 
then the points in the fiber correspond to 
the $O_{\smallkay_{ p}}$-lattices $L$ in the rational Tate module $V_p({A'})$ with $T_p({A'})\subsetneq L\subsetneq T_p({A'})^\vee=\pi^{-1}T_p({A'})$ 
(these are automatically self-dual). 
Hence there are $p+1$ of them. If ${\rm char}\,  k=p$, then the points in the fiber correspond to $O_{\smallkay_{ p}}$-stable 
Dieudonn\'e lattices $M$ in the
rational Dieudonn\'e module of ${A'}$, containing the Dieudonn\'e module $M({A'})$ of ${A'}$, 
with  $M({A'})\subsetneq M\subsetneq M({A'})^\vee=\pi^{-1}M({A'})$. Now there are two cases. First suppose 
that $F\big(\pi^{-1}M({A'})\big)\neq M({A'})$, or equivalently, $V\big(\pi^{-1}M({A'})\big)\neq M({A'})$. 
Then either $F\big(\pi^{-1}M({A'})\big)\subset M$, or $V\big(\pi^{-1}M({A'})\big)\subset M$,
and 
$M$ is uniquely determined as $M=M({A'})+F\big(\pi^{-1}M({A'})\big)$ or $M=M({A'})+V\big(\pi^{-1}M({A'})\big)$, respectively. 
Hence in this case  there is a unique point in the fiber. 
Next suppose that  $F\big(\pi^{-1}M({A'})\big)= M({A'})=V\big(\pi^{-1}M({A'})\big)$. 
Then there are  no constraints on  $M$  and $M$ corresponds to an arbitrary  point in 
the projective line $\mathbb P\big(\pi^{-1}M({A'})/M({A'})\big)$.  

\begin{remark}\label{contracts}
Note that the points  $\xi\in\M({\bf t}'; 1, 1)^* (k)$ with a  fiber of positive dimension have a Dieudonn\'e module $M({A'})$ 
satisfying  $F\big(\pi^{-1}M({A'})\big)= M({A'})=V\big(\pi^{-1}M({A'})\big)$. Let $\xi$ lie in the component $\M^V({\bf t}'; 1, 1)^* (k)$. 
Then this  condition signifies that $\xi$ is supersingular if ${\rm inv}_p(V_p)=1$ (resp.~is superspecial in the sense of Drinfeld if ${\rm inv}_p(V_p)=-1$). 
\end{remark}
By Proposition \ref{structurestar},  $\M({\bf t}; 1, 1)^*$ and $\M({\bf t}'; 1, 1)^*$ are regular;  
then 
$\M({\bf t}; 1, 1)^{*, p}\lra \M({\bf t}; 1, 1)^*$ is an \'etale  Galois covering  by EGA IV, 18.10.16;  furthermore, $\varphi$ is the 
composition of a blow-up morphism and a finite morphism. 
\end{proof}
More generally, for given ${\bf t}$, let $S\subset \{ p\mid p\neq 2, p\mid \Delta, t(p)=0\}$, and define ${\bf t}_S'$ by $t_S'(p)=t(p)$ for $p\notin S$ and
$t_S'(p)=2$ for $p\in S$. Then there exists an \'etale  Galois covering $\M({\bf t}; 1, 1)^{*, S}$ of $\M({\bf t}; 1, 1)^{*}$ with Galois group $(\Z/2\Z)^{S}$ and
a morphism 
\begin{equation*}
\varphi_S:\, \M({\bf t}; 1, 1)^{*, S}\lra \M({\bf t}_S'; 1, 1)^* ,
\end{equation*}
which is finite over $\Spec \Z\setminus S$, and contracts lines in the fibers of $\M({\bf t}; 1, 1)^{*, S}$ over primes $p\in S$.

\bigskip
\obeylines

Department of Mathematics
University of Toronto
40 St. George St., BA6290
Toronto, ON M5S 2E4
Canada

\bigskip

\bigskip
\obeylines
Mathematisches Institut der Universit\"at Bonn  
Endenicher Allee 60 
53115 Bonn, Germany.
email: rapoport@math.uni-bonn.de

\end{document}